\documentclass{article}
\usepackage{amssymb}
\usepackage{amsmath}
\usepackage{amsfonts}
\usepackage{color}
\newtheorem{theorem}{Theorem}

\newtheorem{corollary}[theorem]{Corollary}

\newtheorem{definition}[theorem]{Definition}
\newtheorem{example}[theorem]{Example}

\newtheorem{lemma}[theorem]{Lemma}

\newtheorem{proposition}[theorem]{Proposition}
\newtheorem{remark}[theorem]{Remark}

\newenvironment{proof}[1][Proof]{\noindent\textbf{#1.} }{\ \rule{0.5em}{0.5em}}
\begin{document}

\title{On Partial Smoothness, Tilt Stability and the $\mathcal{VU}$--Decomposition}
\author{A. C. Eberhard\thanks{Email addresses of the authors: andy.eb@rmit.edu.au; yluo@rmit.edu.au and
shuai0liu@gmail.com\newline
This research was in part supported by the ARC Discovery grant no.
DP120100567.}, Y. Luo and S. Liu}
\date{}
\maketitle

\begin{abstract}
Under the assumption of prox-regularity and the presence of a tilt stable
local minimum we are able to show that a $\mathcal{VU}$ like decomposition
gives rise to the existence of a smooth manifold on which the function in
question coincides locally with a smooth function.
\end{abstract}

\section{Introduction}

The study of substructure of nonsmooth functions has led to an enrichment of
fundamental theory of nonsmooth functions \cite{Hare:1, Hare:2, Hare:3,
Lem:1, Lewis:1,Lewis:2, Miller:1}. Fundamental to this substructure is the
presence of manifolds along which the restriction of the nonsmooth function
exhibits some kind of smoothness. In the case of \textquotedblleft partially smooth function\textquotedblright\ \cite{Lewis:1}
an axiomatic approach is used to describe the local structure that is
observed in a number of important examples \cite{Lewis:1, Lewis:2}. In \cite{Lewis:2}
it is shown that the study of tilt stability can be enhanced for the class of
partially smooth functions. In the
theory of the \textquotedblleft$\mathcal{U}$-Lagrangian\textquotedblright\ and the associated \textquotedblleft$\mathcal{VU}$
decomposition\textquotedblright\ \cite{Lem:1, Mifflin:2003} the existence of a smooth manifold substructure is proven for
some special classes of functions \cite{Mifflin:2003, Mifflin:2005}. In the extended theory the presence of so
called \textquotedblleft fast tracks\textquotedblright\ is assumed and these
also give rise to similar manifold substructures \cite{Miller:1, Mifflin:2004}. The $\mathcal{U}$
-Lagrangian is reminiscent of a partial form of \textquotedblleft tilt minimisation\textquotedblright\
\cite{rock:7} and this
observation has motivated this study. As fast tracks and related concepts such as 
 \textquotedblleft identifiable constraints\textquotedblright are designed to aid the
design of methods for the solution of nonsmooth minimization
problems \cite{Wright:1993, Miller:1, Mifflin:2002, Mifflin:2004,
Mifflin:2005,Hare:2014}, it seems appropriate to ask what additional
structure does the existence of a tilt stable local minimum give to the
study of the $\mathcal{VU}$ decomposition \cite{Lem:1}? This is the subject
of the paper.
In the following discussion we denote the extended reals by $\mathbb{R}_{\infty }:=\mathbb{R}\cup \left\{
+\infty \right\} .$ If not otherwise stated we will consider a lower
semi-continuous, extended--real--valued function $f:\mathbb{R}^{n}\rightarrow \mathbb{R}_{\infty }$. 
We denote the limiting subdifferential of
Mordukhovich, Ioffe and Kruger by $\partial f$.

Tilt stability was first studied in \cite{rock:7} for the case of $f$ being
both \textquotedblleft prox-regular\textquotedblright\  at $\bar{x}\ $ for $\bar{z}\in \partial f(\bar{x})$ and 
\textquotedblleft subdifferentially continuous\textquotedblright\  at
$\left( \bar{x},\bar{z}\right)$, in
the sense of Rockafellar and Poliquin \cite{polrock:1}. In \cite{rock:7} a characterisation
of tilt stability is made in terms of certain second order sufficient
optimality conditions. Such optimality conditions have been studied in 
\cite{rock:7, Boris:4, eberhard:8, eberhard:9}. In \cite{eberhard:6, eberhard:8}
it is shown that second order information provided by the coderivative is
closely related to another second order condition framed in terms of the
\textquotedblleft limiting subhessian\textquotedblright\   \cite{ebpenot:2, eberhard:1, eberhard:6}.
These may be thought of as the robust$\backslash$limiting version of symmetric matrices associated 
with a lower, supporting Taylor expansion with a first order component ${z}$ and
second order component  $Q$ (a symmetric matrix). The limiting pairs
$(\bar{z}, \bar{Q})$ are contained in the so called  \textquotedblleft subjet\textquotedblright\  \cite{Crandall:1992} and the second
order components $\bar{Q}$ associated with a given $\bar{z} \in \partial f(\bar{x})$  are contained 
in the limiting subhessian $\underline{\partial}^2 f(\bar{x}, \bar{z})$,  \cite{ebpenot:2, ebioffe:4,  eberhard:9}.
These have been extensively studied and possess a robust calculus similar to that which
exists for the limiting subdifferential \cite{ebioffe:4, eberhard:7}.  One can view
the \textquotedblleft best curvature\textquotedblright\   approximation in the direction $h$ for the function $f$ at $(\bar{x},\bar{z})$ to be
$q(\underline{\partial}^2 f(\bar{x},\bar{z}))(h):= \sup \{\langle Qh, h\rangle \mid Q \in \underline{\partial}^2 f(\bar{x},\bar{z}) \}$,
where we denote by $\langle u, h\rangle$ the usual Euclidean inner product of two vectors
$u, h \in \mathbb{R}^n$.

To complete our
discussion we consider the  $\mathcal{VU}$ decomposition \cite{Lem:1}.  When
 $\operatorname{rel}$-$\operatorname{int} \partial f\left( \bar{x}\right) \neq \emptyset $ we can take $\bar{z}\in \operatorname{rel}$-$\operatorname{int}
 \partial f\left( \bar{x}\right) $ and define $\mathcal{V}:=\operatorname{	span}\left\{ \partial f\left( \bar{x}\right) -\bar{z}\right\} $ and
$\mathcal{U}:=\mathcal{V}^{\perp }$.
The $\mathcal{V}$-space is thought to
capture the directions of nonsmoothness of $f$ at $\bar{x}$ while the $\mathcal{U}$ is thought
to capture directions of smoothness. When $\mathcal{U}^{2}:= \operatorname{dom} 
q(\underline{\partial }^{2}f(\bar{x},\bar{z}))(\cdot )$ is a linear subspace that is
contained in $\mathcal{U}$, we call $\mathcal{U}^{2}$ the second order component of 
$\mathcal{U}$ and  in Lemma \ref{lem:sharp}  we give quite mild condition under
which this is indeed the case.  When $\mathcal{U}^{2}=\mathcal{U}$ we say that a fast-track 
exists at $\bar{x}$ for $\bar{z}\in \partial f\left( \bar{x}\right) $.

 In this paper we
investigate whether the existence of a tilt stable local minimum provides
extra information regarding the existence of a smooth manifold within which a
smooth function interpolates the values of the $f$. We are able to show the
following positive results.
Recall that we say $f$ is quadratically minorised when there exists a quadratic function
$q\left( x\right) :=\alpha -\frac{R}{2}\left\Vert x-\bar{x}\right\Vert ^{2}$
such that $q\leq f$ (globally). All balls $B^{X}_{\varepsilon} (0) := \{ x \in X \mid \|x\| \leq \varepsilon\}$ are closed. 

\begin{theorem}
\label{thm:1}Consider $f:\mathbb{R}^{n}\rightarrow \mathbb{R}_{\infty }$ is
a proper lower semi-continuous function, quadratically minorised, and
prox-regular at $\bar{x}$ for $0\in \partial f(\bar{x})$. Suppose in
addition $f$ admits a nontrivial subspace $\mathcal{U}^{2}:= \operatorname{dom} \left(
\underline{\partial }^{2}f\left( \bar{x},0\right) \right) (\cdot)$ and that $f$ 
has a tilt stable local minimum at $\bar{x}$. Then $\mathcal{U}^2 \subseteq \mathcal{U}$ and for $g\left( w\right) :=
\left[ \operatorname{co}h\right] \left( w\right) $, $h(w):=f(\bar{x}+w)$ and $\{v\left( u\right) \}=
\operatorname{argmin}_{v^{\prime }\in \mathcal{V}^{2}\cap B_{\varepsilon }\left( 0\right) }f\left( \bar{x}+u+v^{\prime
}\right) :\mathcal{U}^{2}\rightarrow \mathcal{V}^{2}:=(\mathcal{U}^2)^{\perp}$, 
there exists a $\delta>0$ such that 
we have $g\left(
u+v\left( u\right) \right) =f\left( \bar{x}+u+v\left( u\right) \right) $ and
$\nabla _{u}g\left( u+v\left( u\right) \right) $ existing as Lipschitz
function for $u\in B_{\delta }^{\mathcal{U}^{2}}\left( 0\right) $. 
\end{theorem}
That is,
$\mathcal{M}:=\left\{ \left( u,v\left( u\right) \right) \mid u\in
B_{\varepsilon }^{\mathcal{U}^{2}}\left( 0\right) \right\} $ is a manifold
on which the restriction to $\mathcal{M}$ of function $g$ coincides with
a smooth $C^{1,1}$ function of $u \in \mathcal{U}$ (tilt stability ensures local uniqueness of the function $v(\cdot)$).
Assuming a little more we obtain the smoothness of $v$ and in addition the
smoothness of the manifold.

\begin{theorem}
\label{thm:2}Consider $f:\mathbb{R}^{n}\rightarrow \mathbb{R}_{\infty }$ is
a proper lower semi-continuous function, quadratically minorised and
prox-regular at $\bar{x}$ for $0\in \partial f(\bar{x})$. Suppose in
addition that $\mathcal{U}^{2} =\mathcal{U}$ is a linear subspace (i.e. $\mathcal{U}$ admits a fast track),  $f$
has a tilt stable local minimum at $\bar{x}$ for $0\in \operatorname{rel}$-$\operatorname{int}\partial f\left( \bar{x}\right) $ and  $\partial ^{\infty }f\left( \bar{x}+u+v\left( u\right) \right)
=\left\{ 0\right\} $ for  $  v\left( u\right)  \in  
\operatorname{argmin}_{v^{\prime }\in \mathcal{V}\cap B_{\varepsilon }\left(
	0\right) }\left\{ g\left( u+v^{\prime }\right) \right\} : \mathcal{U}\rightarrow \mathcal{V}$, 
$u \in B_{\varepsilon }^{\mathcal{U}}\left( 0\right)$. Then 
there exists a $\varepsilon >0$ such that for $g\left( w\right) :=\left[
\operatorname{co}h\right] \left( w\right) $  the function
defined below is a $C^{1,1}\left( B_{\varepsilon }^{\mathcal{U}}\left( 0\right)
\right) $  smooth function
\begin{eqnarray*}
u\mapsto g\left( u+v\left( u\right) \right) &=&f\left( \bar{x}+u+v\left(
u\right) \right) \quad \text{where } \\
\nabla _{w}g\left( u+v\left( u\right) \right) &=&\left( e_{\mathcal{U}},\nabla v\left( u\right) \right) ^{T}\partial g\left( u+v\left( u\right)
\right)
\end{eqnarray*}
($e_{\mathcal{U}}$ is the identity operator on $\mathcal{U}$). Moreover
if we suppose we have a $\delta >0$ (with $\delta \leq \varepsilon$) such that for all $z_{\mathcal{V}}\in
B_{\delta }\left( 0\right) \cap \mathcal{V}\subseteq \partial _{\mathcal{V}}f\left( \bar{x}\right) $ we have a common
\begin{equation}
\{v\left( u\right) \} = \operatorname{argmin}_{v\in \mathcal{V}\cap B_{\varepsilon
}\left( 0\right) }\left\{ f\left( \bar{x}+u+v\right) -\langle z_{\mathcal{V}},v\rangle \right\}  \label{eqn:1}
\end{equation}
for all $u\in B_{\varepsilon }\left( 0\right) \cap \mathcal{U}$. Then $\mathcal{M}:=\left\{ \left( u,v\left( u\right) \right) \mid u\in
B_{\varepsilon }^{\mathcal{U}}\left( 0\right) \right\} $ is a $C^1$ - smooth
manifold on which $u\mapsto f\left( \bar{x}+u+v\left( u\right) \right) $ is $C^{1,1}
\left( B_{\delta }^{\mathcal{U}}\left( 0\right) \right) $ smooth and $u\mapsto v\left( u\right) $ is continuously differentiable.
\end{theorem}
\smallskip

We are also able to produce a lower Taylor approximation for $f$ that holds locally at all points 
inside $\mathcal{M}$, see Corollary \ref{cor:53}. These results differ from those present in the literature in that we impose common structural assumptions on $f$ found elsewhere in the literature on stability of local minima \cite{Drusvy:1,rock:7}, rather than imposing very special structural properties,  as is the approach of \cite{Hare:2014,Wright:1993,Mifflin:2003,Mifflin:2004}. Moreover, we do not assume the a-priori existence of any kind of smoothness of the underlying manifold, as is done in the axiomatic approach in \cite{Lewis:2}, but let smoothness arise from a graded set of assumptions which progressively enforce greater smoothness. In this way the roles of these respective assumptions are  clarified. Finally we note that it is natural in this context to study $C^{1,1}$ smoothness rather than the $C^2$ smoothness used in other works such as \cite{Lewis:2,Mifflin:2002,Miller:1}.

\section{Preliminaries}

The following basic concepts are used repeatedly throughout the paper.

\begin{definition}
	\label{lim:subhessian} Suppose $f:\mathbb{R}^{n}\rightarrow \mathbb{R}_{\infty }$ is a lower semi--continuous function.
	
	\begin{enumerate}
		\item Denote by $\partial _{p}f(\bar{x})$ the proximal subdifferential,
		which consists of all vectors $z$ satisfying $f(x)\geq f(\bar{x})+\langle
		z,x-\bar{x}\rangle -\frac{r}{2}\Vert x-\bar{x}\Vert ^{2}$ in some
		neighbourhood of $\bar{x}$, for some $r\geq 0$, where $\Vert \cdot \Vert $
		denotes the Euclidean norm. Denote by $S_{p}(f)$ the points in the domain of
		$f$ at which $\partial _{p}f(x)\neq \emptyset $.
		
		\item The limiting subdifferential \cite{M06a, rock:6} at $x$ is given by
		\begin{equation*}
		\partial f(x)=\limsup_{x^{\prime }\rightarrow _{f}x}\partial _{p}f(x^{\prime
		}):=\{z\mid \exists z_{v}\in \partial _{p}f(x_{v}),x_{v}\rightarrow _{f}x \text{, with }z_{v}\rightarrow z\},
		\end{equation*}
		where $x^{\prime }\rightarrow _{f}x$ means that $x^{\prime }\rightarrow x$
		and $f(x^{\prime })\rightarrow f(x)$.
\item  The singular
limiting subdifferential is given by
\begin{align*}
\partial ^{\infty }f(x)& =\limsup_{x^{\prime }\rightarrow _{f}x}\!{}^{\infty
}\,\partial _{p}f(x^{\prime }) \\
& :=\{z\mid \exists z_{v}\in \partial _{p}f(x_{v}),x_{v}\rightarrow _{f}x \text{, with }\lambda _{v}\downarrow 0\text{ and }\lambda
_{v}z_{v}\rightarrow z\}.
\end{align*}	

	\end{enumerate}
\end{definition}

\subsection{The $\mathcal{VU}$ decomposition}\label{sec:VU}

Denote the convex hull of a set $C\subseteq \mathbb{R}^{n}$ by $\operatorname{co}C$.
The convex hull of a function $f:\mathbb{R}^{n}\rightarrow \mathbb{R}_{\infty }$
is denoted by $\operatorname{co}f$ and corresponds to the proper
lower-semi-continuous function whose epigraph is given by
$\overline{\operatorname{co	}\operatorname{epi}f}$. In this section we will use a slightly
weaker notion of the $\mathcal{VU}$ decomposition. When  $\operatorname{rel}$-$\operatorname{int}
\operatorname{co}\partial f\left( \bar{x}\right) \neq \emptyset $ we can take $\bar{z}\in \operatorname{rel}$-$\operatorname{int}
\operatorname{co}\partial f\left( \bar{x}\right) $ and define $\mathcal{V}:=\operatorname{	span}
\left\{ \operatorname{co}\partial f\left( \bar{x}\right) -\bar{z}\right\} $ and
$\mathcal{U}:=\mathcal{V}^{\perp }$.

Under the $\mathcal{VU}$ decomposition \cite{Lem:1} for a given $\bar{z}\in
\operatorname{rel}$-$\operatorname{int}\operatorname{co}\partial f(\bar{x})$ we have, by definition,
\begin{equation}
\bar{z}+B_{\varepsilon }\left( 0\right) \cap \mathcal{V}\subseteq \operatorname{co}\partial f\left( \bar{x}\right) 
\quad \text{for some $\varepsilon >0$.}
\label{neqn:6}
\end{equation}
One can then decompose $\bar{z}=\bar{z}_{\mathcal{U}}+\bar{z}_{\mathcal{V}}$
so that when $w=u+v\in \mathcal{U}\oplus \mathcal{V}$ we have $\langle \bar{z},w\rangle =\langle \bar{z}_{\mathcal{U}},
u\rangle +\langle \bar{z}_{\mathcal{V}},v\rangle .$ Indeed we may decompose into the direct
 sum $x=x_{\mathcal{U}}+x_{\mathcal{V}}\in \mathcal{U}\oplus \mathcal{V}$ and use the
following norm for this decomposition $\left\Vert x-\bar{x}\right\Vert
^{2}:=\left\Vert x_{\mathcal{U}}-\bar{x}_{\mathcal{U}}\right\Vert
^{2}+\left\Vert x_{\mathcal{V}}-\bar{x}_{\mathcal{V}}\right\Vert ^{2}.$ As all norms are equivalent
we will at times prefer to use  $\{B^{\mathcal{U}}_{\varepsilon} (\bar{x}_{\mathcal{U}})
\oplus B^{\mathcal{V}}_{\varepsilon} (\bar{x}_{\mathcal{V}}) \}_{\varepsilon > 0}$ which more directly 
reflects the direct sum  $\mathcal{U}\oplus \mathcal{V}$,  where each $B^{(\cdot)}_{\varepsilon} (\cdot)$ is a closed ball of radius $\varepsilon >0$, in their respective space.

Denote the projection onto the subspaces $\mathcal{U}$ and $\mathcal{V}$ by $P_{\mathcal{U}}\left( \cdot \right) $ and $P_{\mathcal{V}}\left( \cdot
\right) $, respectively. Denote by $f|_{\mathcal{U}}$ the restriction of $f$
to the subspace $\mathcal{U}$, $\partial _{\mathcal{V}}f(\bar{x}):=P_{\mathcal{V}}(\partial f(\bar{x}))$ and $\partial _{\mathcal{U}}f(\bar{x}
):=P_{\mathcal{U}}(\partial f(\bar{x}))$.  Let $\delta _{C}(x)$ denote the indicator function of a
set $C$, $\delta _{C}(x)=0$ iff $x\in C$ and $+\infty $ otherwise. Let $f^{\ast }$ denote the convex conjugate of a function $f$.

\begin{remark}
The condition (\ref{neqn:6}) implies one can take $\mathcal{V}:=\operatorname{span}
\left\{ \operatorname{co}\partial f\left( \bar{x}\right) -\bar{z}\right\} =\operatorname{affine}$-$\operatorname{hull}
\left[ \operatorname{co}\partial f\left( \bar{x}\right) \right]
-\bar{z}$ which is independent of the choice of $\bar{z}\in \operatorname{co}\partial f\left( \bar{x}\right) $. 
Moreover, as was observed in \cite[Lemma
2.4]{Miller:1} we have $\bar{z}_{\mathcal{U}}=P_{\operatorname{affine}\text{-}\operatorname{hull}
	\operatorname{co}\partial f\left( \bar{x}\right) }\left( 0\right) $ (see part \ref{part:3} below).
\end{remark}

\begin{proposition}
\label{prop:reg}Suppose $f:\mathbb{R}^{n}\rightarrow \mathbb{R}_{\infty }$
is a proper lower semi--continuous function with (\ref{neqn:6}) holding.

\begin{enumerate}
\item We have
\begin{equation}
\mathcal{U}=\left\{ u\mid -\delta _{\partial f(\bar{x})}^{\ast }(-u)=\delta
_{\partial f(\bar{x})}^{\ast }(u)\right\} .  \label{neqn:4}
\end{equation}

\item \label{part:3} We have
\begin{equation}
\partial f\left( \bar{x}\right) =\left\{ \bar{z}_{\mathcal{U}}\right\}
\oplus \partial _{\mathcal{V}}f\left( \bar{x}\right) .  \label{prop:6:3}
\end{equation}

\item \label{part:4} Suppose there exists  $\varepsilon >0$
such that for all $z_{\mathcal{V}}\in B_{\varepsilon }\left( \bar{z}_{\mathcal{V}}\right) \cap \mathcal{V}\subseteq \partial _{\mathcal{V}}f\left(
\bar{x}\right) $ there is a common
\begin{equation*}
v\left( u\right) \in \operatorname{argmin}_{v\in \mathcal{V}\cap B_{\varepsilon
}\left( 0\right) }\left\{ f\left( \bar{x}+u+v\right) -\langle z_{\mathcal{V}},v\rangle \right\} \cap 
\operatorname{int} B_{\varepsilon} (0)
\end{equation*}
for all $u\in B_{\varepsilon }\left( 0\right) \cap \mathcal{U}$. Then we
have
\begin{equation}
\operatorname{cone}\left[ \partial _{\mathcal{V}}f\left( \bar{x}+u+v\left( u\right)
\right) -\bar{z}_{\mathcal{V}}\right] \supseteq \mathcal{V}.  \label{neqn:26}
\end{equation}

\item \label{part:2} If we impose the addition assumption that $f$ is (Clarke) regular at $\bar{x}$, 
$\bar{z}\in \partial f\left( \bar{x}\right) $  and $\partial ^{\infty }f\left(
\bar{x}\right) \cap \mathcal{V} =\left\{ 0\right\} $. Then the function
\begin{equation*}
	H_{\mathcal{U}}\left( \cdot \right) :=f\left( \bar{x}+\cdot \right) :\mathcal{U}
	\rightarrow \mathbb{R}_{\infty }
\end{equation*}
is strictly differentiable at $0$ and single valued with $\partial H_{\mathcal{U}}\left(
0\right) =\left\{ \bar{z}_{\mathcal{U}}\right\}$ 
and $H_{\mathcal{U}}$ (as a function defined on $\mathcal{U}$) is continuous with $H_{\mathcal{U}}$
and $-H_{\mathcal{U}}$ (Clarke) regular
functions at $0$ (in the sense of \cite{rock:6}).
\end{enumerate}
\end{proposition}

\begin{proof}
(1) If $u\in \mathcal{U}$ then by construction we have
\begin{equation}
-\delta _{\partial f(\bar{x})}^{\ast }(-u)=-\delta _{\operatorname{co}\partial f(\bar{x})}^{\ast }(-u)=\delta _{\operatorname{co}\partial f(\bar{x})}^{\ast
}(u)=\delta _{\partial f(\bar{x})}^{\ast }(u)  \label{neqn:5}
\end{equation}
giving the containment of $\mathcal{U}$ in the right hand side of (\ref{neqn:4}). For $u$ satisfying (\ref{neqn:5}) 
then $\langle z-\bar{z},u\rangle =0$ for all $z\in \operatorname{co}\partial f(\bar{x})$. That is, $u\perp
\lbrack \operatorname{co}\partial f(\bar{x})-\bar{z}]$ and hence $u\perp \mathcal{V}=\mathcal{U}^{\perp }$ verifying $u\in \mathcal{U}$.

(2) Since $\partial f(\bar{x})\subseteq \bar{z}+\mathcal{V}=\bar{z}_{\mathcal{U}}+\mathcal{V}$ always have  $\partial f\left( \bar{x}\right) =\left\{
\bar{z}_{\mathcal{U}}\right\} \oplus \partial _{\mathcal{V}}f\left( \bar{x}\right) .$

(3) When $v\left( u\right) \in \operatorname{argmin}_{v\in \mathcal{V}\cap
B_{\varepsilon }\left( 0\right) }\left\{ f\left( \bar{x}+u+v\right) -\langle
z_{\mathcal{V}},v\rangle \right\} $ for all $u\in B_{\varepsilon }\left(
0\right) \cap \mathcal{U}$ and $z_{\mathcal{V}}\in B_{\varepsilon }\left(
\bar{z}_{\mathcal{V}}\right) \cap \mathcal{V}$ we have, due to the necessary
optimality conditions, that
\begin{equation*}
z_{\mathcal{V}}\in \partial _{\mathcal{V}}f\left( \bar{x}+u+v\left( u\right)
\right)
\end{equation*}
and hence $B_{\varepsilon }\left( \bar{z}_{\mathcal{V}}\right) \cap \mathcal{V}\subseteq \partial _{\mathcal{V}}f\left( \bar{x}+u+v\left( u\right)
\right) $ giving (\ref{neqn:26}).

(4) For $h\left( \cdot \right) :=f\left( \bar{x}+\cdot \right) $ define $H =h
+\delta _{\mathcal{U}} $ so $h\left( u\right) =H\left( u\right) $ when $u\in
\mathcal{U}$. Then as $\partial ^{\infty }f\left( \bar{x}\right) \cap \mathcal{V} =\left\{
0\right\} $, by \cite[Corollary 10.9]{rock:6} we have
\begin{equation*}
\partial H\left( 0\right) \subseteq \partial f\left( \bar{x}\right) + N_{\mathcal{U}}\left( 0 \right) =\partial f\left( \bar{x}\right) +\mathcal{\ V}.
\end{equation*}
Then restricting to $\mathcal{U}$ we have $P_{\mathcal{U}} \partial H\left( 0\right)
\subseteq \partial_{\mathcal{U}}
f\left( \bar{x} \right)$.  Then for $u \in \mathcal{U}$ we have  $\delta_{\partial H(0)}^{\ast} (u) = \delta_{P_{\mathcal{U}}\partial H(0)}^{\ast} (u) \leq  \delta _{\partial f(\bar{x})}^{\ast }(u)$ and so 
\begin{equation*}
-\delta _{\partial f(\bar{x})}^{\ast }(-u)\leq -\hat{d}H(0)(-u)\leq \hat{d}
H(0)(u)\leq \delta _{\partial f(\bar{x})}^{\ast }(u). 
\end{equation*}
As $f$ is regular at $\bar{x}$ we have $\partial^{\infty} f (\bar{x}) = 0^{+} (\partial f (\bar{x}))$ where  the later corresponds to the recession directions of the convex set 
$\partial f (\bar{x})$ (see \cite[Theorem 8.49]{rock:6}).  Then we have $0^+ (\partial f (\bar{x}))  \subseteq \mathcal{V}$. [Take $u \in 0^+ (\partial f (\bar{x}))$ and $z \in \operatorname{rel-int}
\partial f (\bar{x})$. Then by \cite[Theorem 6.1]{rock:1} we have $z + u \in  \operatorname{rel-int}
\partial f (\bar{x})$ and hence $u \in \mathcal{V}$.] Thus for $u \in \mathcal{U} \subseteq (0^+ (\partial f (\bar{x})))^{\circ}$  we have 
$$
\hat{d}H (0)(u) := \limsup_{x \to 0, t \downarrow 0} \inf_{u^{\prime}\to
	u} \frac{1}{t} (f(x+tu^{\prime})-f(x) ) =\delta_{\partial H(0)}^{\ast} (u)=\delta_{P_{\mathcal{U}}\partial H(0)}^{\ast} (u),
$$
see \cite[Definition 8.16, Exercise 8.23]{rock:6}.
It follows that  $-\hat{d}H(0)(-u)=\hat{d}H(0)(u)$ for all $u\in \mathcal{U}$.
Restriction of $H$ to the subspace $\mathcal{U}$, (denoted this function by $H_{\mathcal{U}}$) 
we have   $\partial^{\infty} H_{\mathcal{U}} (0)  \subseteq \partial ^{\infty }f\left( \bar{x}\right)
\cap \mathcal{U}   = \{0\}$ then by
\cite[Theorem 9.18]{rock:6} we have  
$\partial H_{\mathcal{U}} \left( 0\right) $ a singleton
with $H_{\mathcal{U}}$ continuous at $0$ and $H_{\mathcal{U}}$ and $-H_{\mathcal{U}}$
(Clarke) regular. As $\bar{z}_{\mathcal{U}
}\in \partial H_{\mathcal{U}}\left( 0\right) $ we have $\partial H_{\mathcal{U}}\left( 0\right)
=\left\{
\bar{z}_{\mathcal{U}}\right\} $, so $\partial_{\mathcal{U}} f (\bar{x})
=\left\{ \bar{z}_{\mathcal{U}}\right\} $.
\end{proof}

\section{A Primer on Subjets and Subhessians \label{sec:3}}

We will have need to discuss second order behaviour in this paper and as a
consequence it will be useful to define a refinement of this decomposition
that takes into account such second order variations. In most treatments of
the $\mathcal{VU}$ decomposition one finds that by restricting $f$ to
$\mathcal{M}:=\{(u,v(u)) \mid u \in \mathcal{U}\}$ not only do we find
$f$ is smooth we also find that there is better second order behaviour as well \cite{Lem:1}.
This is also
often associated with smooth manifold substructures. Let $\mathcal{S}(n)$ denote the set of symmetric $n\times n$ matrices (endowed with the
Frobenius norm and inner product) for which $\langle Q,hh^{T}\rangle
=h^{T}Qh $. Denote the cone of positive semi-definite matrices by $\mathcal{P}(n)$ and
  $\Delta_{2}f(x,t,z,u):=2\frac{f(x+tu)-f(x)-t\langle z,u\rangle }{t^{2}}$.

\begin{definition} \label{def:6}
Suppose $f:\mathbb{R}^{n}\rightarrow \mathbb{R}_{\infty }$ is a lower
semi--continuous function.

\begin{enumerate}
		\item \label{part:1}The function $f$ is said to be twice sub-differentiable
	(or possess a subjet) at $x$ if the following set is nonempty;
	\begin{equation*}
		\partial ^{2,-}f(x)=\{(\nabla \varphi (x),\nabla ^{2}\varphi
		(x))\,:\,f-\varphi \,\text{ has a local minimum at }\,\,x\text{ with }
		\,\,\varphi \in {\mathcal{C}}^{2}{\mathcal{(}}\mathbb{R}^{n})\}.
	\end{equation*}
	The subhessians at $(x,z)\in \operatorname{graph}\partial f$ are given by $\partial
	^{2,-}f(x,z):=\{Q\in \mathcal{S}(n)\mid (z,Q)\in \partial ^{2,-}f(x)\}$.
	
	\item The limiting subjet of $f$ at $x$ is defined to be: $\underline{\partial }^{2}f(x)=\limsup_{u\rightarrow ^{f}x}\partial ^{2,-}f(u)$ and the
	associated limiting subhessians for $z\in \partial f\left( x\right) $ are $	\underline{\partial }^{2}f(x,z)=\left\{ Q\in \mathcal{S}\left( n\right) \mid
	\left( z,Q\right) \in \underline{\partial }^{2}f(x)\right\} $.
	
	\item We define the rank one barrier cone for $\underline{\partial }^{2}f(x,z)$ as
\begin{equation*}
		b^{1}(\underline{\partial }^{2}f(x,z)):=\{h\in \mathbb{R}^{n}\mid q\left(
		\underline{\partial }^{2}f(x,z)\right) (h):=\sup \left\{ \langle Qh,h\rangle
		\mid Q\in \underline{\partial }^{2}f(x,z)\right\} <\infty \}.	
\end{equation*}
\item  Denoting $S_{2}(f)=\{x\in \operatorname{dom}\,(f)\mid \nabla ^{2}f(x) \text{
exists} \} $, then the limiting Hessians at $(\bar{x}, \bar{z})$ are given
by:
\begin{eqnarray*}
\overline{D}^{2}f(\bar{x},\bar{z}) &=&\{Q\in \mathcal{S}(n)\mid
Q=\lim_{n\rightarrow \infty }\nabla ^{2}f(x_{n}) \\
&&\qquad \text{where }\{x_{n}\}\subseteq S_{2}(f)\text{, }x_{n}\rightarrow
^{f}\bar{x}\text{ and }\nabla f(x_{n})\rightarrow \bar{z}\}.
\end{eqnarray*}

\item Define the second order Dini-directional derivative of $f$ by $f_{\_}^{\prime \prime }(\bar{x},z,h)=\liminf_{t\downarrow 0,u\rightarrow
h}\Delta _{2}f(\bar{x},t,z,u)$.
\end{enumerate}
\end{definition}

Define $\partial ^{2,+}f(x,z):= -\partial ^{2,-}(-f)(x,-z)$ then when $Q\in
\partial ^{2,-}f(x,z)\cap \partial ^{2,+}f(x,z)$ it follows that $Q=\nabla
^{2}f\left( x\right) $ and $z=\nabla f\left( x\right) $. If $f_{\_}^{\prime
\prime }(\bar{x},z,h)$ is finite then $f_{\_}^{\prime }(\bar{x},h):=\liminf_{{{t\downarrow 0}} \atop {{u\rightarrow h }}}\frac{1}{t}(f(\bar{x}+tu)-f(\bar{x}))=\langle
z,h\rangle $. It must be stressed that these second order objects may not
exist everywhere but as $\partial ^{2,-}f(x)$ is non--empty on a dense
subset of its domain \cite{Crandall:1992} when $f$ is lower semi--continuous then at worst so are
the limiting objects. In finite dimensions this concept is closely related
to the proximal subdifferential  (as we discuss below). The subhessian is always a closed convex
set of matrices while $\underline{\partial }^{2}f(\bar{x},z)$ may not be
convex (just as $\partial _{p}f(\bar{x})$ is convex while $\partial f(\bar{x})$ often is not).

A function $f$ is \emph{para-concave} around $\bar{x}$ when there exists a $c>0$ and a ball $B_{\varepsilon }\left( \bar{x}\right) $ within which the
function $x\mapsto $ $f\left( x\right) -\frac{c}{2}\left\Vert x\right\Vert
^{2}$ is finite concave (conversely $f$ is para-convex around $\bar{x}$ iff $-f$ is para-concave around $\bar{x}$). If a function is para--concave or para--convex we have (by Alexandrov's
theorem) the set $S_{2}(f)$ is of full Lebesgue measure in $\operatorname{dom}\,f$.  A function is $C^{1,1}$  when $\nabla f $ exists and satisfies a Lipschitz property. In \cite[Lemma 2.1]{eberhard:6}, it is noted that $f$ is locally $C^{1,1}$ iff $f$ is
simultaneously a locally para-convex and para-concave function. The next
observation was first made in \cite[Prposition 4.2]{ebpenot:2} and later
used in \cite[Proposition 6.1]{ebioffe:4}.

\begin{proposition}[\protect\cite{ebioffe:4}, Proposition 6.1]
\label{prop:ebpenot}If $f$ is lower semi--continuous then for $z\in \partial
f(\bar{x})$ we have
\begin{equation}
\overline{D}^{2}f(\bar{x},z)-\mathcal{P}(n)\subseteq \underline{\partial }^{2}f(\bar{x},z).  \label{ebneqn:8}
\end{equation}
If we assume in addition that $f$ is continuous and a para--concave function
around $\bar{x}$ then equality holds in (\ref{ebneqn:8}).
\end{proposition}

A weakened form of para-convexity is prox-regularity.

\begin{definition} [\protect\cite{polrock:1}]
\label{def:proxreg}Let the function $f:\mathbb{R}^{n}\rightarrow \mathbb{R}_{\infty }$ be finite at $\bar{x}$.

\begin{enumerate}
\item The function $f$ is prox--regular at $\bar{x}$ for $\bar{z}$ with
respect to $\varepsilon >0$ and $r \geq0 $, where $\bar{z}\in \partial f(\bar{x})$, if $f $ is locally lower semi--continuous at $\bar{x}$ and
\begin{equation*}
f(x^{\prime })\geq f(x)+\langle z,x^{\prime }-x\rangle -\frac{r}{2}\Vert
x^{\prime }-x\Vert ^{2}
\end{equation*}
whenever $\Vert x^{\prime }-\bar{x}\Vert \leq \varepsilon $ and $\Vert x-\bar{x}\Vert \leq \varepsilon $ and $\left\vert f(x)-f(\bar{x})\right\vert
\leq \varepsilon \ $ with $\Vert z-\bar{z}\Vert \leq \varepsilon $ and $z\in
\partial f(x)$.

\item The function $f$ is subdifferentially continuous at $\bar{x}$ for $\bar{z}$, where $\bar{z}\in \partial f(\bar{x})$, if for every $\delta >0$
there exists $\varepsilon >0$ such that $\left\vert f(x)-f(\bar{x}
)\right\vert \leq \delta $ whenever $|x-\bar{x}|\leq \varepsilon $ and $|z-
\bar{z}|\leq \varepsilon $ with $z\in \partial f(x).$
\end{enumerate}
\end{definition}

\begin{remark}
 In this paper we adopt the convention that limiting subgradients must exist at $\bar{x}$ to invoke this
definition.
We say that $f$ is prox-regular at $\bar{x}$ iff it is
prox-regular with respect to each $\bar{z}\in \partial f\left( \bar{x}\right) $ (with respect to some $\varepsilon >0$ and $r\geq 0$).
\end{remark}
\begin{remark}\label{rem:jets}
We shall now discuss a well known alternative characterisation of $(z,Q)\in
\partial ^{2,-}f(\bar{x})$, see \cite{ebpenot:2}. By taking the $\varphi \in
C^{2}(\mathbb{R}^{n})$ in Definition \ref{def:6} and expanding
using a Taylor expansion we may equivalently assert that there exists a $\delta >0$ for which
\begin{equation}
f(x)\geq f(\bar{x})+\langle z,x-\bar{x}\rangle +\frac{1}{2}h^{T}Qh+o(\Vert x-\bar{x}\Vert )\quad \text{ for all }x\in B_{\delta }(\bar{x}),
\label{taylor}
\end{equation}
where $o\left( \cdot \right) $ is the usual Landau small order notation. It
is clear from (\ref{taylor}) that we have $(z,Q)\in \partial^{2,-}f(\bar{x}) $ implies $z\in \partial _{p}f(\bar{x})$ as
\begin{equation*}
f(x)\geq f(\bar{x}) + \langle z,x-\bar{x}\rangle -\frac{r}{2}\Vert x-\bar{x}\Vert \quad \text{ for all }x\in B_{\delta }(\bar{x})
\end{equation*}
when $r>\Vert Q\Vert _{F}$ and $\delta >0$ sufficiently reduced. Moreover $z\in \partial _{p}f(\bar{x})$ implies $(z,-rI)\in \partial ^{2,-}f(\bar{x})$. 
From the definition of prox-regularity at $\bar{x}$ for $\bar{z}$ (and the
choice of $x=\bar{x}$) we conclude that we must have $\bar{z}\in \partial
_{p}f(\bar{x})$ and hence $\partial ^{2,-}f(\bar{x},\bar{z})\neq \emptyset $. Moreover the definition of prox-regularity implies the limiting
subgradients are actually proximal subgradients locally i.e. within an "$f$-attentive neighbourhood of $\bar{z}$" \cite{polrock:1}. When $f$ is subdifferentially
continuous we may drop the $f$-attentiveness and claim $B_{\delta }(\bar{z})\cap \partial f(\bar{x})=B_{\delta }(\bar{z})\cap \partial _{p}f(\bar{x})$
for some sufficiently small $\delta >0$. The example 4.1 of \cite{Lewis:2} show that this neighbourhood can reduce to a singleton $\{\bar{z}\}$. When we have a tilt stable local minimum at $\bar{x}$ or $\bar{z} \in \operatorname{rel-int} \partial f (\bar{x})$ then this situation cannot occur.
\end{remark}

\begin{remark}
	\label{rem:limhess}We denote $(x^{\prime },z^{\prime })\rightarrow
	_{S_{p}(f)}(\bar{x},z)$ to mean $x^{\prime }\rightarrow ^{f}\bar{x}$, $\
	z^{\prime }\in \partial _{p}f(x^{\prime })$ and $z^{\prime }\rightarrow z$.
	As $\partial ^{2,-}f(x^{\prime },z^{\prime })\neq \emptyset $ iff $z^{\prime
	}\in \partial _{p}f\left( x^{\prime }\right) $ it follows via an elementary
	argument that
	\begin{equation*}
	\underline{\partial }^{2}f(\bar{x},\bar{z})=\limsup_{(x^{\prime },z^{\prime
		})\rightarrow _{S_{p}(f)}(\bar{x},\bar{z})}\partial ^{2,-}f(x^{\prime
	},z^{\prime }).
	\end{equation*}
\end{remark}

Denote the recession directions of a convex set $C$ by $0^{+}C$. Noting that
$\langle Q,uv^{T}\rangle =v^{T}Qu$ one may see the motivation for the
introduction of the rank-1 support in (\ref{taylor}). The rank-1 support $q\left( \mathcal{A}\right) (u,v):=\sup \left\{ \langle Q,uv^{T}\rangle \mid
Q\in \mathcal{A}\right\} $ for a subset $\mathcal{A}\subseteq \mathcal{S}\left( n\right) $, in our case $\mathcal{A}=\underline{\partial }^{2}f(\bar{x},z)$. We see from (\ref{taylor}) that when we have $Q\in \partial ^{2,-}f(\bar{x},\bar{z})$ then $Q-P\in \partial ^{2,-}f(\bar{x},\bar{z})$ for any $n\times n$ positive semi-definite matrix $P \in \mathcal{P}(n)$. Thus we always have $-\mathcal{P}(n)\subseteq
0^{+}\partial ^{2,-}f(\bar{x},\bar{z})$ where $\partial ^{2,-}f(\bar{x},\bar{z}) \subseteq \underline{\partial }^{2}f(\bar{x},\bar{z})$.

\begin{theorem}[\protect\cite{eberhard:1}, Theorem 1]
\label{ebthm:rank:1} Let $g:\mathbb{R}^{n}\rightarrow \mathbb{R}_{\infty }$
be proper (i.e. $g(u)\neq -\infty $ anywhere) and $\operatorname{dom}g\neq \emptyset
$. For $u,v\in \mathbb{R}^{n}$, define $q(u,v)=\infty $ if $u$ is not a
positive scalar multiple of $v$ or vice versa, and $q(\alpha u,u)=q(u,\alpha
u)=\alpha g(u)$ for any $\alpha \geq 0$. Then $q$ is a rank one support of a
set $\mathcal{A}\subseteq \mathcal{S}(n)$ with $-\mathcal{P}(n)\subseteq
0^{+}\mathcal{A}$ if and only if

\begin{enumerate}
\item $g$ is positively homogeneous of degree 2.

\item $g$ is lower semicontinuous.

\item $g(-u)=g(u)$ (symmetry).
\end{enumerate}
\end{theorem}

For the sets $\mathcal{A}\subseteq \mathcal{S}(n)$ described in Theorem \ref{ebthm:rank:1} one only needs to consider the support defined on $\mathbb{R}^{n}$ by $q\left( \mathcal{A}\right) (h):=\sup \left\{ \langle
Q,hh^{T}\rangle \mid Q\in \mathcal{A}\right\} $. On reflection it is clear
that all second order directional derivative possess properties 1. and 3. of
the above theorem and those that are topologically well defined possess 2.
as well. We call
\begin{equation*}
\mathcal{A}^{1}:=\{Q\in \mathcal{S}(n)\mid q(\mathcal{A})(h)\geq \langle
Q,hh^{T}\rangle \text{, }\forall h\}
\end{equation*}
the symmetric rank--1 hull of $\mathcal{A}\subseteq \mathcal{S}(n)$. Note
that by definition $q(\mathcal{A})(h)=q(\mathcal{A}^{1})(h)$. When $\mathcal{A}=\mathcal{A}^{1}$, we say $\mathcal{A}$ is a symmetric rank--1
representer. Note that if $Q\in \mathcal{A}^{1}$, then $Q-P\in \mathcal{A}^{1}$ for $P\in \mathcal{P}(n)$ so always $-\mathcal{P}(n)\subseteq 0^{+}\mathcal{A}
$. The rank one barrier cone for a symmetric rank-1 representer is denoted
by $b^{1}(\mathcal{A}):=\{h\in \mathbb{R}^{n}\mid q\left( \mathcal{A}\right)
(h)<\infty \}$. Note that  rank-1 support is an even, positively homogeneous
degree 2 function (i.e. $q\left( \mathcal{A}\right) (h)=q\left( \mathcal{A}\right) (-h)$ and $q\left( \mathcal{A}\right) (th)=t^{2}q\left( \mathcal{A}\right) (h)$). Moreover its domain is the union of a cone $C:= \operatorname{dom} f^{\prime\prime}_{-} (\bar{x}, \bar{z}, \cdot)$ and its negative
i.e.
\begin{equation}
\operatorname{dom}q\left( \mathcal{A}\right) (\cdot ):=b^{1}(\mathcal{A})=C\cup
\left( -C\right) .  \label{neqn:eb34}
\end{equation}
In the first order case we have $\delta^{\ast}_{\partial _{p}f(\bar{x})} (h) \leq
f_{\_}^{\prime }(\bar{x},h)$. A related second order inequality was first observed
in \cite{eberhard:1} \vspace*{-0.2cm}
\begin{equation*}
\vspace*{-0.2cm}q\left( \partial ^{2,-}f(\bar{x},z)\right) (u)=\min
\{f_{\_}^{\prime \prime }(\bar{x},z,u),f_{\_}^{\prime \prime }(\bar{x}
,z,-u)\}=f_{s}^{\prime \prime }\left( \bar{x},z,u\right)
:=\liminf_{t\rightarrow 0,u^{\prime }\rightarrow u}\Delta
_{2}f(x,t,z,u^{\prime })\text{.}
\end{equation*}
Hence if we work with subjets we are in effect dealing with objects dual to
the lower, symmetric, second-order epi-derivative $f_{\_}^{\prime \prime }(\bar{x},z,\cdot)$. Many text book examples of these
quantities can be easily constructed. Moreover there exists a robust calculus for the limiting
subjet \cite{eberhard:7, ebioffe:4}. Furthermore as noted in example 51 of \cite{eberhard:8} the
 qualification condition  for the sum rule for the limiting subjet can hold while for the same problem the basic qualification condition for the sum rule for the limiting (first order) subdifferential can fail to hold. This demonstrates the value of considering  pairs $(z,Q)$.

\begin{example} 	\label{Ex:1}Consider the convex function on $\mathbb{R}^{2}$ given by
$
	f(x,y)=\left\vert x-y\right\vert.
$
	Take $\left(  x,y\right)  =(0,0)$ and $z=(0,0)\in\partial f(0,0)$ then
	$Q=\left(
	\begin{array}
	[c]{cc}%
	\alpha & \gamma\\
	\gamma & \beta
	\end{array}
	\right)  \in\partial^{2,-}f \left(  (0,0)\,(0,0)\right)  $ iff locally
	around $(0,0)$ we have
	\[
	\left\vert x-y\right\vert \geq\frac{1}{2}\left(
	\begin{array}
	[c]{cc}%
	x & y
	\end{array}
	\right)  \left(
	\begin{array}
	[c]{cc}%
	\alpha & \gamma\\
	\gamma & \beta
	\end{array}
	\right)  \left(
	\begin{array}
	[c]{c}%
	x\\
	y
	\end{array}
	\right)  =\frac{1}{2}\left(  \alpha x^{2}+2\gamma xy+\beta y^{2}\right)
	+o\left(  \left\Vert \left(  x,y\right)  \right\Vert ^{2}\right)  \text{.}%
	\]
	This inequality only bites when $x=y$ in which case
	\[
	0\geq\frac{x^{2}}{2}\left(  \alpha+2\gamma+\beta\right)  +o\left(
	x^{2}\right)  \quad\text{or \quad}0\geq\alpha+2\gamma+\beta+\frac{o\left(
		x^{2}\right)  }{x^{2}}\quad\text{so \quad}0\geq\alpha+2\gamma+\beta\text{.}%
	\]
	Consequently
	\[
	\partial^{2,-}f\left(  (0,0)\,(0,0)\right)  =\left\{  Q=\left(
	\begin{array}
	[c]{cc}%
	\alpha & \gamma\\
	\gamma & \beta
	\end{array}
	\right)  \mid0\geq\alpha+2\gamma+\beta\right\}  .
	\]
	The extreme case is when $\alpha+2\gamma+\beta=0$ and two examples of $Q$
	attaining this extremal value are:
	\[
	Q_{1}=\alpha\left(
	\begin{array}
	[c]{cc}%
	1 & 0\\
	0 & -1
	\end{array}
	\right)  \quad\text{and \quad}Q_{2}=\alpha\left(
	\begin{array}
	[c]{cc}%
	1 & -1\\
	-1 & 1
	\end{array}
	\right)  \text{.}%
	\]
 Also
	\begin{align*}
	q\left(  \partial^{2,-}f \left(  (0,0)\,(0,0)\right)  \right)  (h_{1}%
	,h_{2})  &  =\left\{
	\begin{array}
	[c]{cc}%
	0 & \text{if }h_{1}=h_{2}\\
	+\infty & \text{otherwise}%
	\end{array}
	\right\}  = f^{\prime\prime}_{s} (0,0),(0,0),(h_1,h_2)) \\
	\text{and so\quad}b^{1}\left(  \partial^{2,-}f \left(  (0,0)\,(0,0)\right)
	\right)   &  =\left\{  \left(  h_{1},h_{2}\right)  \mid h_{1}=h_{2}\right\}
	\subsetneq\mathbf{R}^{2}.
	\end{align*}
\end{example}

\begin{remark} \label{rem:rankone}
	Knowing the rank-1 barrier cone of a rank-1 representer $\mathcal{A}$ tells us a lot about it's structure. This is no small part to the fact that it consists only of symmetric matrices.  This  discussion has been carried out in quite a bit of detail in \cite{eberhard:7}.   From convex analysis we know that
	the barrier cone (the points at which the support function is finite valued) is polar to the recession directions. In  \cite[Lemma 14]{eberhard:7} it is shown that for a rank-1 representer (using the Frobenious inner product on $\mathcal{S}$ (n))
	this corresponds to $(0^+ \mathcal{A})^{\circ} = \mathcal{P} (b^1 (\mathcal{A}))
	:= \{ \sum_{i \in F} u_i u_i^T \mid u_i \in b^1 (\mathcal{A})\,  \text{for a finite index set $F$}\}$.
	Moreover in \cite[Lemma 24]{eberhard:7} it is shown that 
	$ \mathcal{P} (b^1 (\mathcal{A}))^{\circ} \cap \mathcal{P} (n) = 
	\mathcal{P} (b^1 (\mathcal{A})^{\perp})$. Denoting $\mathcal{U}^2 := b^1 (\mathcal{A})$
	and $\mathcal{V}^2 =  (\mathcal{U}^2)^{\perp}$ we deduce that 
	$\mathcal{P}(\mathcal{V}^2) = (0^+ \mathcal{A}) \cap \mathcal{P} (n)$. This 
	explains why $q(\mathcal{A}) (w) = +\infty$ when $w \notin \mathcal{U}^2$. 
	Since we always have  $-\mathcal{P} (\mathcal{V}^2) \subseteq -\mathcal{P} (n) \subseteq 0^{+} \mathcal{A}$ it follows that $\mathcal{P}(\mathcal{V}^2) -\mathcal{P}(\mathcal{V}^2) \subseteq 0^{+} \mathcal{A}$. Furthermore we find that for any $w = w_{\mathcal{U}^2} + w_{\mathcal{V}^2}$ we then have for $\mathcal{S}(\mathcal{V}^2)$, denoting the symmetric 
	linear mapping from $\mathcal{V}^2$ into $\mathcal{V}^2$, that 
	\[
	\mathcal{A}w =\mathcal{A}w_{\mathcal{U}^2} + \mathcal{A}w_{\mathcal{V}^2}
	\supseteq \mathcal{A}w_{\mathcal{U}^2} + [\mathcal{P}(\mathcal{V}^2) -\mathcal{P}(\mathcal{V}^2)]w_{\mathcal{V}^2}=  \mathcal{A}w_{\mathcal{U}^2} + \mathcal{S}(\mathcal{V}^2)w_{\mathcal{V}^2}. 
	\]

\end{remark}
\subsection{A second order $\mathcal{VU}$ decomposition }

The result \cite{eberhard:6}, Corollary 6.1 contains a number of  observations that
characterise the rank-1 support of the limiting subhessians. We single out the following
which is of particular interest for this paper.

\begin{proposition}[\protect\cite{eberhard:6}, Corollary 6.1]
\label{limpara} Suppose that $f:\mathbb{R}^{n}\rightarrow \mathbb{R}_{\infty
}$ is quadratically minorised and is prox--regular at $\bar{x}\ $ for $\bar{z}\in \partial f(\bar{x})$ with respect to $\varepsilon $ and $r.$ Then $h\mapsto q\left( \underline{\partial }^{2}f(\bar{x},\bar{z})\right) (h)+r\Vert h\Vert ^{2}$ is convex.
\end{proposition}

\begin{proof}
 For the convenience of the reader we provide a
self contained proof of this in the Appendix A.
\end{proof}

\begin{corollary}
Suppose that $f$ is quadratically minorised and is prox--regular at $\bar{x}$
for $\bar{z}\in \partial f(\bar{x})$ with respect to $\varepsilon $ and $r.$
Then $b^{1}(\underline{\partial }^{2}f(\bar{x},\bar{z}))$ is a linear
subspace of $\mathbb{R}^{n}$.
\end{corollary}

\begin{proof}
Note that $b^{1}(\underline{\partial }^{2}f(\bar{x},\bar{z}))=\operatorname{dom}[q\left( \underline{\partial }^{2}f(\bar{x},\bar{z})\right) (\cdot )]$ is
convex under the assumption of Proposition \ref{limpara}. Let $C$ be the
cone given in (\ref{neqn:eb34}) then $b^{1}(\underline{\partial }^{2}f(\bar{x},\bar{z}))=\operatorname{co}(C\cup (-C))=\operatorname{span}C$. As $b^{1}(\underline{\partial }^{2}f(\bar{x},\bar{z}))$ is a symmetric convex cone it is a
subspace.
\end{proof}

\smallskip

\begin{definition}
Let the function $f:\mathbb{R}^{n}\rightarrow \mathbb{R}_{\infty }$ be
finite at $\bar{x}$. When $b^{1}(\underline{\partial }^{2}f(\bar{x},\bar{z}
)) $ is a linear subspace of $\mathbb{R}^{n}$ and $b^{1}(\underline{\partial
}^{2}f(\bar{x},\bar{z}))\subseteq \mathcal{U}$ we call $\mathcal{U}
^{2}:=b^{1}(\underline{\partial }^{2}f(\bar{x},\bar{z}))$ a second order
component of the $\mathcal{U}$-space.
\end{definition}

We will now justify this definition via the following results.

\begin{lemma}
\label{lem:sharp}Suppose $f:\mathbb{R}^{n}\rightarrow \mathbb{R}_{\infty }$
is quadratically minorised and is prox--regular at $\bar{x}\ $ for $\bar{z}\in \partial f(\bar{x})$ with respect to $\varepsilon $ and $r.$ Suppose in
addition that $\bar{z}\in \operatorname{rel}$-$\operatorname{int}\partial f(\bar{x})$. Then
for any $\beta \geq 0$ there is $\varepsilon ^{\prime }>0$ (independent of $\beta $) and a $\epsilon _{\beta }>0$ ($\beta $ dependent) such that we have
$f\left( \bar{x}+u+v\right) \geq f\left( \bar{x}\right) +\langle \bar{z},u+v\rangle +\frac{\beta }{2}\left\Vert v\right\Vert ^{2}-\frac{r}{2}\left\Vert u\right\Vert ^{2}$ whenever $v\in B_{\epsilon _{\beta }}\left(
0\right) $ and $u\in B_{\varepsilon ^{\prime }}\left( 0\right) $.

Moreover we have
\begin{equation}
\mathcal{U}^{2}\subseteq \mathcal{U}=\left\{ h\mid -\delta _{\partial f(\bar{x})}^{\ast }(-h)=\delta _{\partial f(\bar{x})}^{\ast }(h)=\langle \bar{z},h\rangle \right\} .  \label{eqn:44}
\end{equation}
\end{lemma}

\begin{proof}
By the prox-regularity of $f$ at $\bar{x}$ for $\bar{z}\in \partial f(\bar{x})$ with respect to $\varepsilon $ and $r>0$ we have $B_{\delta }(\bar{z}
)\cap \partial f(\bar{x})=B_{\delta }(\bar{z})\cap \partial _{p}f(\bar{x})$
for some sufficiently small $\delta >0$. Thus $\bar{z}\in \operatorname{rel}$-$\operatorname{int}\partial _{p}f(\bar{x})$ and there exists a $\varepsilon ^{\prime }\leq
\min \{\varepsilon ,\delta \}$ such that $\bar{z}+\varepsilon ^{\prime
}B_{1}\left( 0\right) \cap \mathcal{V}\subseteq \partial _{p}f(\bar{x})$ and
$r>0$ such that for $u+v\in B_{\varepsilon ^{\prime }}^{\mathcal{U}}\left( 0\right) \times
B_{\varepsilon ^{\prime }}^{\mathcal{V}}\left( 0\right) $ we have
\begin{eqnarray}
f\left( \bar{x}+u+v\right) &\geq &f\left( \bar{x}\right) +\langle
z,u+v\rangle -\frac{r}{2}\left[ \left\Vert u\right\Vert ^{2}+\left\Vert
v\right\Vert ^{2}\right] \quad \text{for all }z\in \bar{z}+\varepsilon
^{\prime }B_{1}\left( 0\right) \cap \mathcal{V}\   \notag \\
&\geq &f\left( \bar{x}\right) +\langle \bar{z}_{\mathcal{V}},v\rangle
+\langle \bar{z}_{\mathcal{U}},u\rangle +\left( \varepsilon ^{\prime }-\frac{r\left\Vert v\right\Vert }{2}\right) \left\Vert v\right\Vert -\frac{r}{2}
\left\Vert u\right\Vert ^{2}  \text{\  for  }v\in \varepsilon ^{\prime
}B_{1}\left( 0\right) \cap \mathcal{V}  \label{neqn:14} \\
&\geq &f\left( \bar{x}\right) +\langle \bar{z}_{\mathcal{V}},u+v\rangle +\frac{\beta }{2}\left\Vert v\right\Vert ^{2}-\frac{r}{2}\left\Vert
u\right\Vert ^{2}\quad \text{for all }v\in \min \{\varepsilon ^{\prime },\frac{2\varepsilon ^{\prime }}{\beta +r}\}B_{1}\left( 0\right) \cap \mathcal{V},  \notag
\end{eqnarray}
where the last inequality holds due to the fact that $\varepsilon^{\prime}-\frac{r \Vert v \Vert}{2} \geq \beta \Vert v \Vert$. 
Now choose $\epsilon _{\beta }=\min \{\varepsilon ^{\prime },\frac{2\varepsilon ^{\prime }}{\beta +r}\}$.

This inequality implies that for all $\beta >0$ we have
$
\beta I\in P_{\mathcal{V}}^{T}\partial^{2,-}f(\bar{x},\bar{z})P_{\mathcal{V}}
$
and hence when $P_{\mathcal{V}} h \ne 0$ (or $h \notin \mathcal{U}$) we have
$ q\left( \underline{\partial }^{2}f(\bar{x},\bar{z})\right) (h) = +\infty$ and so
$h \notin \mathcal{U}^2$.
\end{proof}

\smallskip

\begin{remark}
This result may hold trivially with both $\mathcal{U} = \mathcal{U}^2 =\{0\}$.  Consider the function $f :\mathbb{R}^2 \to \mathbb{R}$ given by:
\[
f\left(  x,y\right)  =\left\{
\begin{array}
[c]{lc}%
\max\left\{0,x+y\right\}   & :\text{for }x\leq0\text{, }y\geq0\\
\max\left\{0,-x+y\right\}   & :\text{for }x\geq0\text{, }y\geq0\\
\max\left\{0,x-y\right\}   & :\text{for }x\leq0\text{, }y\leq0\\
\max\left\{0,-x-y\right\}   & :\text{for }x\geq0\text{, }y\leq0
\end{array}
\right.
\]
and take $\bar{x} = (0,0)$. Then
$\partial f\left(  0,0\right)  \supseteq \left\{  \left(  0,0\right)  ,\left(
1,1\right)  ,\left(  -1,1\right)  ,\left(  1,-1\right)  ,\left(  -1,-1\right)
\right\}  $ and $\mathcal{U}=\left\{  0\right\}  $ with  $\mathcal{V}
=\mathbb{R}^{2}.$ We have $f$ is prox-regular at $\bar{x}=(0,0)$ for $\bar{z} = (0,0)$
and quadratically memorised (by the zero quadratic). We have $\mathcal{U}^{2}=\left\{  0\right\}  $ as we have $Q_{1}
=\pm\beta\left(  1,1\right)  \left(
\begin{array}
[c]{c}%
1\\
1
\end{array}
\right)  =\pm\beta\left(
\begin{array}
[c]{cc}%
1 & 1\\
1 & 1
\end{array}
\right)  $ and $Q_{2}=\pm\beta\left(  -1,1\right)  \left(
\begin{array}
[c]{c}%
-1\\
1
\end{array}
\right)  =\pm\beta\left(
\begin{array}
[c]{cc}%
1 & -1\\
-1 & 1
\end{array}
\right)  $ with $Q_{1}$, $Q_{2}\in\underline{\partial}^{2}f\left( ( 0,0), (0,0) \right)
$ for all $\beta\geq0$ (approach $(0,0)$ along $x=y$ and $y=-x$ for
$z\rightarrow0$) . Then $q\left(  \underline{\partial}^{2}f\left(  (0,0), (0,0) \right)
\right)  \left(  u,w\right)  =+\infty\geq\beta\max\left\{  \left(
-u+w\right)  ^{2},\left(  u+w\right)  ^{2}\right\}  $ for all $\left(
u,w\right)  \neq\left(  0,0\right)  $ and $\beta\geq 0$.

We note that the examples developed in \cite[Exampls 2, 3]{Mifflin:2004:2} show that the 
assumption that $\bar{z}\in \operatorname{rel}$-$\operatorname{int}\partial f(\bar{x})$
is necessary for Lemma \ref{lem:sharp} to hold. 
\end{remark}

We finish by
generalizing the notion of "fast track" \cite{Lem:1}. 

\begin{definition}
We say $f$ possesses a "fast track" at $\bar{x}$ iff there exists $\bar{z}
\in \partial f\left( \bar{x}\right) $ for which
\begin{equation*}
\mathcal{U}^{2}=b^{1}( \underline{\partial }^{2}f(\bar{x},\bar{z}))=\mathcal{U}.
\end{equation*}
\end{definition}

In the next section after we have introduced the localised $\mathcal{U}$-Lagrangian we will justify this definition further. From Proposition \ref{prop:ebpenot} we see that $\mathcal{U}^{2}=b^{1}(\underline{\partial }^{2}f(\bar{x},\bar{z}))$ provides the subspace within which the eigen-vectors of
the limiting Hessians remain bounded.

\begin{lemma}
\label{lem:boundprox} Suppose $f$ is quadratically minorised and
prox-regular at $\bar{x}$ for $\bar{z}\in \partial f(\bar{x})$
which possesses a nontrivial second order component $\mathcal{U}^{2}\subseteq \mathcal{U}$. Then for all $\left\{ x_{k}\right\} \subseteq
S_{2}(f)$, $x_{k}\rightarrow ^{f}\bar{x}$ with $z_{k}\rightarrow \bar{z}$
and all $h\in \mathcal{U}^{2}$ there is a uniform bound $M>0$ such that for $Q_{k}\in \partial ^{2,-}f\left( x_{k},z_{k}\right) $ we have
\begin{equation}
\langle Q_{k},hh^{T}\rangle \leq M\Vert h\Vert ^{2}\quad \text{ for }k\text{
sufficiently large}.  \label{neqn:34}
\end{equation}
\end{lemma}

\begin{proof}
We have for all $Q\in \underline{\partial }^{2}f(\bar{x},\bar{z})$ and any $h\in \mathcal{U}^{2}$ that
\begin{equation*}
\langle Q,hh^{T}\rangle \leq q\left( \underline{\partial }^{2}f(\bar{x},\bar{z})\right) (h)<+\infty .
\end{equation*}
As $f$ is prox-regular, by Proposition \ref{limpara} $q\left( \underline{\partial }^{2}f(\bar{x},\bar{z})\right) (\cdot )+r\Vert \cdot \Vert ^{2}$ is
convex and finite valued on $\mathcal{U}^{2}$, a closed subspace and
therefore is locally Lipschitz. Thus $q\left( \underline{\partial }^{2}f(\bar{x},\bar{z})\right) (\cdot )$ is locally Lipschitz continuous on $\mathcal{U}^{2}$. Moreover a compactness argument allows us to claim it is
Lipschitz continuous on the unit ball inside the space $\mathcal{U}^{2}$ and
thus obtains a maximum, over the unit ball restricted to the space $\mathcal{U}^{2}$. Hence
\begin{equation*}
\max_{\left\{ h\in \mathcal{U}^{2}\mid \left\Vert h\right\Vert \leq
1\right\} }q\left( \underline{\partial }^{2}f(\bar{x},\bar{z})\right)
(h)\leq K
\end{equation*}
for some $K>0$. On multiplying by $\Vert h\Vert ^{2}$ for $h\in \mathcal{U}^{2}$ and using the positive homogeneity of degree 2 of the rank-1 support
results in following inequality
\begin{equation*}
\langle Q,hh^{T}\rangle \leq q\left( \underline{\partial }^{2}f(\bar{x},\bar{z})\right) (h)\leq K\left\Vert h\right\Vert ^{2}
\end{equation*}
for all $Q\in \underline{\partial }^{2}f(\bar{x},\bar{z})$ and any $h\in
\mathcal{U}^{2}$. Take an arbitrary sequence $(x_{k},z_{k})\rightarrow
_{S_{p}(f)}(\bar{x},\bar{z})$ and $Q_{k}\in \partial ^{2,-}f(x_{k},z_{k})$
with $Q_{k}\rightarrow Q\in \underline{\partial }^{2}f(\bar{x},\bar{z})$
then by taking $M=2K$ we have
\begin{equation*}
\langle Q_{k},hh^{T}\rangle \leq M\left\Vert h\right\Vert ^{2}\quad \text{for }k\text{ sufficiently large. }
\end{equation*}
Moreover any sequence $\left\{ x_{k}\right\} \subseteq
S_{2}(f)$, $x_{k}\rightarrow ^{f}\bar{x}$ with $z_{k}\rightarrow \bar{z}$ has 
$(x_{k},z_{k})\rightarrow_{S_{p}(f)}(\bar{x},\bar{z})$. 
\end{proof}

\subsection{Some Consequences for Coderivatives of $C^{1,1}$ Functions}
As usual we have denoted the indicator function of a set $\mathcal{A}$ by $\delta _{\mathcal{A}}(Q)$ which equals zero if $Q\in \mathcal{A}$ and $+\infty $ otherwise. In general for the recession directions $0^{+}\mathcal{A}^{1}\supseteq -\mathcal{P}(n)$. Consequently the convex support function $\delta _{\mathcal{A}^{1}}^{\ast }\left( P\right) :=\sup \left\{ \langle
Q,P\rangle :=\operatorname{tr}QP\mid Q\in \mathcal{A}^{1}\right\} =+\infty $ if $P\notin \mathcal{P}(n)$. It is noted in \cite[Proposition 4]{eberhard:1} that $0^{+}\mathcal{A}^{1}=-\mathcal{P}(n)$ iff $q(\mathcal{A})(h)<+\infty $ for all $h$.

\begin{lemma}[\protect\cite{eberhard:5}, Lemma 7]
\label{lem:clo}For any $\mathcal{A}\subseteq \mathcal{S}(n)$, then $\operatorname{co}
\left( \mathcal{A}-\mathcal{P}(n)\right) =\mathcal{A}^{1}\text{.}$
\end{lemma}

 For any multi--function $F:\mathbb{R}^{n}\rightrightarrows \mathbb{R}^{m}$ we denote its graph by $\operatorname{Graph}F:=\left\{ (x,y)\mid y\in F(x)\right\} $. The \emph{Mordukhovich coderivative\/} is defined as
\begin{equation*}
D^{\ast }F(x,y)(w):=\{p\in \mathbb{R}^{n}\mid (p,-w)\in \partial \delta _{\operatorname{Graph}\,F}(x,y):=N_{\operatorname{Graph}\,F}(x,y)\}
\end{equation*}
and a second order object $D^{\ast }\left( \partial f\right) (\bar{x},\bar{z})(h)$
is obtained by applying this construction to $F\left( x\right) =\partial
f\left( x\right) $ for $\bar{z} \in \partial f (\bar{x})$. 
We can combine this observation with \cite[Theorem 13.52]{rock:6} that gives
a characterisation of the convex hull of the coderivative in terms of
limiting Hessians for a $C^{1,1}$ function $f$.

\begin{corollary}
	\label{coderiv} Suppose $f$ is locally $C^{1,1}$ around $x$ then the
	Mordukhovich coderivative satisfies
	\begin{eqnarray}
	\operatorname{co}D^{\ast }(\partial f)(x, z)(h) &=&\operatorname{co}\{Ah\mid A=\lim_{k}\nabla
	^{2}f(x^{k})\text{ for some }x\text{$^{k}$($\in S_{2}(f)$)$\rightarrow $}x
	\text{ with }\nabla f\left( x^{k}\right) \rightarrow z\}  \label{neqn:23} \\
	&=&\operatorname{co}\left[ \overline{D}^{2}f(x,z)h\right] =\left[ \operatorname{co}\overline{D}^{2}f(x,z)\right] h\subseteq \left[ \left( \overline{D}^{2}f(x,z)\right)
	^{1}\right] h.  \notag
	\end{eqnarray}
	and
	\begin{equation*}
	\delta _{D^{\ast }(\partial f)(x|z)(h)}^{\ast }\left( h\right) =q\left(
	\underline{\partial }^{2}f( {x},z)\right) \left( h\right) =q\left(
	\overline{D}^{2}f(x,z)\right) \left( h\right) .
	\end{equation*}
\end{corollary}

\begin{proof}
	The first equality of (\ref{neqn:23}) follows from \cite[Theorem 13.52]
	{rock:6} and the second a restatement in terms of $\overline{D}^{2}f(x,z)$.
	The third equality follows from preservation of convexity under a linear
	mapping. Clearly $\operatorname{co}\overline{D}^{2}f(x,z)\subseteq \operatorname{co}\left[
	\overline{D}^{2}f(x,z)-\mathcal{P}(n)\right] =\overline{D}^{2}f(x,z)^{1}$ by
	Lemma \ref{lem:clo}. Moreover we must have by Proposition \ref{prop:ebpenot}
	and the linearity of $Q\mapsto \langle Q,hh^{T}\rangle $ that 
	\begin{eqnarray*}
		q\left( \underline{\partial }^{2}f( {x},z)\right) \left( h\right)
		&=&q\left( \underline{\partial }^{2}f( {x},z)^{1}\right) \left( h\right)
		=q\left( \overline{D}^{2}f(x,z)^{1}\right) \left( h\right) =q\left(
		\overline{D}^{2}f(x,z)-\mathcal{P}(n)\right) \left( h\right) \\
		&=&\sup \left\{ \langle v,h\rangle \mid v\in \operatorname{co}\left[ \overline{D}
		^{2}f(x,z)h\right] \right\} =\sup \left\{ \langle v,h\rangle \mid v\in \operatorname{co}D^{\ast }(\partial f)(x, z)(h)\right\} \\
		&=&\sup \left\{ \langle v,h\rangle \mid v\in D^{\ast }(\partial f)(x, z)(h)\right\}
		=\delta _{D^{\ast }(\partial f)(x, z)(h)}^{\ast }\left( h\right) .
	\end{eqnarray*}
\end{proof}

A central assumption in this paper will be the presence of the following notion of local minimizer.

\begin{definition} [\protect\cite{rock:7}]
	\label{def:tilt}A point $\bar{x}$ gives a tilt stable local
	minimum of a function $f:\mathbb{R}^{n}\rightarrow \mathbb{R}_{\infty }$ if $f\left( \bar{x}\right) $ is finite and there exists an $\varepsilon >0$ such
	that the mapping
	\begin{equation}
		m_{f}:v\mapsto \operatorname{argmin}_{\left\Vert x-\bar{x}\right\Vert \leq
			\varepsilon }\left\{ f\left( x\right) -\langle x,v\rangle \right\}
		\label{eqn:101}
	\end{equation}
	is single valued and Lipschitz on some neighbourhood of $0$ with $	m_{f}\left( 0\right) =\bar{x}$.
\end{definition}

 In \cite[Theorem 1.3]{rock:7} a criterion
for tilt stability was given in terms of second order construction based on
the coderivative of the subdifferential. Assume the first--order condition $0\in \partial f(\bar{x})$ holds. In \cite{rock:7} the second order sufficiency condition
\begin{equation}
\forall \left\Vert h\right\Vert =1\text{, }p\in D^{\ast }\left( \partial
f\right) (\bar{x}, 0)(h)\text{ we have }\langle p,h\rangle   >0
\label{neqn:2}
\end{equation}
is studied and shown to imply a tilt--stable local minimum when $f$ is both subdifferentially continuous and 
prox-regular at $\bar{x}$ for $\bar{z}\in \partial f(\bar{x})$.
We  may reinterpreting the condition (\ref{neqn:2}) for $C^{1,1}$ functions. Indeed thanks to Corollary \ref{coderiv} condition (\ref{neqn:2}) is equivalent to the following.

\begin{corollary}\label{cor:equiv}
	If $f$ is locally $C^{1,1}$ around $x$ then condition (\ref{neqn:2}) is
	equivalent to the existence of $\beta >0$ such that:
	\begin{equation*}
	\forall Q\in \overline{D}^{2}f(x,0)\quad \text{we have }\langle
	Q,hh^{T}\rangle  \geq \beta >0 \quad \text{for all } \| h \|=1.
	\end{equation*}
\end{corollary}

\begin{proof}
	By a simple convexity argument (\ref{neqn:2}) is equivalent to $\langle
	v,h\rangle   >0$ for all $v\in \operatorname{co}D^{\ast }(\partial
	f)(x|0)(h)=\left[ \operatorname{co}\overline{D}^{2}f(x,0)\right] h$ from which we
	have an equivalent condition that $\langle Qh,h\rangle   >0$ for
	all $Q\in \operatorname{co} \overline{D}^{2}f(x,0).$ But $\langle Qh,h\rangle
	=\langle Q,hh^{T}\rangle $ (the Frobenius inner product) and linearity in $Q$
	gives $\langle Qh,h\rangle >0$ for all $Q\in \overline{D}^{2}f(x,0)$ as an equivalent condition.
	Finally we note that $\overline{D}^{2}f(x,0)$ is closed and uniformly bounded due to the local
	Lipschitzness of the gradient $x \mapsto \nabla f(x)$ so via a compactness argument $\langle Qh,h\rangle \geq \beta >0$ for some $\beta >0$.
\end{proof}

\begin{remark}
	It would be interesting to have characterisation of subjets for functions other than 
	those that are $C^{1,1}$ smooth, in order to compare with their corresponding second order coderivative. 
	Consider a characterisation of the coderivative for a class of functions found in 
	\cite[Coroillary 5.4, Theorem 5.3]{Lewis:2} (which are not a $C^{1,1}$ functions). 
	Then we have:
	\[
	D^{\ast} (\partial f )(\bar{x} , 0 ) (w) 
	=\left\{
	\begin{array}
	[c]{cl}
	\nabla^2_{\mathcal{M}} f( \bar{x}) w + N_{\mathcal{M}} (\bar{x})  & :\text{for } w \in T_{\mathcal{M}} ( \bar{x}) \\
	\emptyset & :\text{for }   w \not\in T_{\mathcal{M}} ( \bar{x}) . \\
	\end{array}
	\right.
	\]
	In this context of this paper we have  $\mathcal{M} := \{ (u ,  v (u) )  \mid u \in \mathcal{U} \}$ and if we assume this  is a $C^2$ smooth manifold we have  $\nabla^2_{\mathcal{M}} f( \bar{x}) w  
	= \frac{d^2}{dt^2} f(\bar{x} + tw + v(tw) )|_{t=0}$ with $T_{\mathcal{M}} ( \bar{x}) 
	= \mathcal{U}^2$ and $N_{\mathcal{M}} (\bar{x}) = \mathcal{V}^2$. 
	It seems possible that the calculus provided by \cite{ebioffe:4, eberhard:7} could provide an 
	avenue to calculate $\underline{\partial}^2 f (\bar{x}, 0 )$ for this class of functions. 
\end{remark}

\section{The localised $\mathcal{U}^{\prime}$-Lagrangian}

For the remainder of the paper we will assume $\bar{z} \in \operatorname{rel}$-$\operatorname{int}\partial f\left( \bar{x}\right) \neq \emptyset $ and so $\mathcal{V}:=\operatorname{span}
\left\{ \partial f\left( \bar{x}\right) -\bar{z}\right\} $, $\mathcal{U} = (\mathcal{V})^{\perp}$, as  defined in \cite{Lem:1, Mifflin:2003} and 
coinciding with the space defined in section \ref{sec:VU}.  When
discussing tilt stability we will to assume $\bar{z}=0\in \partial f\left(
\bar{x}\right) $. Then we define the localised $\mathcal{U}^{\prime }$-Lagrangian, for any subspace $\mathcal{U}^{\prime }\subseteq \mathcal{U}$
and some $\varepsilon >0$, to be the function
\begin{equation*}
L_{\mathcal{U^{\prime }}}^{\varepsilon }\left( u\right) :=\left\{
\begin{array}{cc}
\inf_{v^{\prime }\in \mathcal{V^{\prime }}\cap B_{\varepsilon }\left(
0\right) }\left\{ f\left( \bar{x}+u+v^{\prime }\right) -\langle \bar{z}_{\mathcal{V}^{\prime }},v^{\prime }\rangle \right\} & \text{for }u\in
\mathcal{U}^{\prime }\cap B_{\varepsilon }\left( 0\right):= B_{\varepsilon }^{\mathcal{U^{\prime }}}\left( 0\right) \\
+\infty & \text{otherwise}
\end{array}
\right.
\end{equation*}
where $\mathcal{V^{\prime }}:=\mathcal{U^{\prime }}^{\perp }$. Let
\begin{equation}
v\left( u\right) \in \operatorname{argmin}_{v^{\prime }\in \mathcal{V^{\prime }}\cap
B_{\varepsilon }\left( 0\right) }\left\{ f\left( \bar{x}+u+v^{\prime
}\right) -\langle \bar{z}_{\mathcal{V^{\prime }}},v^{\prime }\rangle
\right\} .  \label{eqn:78}
\end{equation}
This Lagrangian differs from the modification introduced by Hare \cite{Hare:3} in that $L_{\mathcal{U^{\prime }}}^{\varepsilon }\left( \cdot
\right) $ is locally well defined on $\mathcal{U^{\prime }}$ due to the
introduction of the ball $B_{\varepsilon }^{\mathcal{V^{\prime }}}\left(
0\right) =\mathcal{V^{\prime }}\cap B_{\varepsilon }\left( 0\right) $ over
which the infimum is taken. Hare assumes a quadratic minorant to justify a
finite value for a sufficiently large regularization parameter used in the
so-called quadratic sub-Lagrangian. Define for $u\in \mathcal{U^{\prime }}$
and $v\left( \cdot \right) :\mathcal{U^{\prime }}\rightarrow B_{\varepsilon }^{\mathcal{V^{\prime }}}\left( 0\right) $ the auxiliary functions
\begin{eqnarray*}
k_{v}\left( u\right) := &&h\left( u+v\left( u\right) \right) -\langle \bar{z}_{\mathcal{V^{\prime }}},u+v\left( u\right) \rangle \\
\text{where\quad }h\left( w\right) := &&f\left( \bar{x}+w\right) +\delta _{B_{\varepsilon }^{\mathcal{U^{\prime }}}\left( 0\right) 
\oplus B_{\varepsilon }^{\mathcal{V^{\prime }}}\left( 0\right) }\left( w\right) .
\end{eqnarray*}
Then
\begin{equation*}
L_{\mathcal{U^{\prime }}}^{\varepsilon }\left( u\right) :=\inf_{v^{\prime
}\in \mathcal{V^{\prime }}}\left\{ h\left( u+v^{\prime }\right) -\langle
\bar{z}_{\mathcal{V}^{\prime }},v^{\prime }\rangle \right\} .
\end{equation*}
When $v(\cdot )$ is chosen as in (\ref{eqn:78}) we have $L_{\mathcal{U^{\prime }}}^{\varepsilon }\left( u\right) =k_{v}(u)$ with both infinite
outside $B_{\varepsilon }^{\mathcal{U^{\prime }}}\left( 0\right) $.

\begin{lemma} \label{lem:27}
Suppose $f:\mathbb{R}^{n}\rightarrow \mathbb{R}_{\infty }$ is a proper lower
semi-continuous function and assume $v(\cdot )$ is chosen as in (\ref{eqn:78}). The conjugate of $k_{v}:\mathcal{U}^{\prime }\rightarrow \mathbb{R}_{\infty }$ with respect to $\mathcal{U^{\prime }}$ is given by
\begin{equation}
k_{v}^{\ast }\left( z_{\mathcal{U^{\prime }}}\right) :=\sup_{u\in \mathcal{U^{\prime }}}\left\{ \langle u,z_{\mathcal{U^{\prime }}}\rangle -k_{v}\left(
u\right) \right\} =h^{\ast }\left( z_{\mathcal{U^{\prime }}}+\bar{z}_{\mathcal{V^{\prime }}}\right) =\left( L_{\mathcal{U^{\prime }}}^{\varepsilon
}\right) ^{\ast }\left( z_{\mathcal{U^{\prime }}}\right) .  \label{neqn:33}
\end{equation}
\end{lemma}

\begin{proof}
By direct calculation we have
\begin{eqnarray*}
k_{v}^{\ast }\left( z_{\mathcal{U^{\prime}}}\right) &=&\sup_{u\in \mathcal{U^{\prime}} }\left\{ \langle u,z_{\mathcal{U^{\prime}}}\rangle -\left\{
h\left( u+v\left( u\right) \right) -\langle \bar{z}_{\mathcal{V^{\prime}}
},u+v\left( u\right) \rangle \right\} \right\} \\
&=&\sup_{u\in \mathcal{U^{\prime}}}\left\{ \langle u,z_{\mathcal{U^{\prime}}
}\rangle -\min_{v^{\prime }\in \mathcal{V^{\prime}}}\left\{ h\left(
u+v^{\prime }\right) -\langle \bar{z}_{\mathcal{V^{\prime}}},u+
v^{\prime}\rangle \right\} \right\} \\
&=&\sup_{\left( u,v^{\prime }\right) \in \mathcal{U^{\prime}}\oplus \mathcal{V^{\prime}} }\left\{ \langle u+v^{\prime},z_{\mathcal{U^{\prime}}}+ \bar{z}_{\mathcal{V^{\prime}}}\rangle 
-h\left( u+v^{\prime }\right) \right\} =h^{\ast }\left(
z_{\mathcal{U^{\prime}}}+\bar{z}_{\mathcal{V^{\prime}}}\right)
\end{eqnarray*}
as $\langle z_{\mathcal{U^{\prime}}},v^{\prime }\rangle =0$ for all $v^{\prime }\in \mathcal{V^{\prime}}$. Also
\begin{eqnarray*}
k_{v}^{\ast }\left( z_{\mathcal{U^{\prime}}}\right) &=& \sup_{u\in \mathcal{U^{\prime}} }\left\{ \langle u,z_{\mathcal{U^{\prime}}}\rangle
-\min_{v^{\prime }\in \mathcal{V^{\prime}} }\left\{ h\left( u+v^{\prime
}\right) -\langle \bar{z}_{\mathcal{V^{\prime}} },u+v^{\prime }\rangle
\right\} \right\} \\
&=&\sup_{u\in \mathcal{U^{\prime}}}\left\{ \langle u,z_{\mathcal{U^{\prime}}
}\rangle -L_{\mathcal{U^{\prime}}}^{\varepsilon }\left( u\right) \right\}
=\left( L_{\mathcal{U^{\prime}} }^{\varepsilon }\right) ^{\ast } \left( z_{\mathcal{U^{\prime}}}\right) .
\end{eqnarray*}
\end{proof}

When we assume $\bar{x}$ gives a tilt stable local minimum of a function $f:\mathbb{R}^{n}\rightarrow \mathbb{R}_{\infty }$ we shall choose the $\varepsilon >0$ to be consistent with the definition of  tilt stability at $\bar{x}$ for the neighbourhood
\begin{equation*}
B_{\varepsilon }^{\mathcal{U^{\prime }}}\left( \bar{x}_{\mathcal{U^{\prime }}}\right) \oplus B_{\varepsilon }^{\mathcal{V^{\prime }}}\left( \bar{x}_{\mathcal{V^{\prime }}}\right) :=\left\{ \left( x_{\mathcal{U^{\prime }}},x_{\mathcal{V^{\prime }}}\right) \in \mathcal{U^{\prime }}\oplus \mathcal{V^{\prime }}\mid \left\Vert x_{\mathcal{U^{\prime }}}-\bar{x}_{\mathcal{U^{\prime }}}\right\Vert \leq \varepsilon \text{ and }\left\Vert x_{\mathcal{V^{\prime }}}-\bar{x}_{\mathcal{V^{\prime }}}\right\Vert \leq \varepsilon
\right\} 
\end{equation*}
where $\varepsilon$ is reduced to contain the above neighbourhood in a larger ball
$\{x \in \mathbb{R}^n \mid \| x - \bar{x} \| \leq \hat{\varepsilon} \}$ on which tilt stability holds. 
We will rely on the results of \cite{Drusvy:1}. From definition \ref{def:tilt} we have on $B_{\varepsilon }^{\mathcal{U^{\prime }}}\left( \bar{x}_{\mathcal{U^{\prime }}}\right) \oplus B_{\varepsilon }^{\mathcal{V^{\prime }
}}\left( \bar{x}_{\mathcal{V^{\prime }}}\right) $ that
\begin{equation}
f\left( x\right) \geq f\left( m_{f}\left( v\right) \right) +\langle
x-m_{f}\left( v\right) ,v\rangle \label{leo:query}
\end{equation}
where $m_{f}(\cdot )$ is as defined in (\ref{eqn:101}). That is, we have a supporting tangent plane to the epigraph of $f +\delta_{B_{\varepsilon }^{\mathcal{U^{\prime }}}\left( \bar{x}_{\mathcal{U^{\prime }}}\right) \oplus B_{\varepsilon }^{\mathcal{V^{\prime }
		}}\left( \bar{x}_{\mathcal{V^{\prime }}}\right) }$. As the convex hull of any set (including the  epigraph of $f +\delta_{B_{\varepsilon }^{\mathcal{U^{\prime }}}\left( \bar{x}_{\mathcal{U^{\prime }}}\right) \oplus B_{\varepsilon }^{\mathcal{V^{\prime }
	}}\left( \bar{x}_{\mathcal{V^{\prime }}}\right) }$) must remain on the same side of any supporting hyperplane (in this case the hyperplane $(x,\alpha) \mapsto \langle  (x,\alpha) -(m_{f}\left( v\right)  ,  f(m_{f}\left( v\right) ) , (v, -1)\rangle \leq 0$) we may deduce that (again locally) 
\begin{equation*}
\operatorname{co}f\left( x\right) \geq f\left( m_{f}\left( v\right) \right) +\langle
x-m_{f}\left( v\right) ,v\rangle .
\end{equation*}
This observation leads to the following minor rewording of the result from
\cite{Drusvy:1}. It shows that there is a strong convexification process
involved with tilt stability.

\begin{proposition} [\protect\cite{Drusvy:1}, Proposition 2.6]
 \label{prop:co}Consider $f:\mathbb{R}^{n}\rightarrow \mathbb{R}_{\infty }$ is a proper lower semi-continuous
function and suppose that $\bar{x}$ give a tilt stable local minimum of $f$.
Then for all sufficiently small $\varepsilon >0,$ in terms of the function $h\left( w\right) :=f\left( \bar{x}+w\right) +\delta _{B_{\varepsilon }^{\mathcal{U^{\prime }}}\left( 0\right) \oplus B_{\varepsilon }^{\mathcal{V^{\prime }}}\left( 0\right) }\left( w\right) $ we have
\begin{equation*}
\operatorname{argmin}_{x\in B_{\varepsilon }^{\mathcal{U^{\prime }}}\left( \bar{x}_{\mathcal{U^{\prime }}}\right) \oplus B_{\varepsilon }^{\mathcal{V^{\prime }}}\left( \bar{x}_{\mathcal{V^{\prime }}}\right) }\left[ f\left( x\right)
-\langle x,z\rangle \right] =\operatorname{argmin}_{\left( u^{\prime },v^{\prime
}\right) \in \mathcal{U^{\prime }}\oplus \mathcal{V^{\prime }}}\left[ \operatorname{co}h\left( u^{\prime }+v^{\prime }\right) -\langle u^{\prime }+v^{\prime },z\rangle \right] +\bar{x}
\end{equation*}
for all $z$ sufficiently close to $0.$ Consequently $0$ is a tilt stable
local minimum of $\operatorname{co}h$.
\end{proposition}

We now study the subgradients of the $\mathcal{U^{\prime }}$-Lagrangian.
In order to simplify statements we introduce the following modified function:
\begin{equation*}
m_{h}  : z \mapsto \operatorname{argmin}_{\left( u^{\prime },v^{\prime }\right)
\in \mathcal{U^{\prime }}\oplus \mathcal{V^{\prime }}}\left[ \operatorname{co}h\left( u^{\prime }+v^{\prime }\right) -\langle u^{\prime }+v^{\prime },z\rangle \right]
\end{equation*}
then we have $m_{h}\left( z\right) +\bar{x}=m_{f}\left( z\right) $ for $m_{f}\left( z\right) :=\operatorname{argmin}_{x\in B_{\varepsilon }^{\mathcal{U^{\prime }}}\left( \bar{x}_{\mathcal{U^{\prime }}}\right) \oplus
B_{\varepsilon }^{\mathcal{V^{\prime }}}\left( \bar{x}_{\mathcal{V^{\prime }}}\right) }\left[ f\left( x\right) -\langle x,z\rangle \right] $. The next
result shows that under the assumption of tilt stability we have $u:=P_{\mathcal{U^{\prime }}}\left[ m_{h}\left( z_{\mathcal{U^{\prime }}}+\bar{z}_{\mathcal{V^{\prime }}}\right) \right] $ iff
\begin{equation}
z_{\mathcal{U^{\prime }}}\in \partial _{\operatorname{co}}\ L_{\mathcal{U^{\prime }}}^{\varepsilon } \left( u\right)  \label{neqn:10}
\end{equation}
where $\partial _{\operatorname{co}}g\left( u\right) :=\left\{ z\mid g\left(
u^{\prime }\right) -g\left( u\right) \geq \langle z,u^{\prime }-u\rangle
\text{ for all }u^{\prime }\right\} $ corresponds to the subdifferential of
convex analysis. In passing we note that tilt stability
of $f$ at $\bar{x}$ implies $\partial_{\operatorname{co}} f (\bar{x}) \ne \emptyset$. 

\begin{remark} \label{rem:27}
When $f:\mathbb{R}^{n}\rightarrow \mathbb{R}_{\infty }$ has a tilt-stable
local minimum at $\bar{x}$ then for $\bar{z}$ sufficiently small we must also have
$g\left( x\right) :=f\left( x\right) -\langle \bar{z},x\rangle $ possessing
a tilt stable local minimum at $\left\{ \bar{x}\right\} =m_{f}\left( \bar{z}\right)$. In this 
way we may obtain a unique Lipschitz continuous selection 
\[
\{ m_{h}\left( z_{\mathcal{U^{\prime }}}+\bar{z}_{\mathcal{V^{\prime }}}\right)  \}
= \operatorname{argmin}_{\left( u^{\prime },v^{\prime }\right) \in \mathcal{U} \oplus \mathcal{V}}\left[ h \left( u^{\prime }+v^{\prime }\right)
-\langle v^{\prime },\bar{z}_{\mathcal{V^{\prime }}}\rangle -\langle
u^{\prime },z_{\mathcal{U^{\prime }}}\rangle \right] 
\]
in a neighbourhood of $z_{\mathcal{U^{\prime }}} \in B_{\varepsilon} (\bar{z}_{\mathcal{U^{\prime }}}) $ (where $\bar{z}_{\mathcal{U^{\prime }}} \ne 0$). 
\end{remark}

\begin{proposition}
\label{prop:LU} Let $f:\mathbb{R}^{n}\rightarrow \mathbb{R}_{\infty }$ be a
proper lower semi-continuous function with $f-\langle \bar{z},\cdot \rangle $
having a tilt-stable local minimum at $\bar{x}$.

\begin{enumerate}
\item Then $L_{\mathcal{U^{\prime }}}^{\varepsilon }\left( \cdot \right) $ is closed, proper convex function that is finite valued for $u \in B_{\varepsilon }^{\mathcal{U^{\prime }}}\left( 0\right)$. 
\item  Let $u:=P_{\mathcal{U^{\prime }}}\left[ m_{h}\left( z_{\mathcal{U^{\prime }}}+\bar{z}_{\mathcal{V^{\prime }}}\right) \right] \in \operatorname{int}
B_{\varepsilon }^{\mathcal{U^{\prime }}}\left( 0\right) $ (where $z_{\mathcal{U}^{\prime }}\in \mathcal{U}^{\prime }$) then 
\begin{equation}
L_{\mathcal{U^{\prime }}}^{\varepsilon }\left( u^{\prime }\right) -L_{\mathcal{U^{\prime }}}^{\varepsilon }\left( u\right) \geq \langle z_{\mathcal{U^{\prime }}},u^{\prime }-u\rangle \quad \text{for }u^{\prime }\in
B_{\varepsilon }^{\mathcal{U^{\prime }}}\left( 0\right) .  \label{neqn:9}
\end{equation}
Moreover $L_{\mathcal{U^{\prime }}}^{\varepsilon }\left( u\right)
=\min_{v^{\prime }\in \mathcal{V^{\prime }}}\left[ \operatorname{co}h\left(
u+v^{\prime }\right) -\langle v^{\prime },\bar{z}_{\mathcal{V^{\prime }}
}\rangle \right] $ for which the minimum is attained at $v\left( u\right) =P_{\mathcal{V^{\prime }}}\left[ m_{h}\left( z_{\mathcal{U^{\prime }}}+\bar{z}_{\mathcal{V^{\prime }}}\right) \right] $ where $v\left( 0\right) =0$.

\item Conversely suppose (\ref{neqn:10}) holds at any given $u\in
B_{\varepsilon }^{\mathcal{U^{\prime }}}\left( 0\right) $ and let $v(u)$ be
as defined in (\ref{eqn:78}). Then we have $u=P_{\mathcal{U^{\prime }}}\left[
m_{h}\left( z_{\mathcal{U^{\prime }}}+\bar{z}_{\mathcal{V^{\prime }}}\right)
\right] $ and $v\left( u\right) =P_{\mathcal{V^{\prime }}}\left[ m_{h}\left(
z_{\mathcal{U^{\prime }}}+\bar{z}_{\mathcal{V^{\prime }}}\right) \right] \in
\operatorname{int}B_{\varepsilon }^{\mathcal{V^{\prime }}}\left( 0\right) $ for $\|u\|$ sufficiently small.
\end{enumerate}
\end{proposition}

\begin{proof}
Consider 1.  By  Proposition \ref{prop:co} we have
\begin{equation*}
L_{\mathcal{U^{\prime }}}^{\varepsilon }\left( u\right) =\min_{v^{\prime
	}\in \mathcal{V^{\prime }}}\left[ h\left( u+v^{\prime }\right) -\langle
v^{\prime },\bar{z}_{\mathcal{V^{\prime }}}\rangle \right] =\min_{v^{\prime
	}\in \mathcal{V^{\prime }}}\left[ \operatorname{co}h\left( u+v^{\prime }\right)
-\langle v^{\prime },\bar{z}_{\mathcal{V^{\prime }}}\rangle \right] .
\end{equation*}
Hence $L_{\mathcal{U^{\prime }}}^{\varepsilon }\left( u^{\prime}\right) $ is a "marginal mapping" corresponding to a coercive closed convex function $F(u^{\prime},v^{\prime}):=\operatorname{co}h\left( u^{\prime}+v^{\prime }\right)
-\langle v^{\prime },\bar{z}_{\mathcal{V^{\prime }}}\rangle$. Applying \cite[Theorem 9.2]{rock:1} the result follows on viewing  $L_{\mathcal{U^{\prime }}}^{\varepsilon }\left( u^{\prime}\right) $  as the "image of $F$ under the linear mapping $A$" given by the projection $u^{\prime} := A (u^{\prime}, v^{\prime}) :=P_{\mathcal{U}}(u^{\prime}, v^{\prime})$ onto $\operatorname{int} B_{\varepsilon }^{\mathcal{U^{\prime }}}\left( 0\right) $. 

For the second part we have $z=z_{\mathcal{U^{\prime }}}+\bar{z}_{\mathcal{V^{\prime }}}$, where only the $\mathcal{U^{\prime }}$ component
varies. The following minimum attained at the unique point $m_{h}\left( z_{\mathcal{U^{\prime }}}+\bar{z}_{\mathcal{V^{\prime }}}\right) $ that
uniquely determines the value of $u\in \mathcal{U^{\prime }}$:
\begin{eqnarray}
\left\{ u+v(u)\right\} &:= &m_{h}\left( z_{\mathcal{U^{\prime }}}+\bar{z}_{\mathcal{V^{\prime }}}\right)  \notag \\
&=&\operatorname{argmin}_{\left( u^{\prime },v^{\prime }\right) \in \mathcal{U} \oplus \mathcal{V}}\left[ h\left( u^{\prime }+v^{\prime }\right)
-\langle v^{\prime },\bar{z}_{\mathcal{V^{\prime }}}\rangle -\langle
u^{\prime },z_{\mathcal{U^{\prime }}}\rangle \right]  \notag \\
\text{and so \quad }\left\{ u\right\} &=&\operatorname{argmin}_{u^{\prime }\in
B_{\varepsilon }^{\mathcal{U^{\prime }}}\left( 0\right) }\left[
\min_{v^{\prime }\in \mathcal{V^{\prime }}}\left[ h\left( u^{\prime
}+v^{\prime }\right) -\langle v^{\prime },\bar{z}_{\mathcal{V^{\prime }}
}\rangle \right] -\langle u^{\prime },z_{\mathcal{U^{\prime }}}\rangle
\right] ,  \label{neqn:80}
\end{eqnarray}
where $u:=P_{\mathcal{U^{\prime }}}\left[ m_{h}\left( z_{\mathcal{U^{\prime }}
}+\bar{z}_{\mathcal{V^{\prime }}}\right) \right] $ and $v\left( u\right)
:=P_{\mathcal{V^{\prime }}}\left[ m_{h}\left( z_{\mathcal{U^{\prime }}}+\bar{z}_{\mathcal{V^{\prime }}}\right) \right]$. 
As $m_{h}\left( \cdot \right) $ is a single valued Lipschitz function and $\operatorname{co}h$ has a local minimum at $0$ then  $v\left( 0\right) =0$ because
$\left\{ 0\right\} =\operatorname{argmin}_{v^{\prime }\in \mathcal{V^{\prime }}}\left[ \operatorname{co}h\left( v^{\prime }\right) -\langle v^{\prime },\bar{z}_{\mathcal{V^{\prime }}}\rangle \right] $. Hence by continuity  $v(u) \in \operatorname{int} B_{\varepsilon }^{\mathcal{U^{\prime }}}$ for $\|u\|$ sufficiently small. 
	The objective value on this
minimization problem equals
\begin{equation}
\min_{u^{\prime }\in B_{\varepsilon }^{\mathcal{U^{\prime }}}\left( 0\right)
}\left[ L_{\mathcal{U^{\prime }}}^{\varepsilon }\left( u^{\prime }\right)
-\langle u^{\prime },{z}_{\mathcal{U^{\prime }}}\rangle \right] =L_{\mathcal{U^{\prime }}}^{\varepsilon }\left( u\right) -\langle u,z_{\mathcal{U^{\prime
}}}\rangle ,  \label{eqn:79}
\end{equation}
giving (\ref{neqn:9}).

For the third part we note that (\ref{neqn:10}) is equivalent to (\ref{neqn:9}) and hence equivalent to the identity (\ref{eqn:79}), which affirms
that the minimizer in the $\mathcal{U^{\prime }}$ space is attained at $u$
and thus the minimizer in the $\mathcal{V^{\prime }}$ space in the
definition of $L_{\mathcal{U^{\prime }}}^{\varepsilon }\left( u\right) $ is
attained at $v(u)$. This in turn can be equivalently written as (\ref{neqn:80}) which affirms that $u=P_{\mathcal{U^{\prime }}}\left[ m_{h}\left(
z_{\mathcal{U^{\prime }}}+\bar{z}_{\mathcal{V^{\prime }}}\right) \right] $
and $v\left( u\right) =P_{\mathcal{V^{\prime }}}\left[ m_{h}\left( z_{\mathcal{U^{\prime }}}+\bar{z}_{\mathcal{V^{\prime }}}\right) \right] \in
B_{\varepsilon }^{\mathcal{V^{\prime }}}\left( 0\right) $.
\end{proof}
\smallskip

\begin{remark}\label{rem:implicit1}
	In principle the knowledge of $m_f$ and $\mathcal{U}$ should allow one to construct the function $v(\cdot)$. One can perform a rotation of coordinates and a translation of 
	$\bar{x}$ to zero so that we have then $f$ represented as $h: \mathcal{U} \times \mathcal{V}
	\to \mathbb{R}_{\infty}$ and correspondingly obtain $m_h$. Now decompose 
	$m_h (z_{\mathcal{U}} + \bar{z}_{\mathcal{V}}) = m^h_{\mathcal{U}}(z_{\mathcal{U}} ) +  m^h_{\mathcal{V}}(z_{\mathcal{U}} )$ (where we have drop the reference to $\bar{z}_{\mathcal{V}}$ as it's value is fixed). Then eliminate the variable $z_{\mathcal{U}} $ 	from the system of equations $u = m^h_{\mathcal{U}}(z_{\mathcal{U}} ) $ and 
	$v = m^h_{\mathcal{V}}(z_{\mathcal{U}} )$ to obtain $v(u)$. This solution is unique
	under the assumption of a tilt stable local minimum. Indeed one can interpret $v(\cdot)$ as an implicit function. This point of view has been used by numerous authors
	\cite[Theorem 2.2]{Miller:1}, \cite[Theorem 6.1]{Lewis:1} and with regard to 
	$C^2$-smooth manifolds see \cite[Theorem 2.6]{Lewis:2}. This last result indicates that when 
	$f$ is "partially smooth" with respect to a $C^2$-smooth manifold $\mathcal{M}$ then 
	the form of $v( \cdot)$ is accessible via the implicit function theorem. Moreover  there is a local description $\mathcal{M} = \{ (u, v(u)) \mid u \in \mathcal{U} \cap B_{\varepsilon} (0)\}$. An interesting example of 
	this sort of approach can be found in \cite[Theorem 4.3]{Lem:1}. Here the exact penalty function of a convex nonlinear optimisation problem is studied where $\bar{x}$ is chosen to be the minimizer. The function 
	$v(\cdot)$ is characterised as the solution to a system of equation associated with the active 
	constraints at $\bar{x}$  for the associated nonlinear programming problem. A similar analysis may be applied to the illustrative example of $C^2$ smooth function $f$ restricted to a polyhedral set $P := \{ x \in \mathbb{R}^n \mid l_i (x) \leq 0 \text{ for } i \in I:=\{1,\dots,m\}\}$, where
	$l_i$ are affine functions and  $I(\bar{x}):=\{ i \in I \mid l_i (\bar{x}) = 0 \}$ are the active constraints. Assume $\{\nabla l_i (\bar{x})\}_{i \in I(\bar{x})}$ are linearly independent. When the optimal solution $\bar{x} \in \operatorname{int} P$ then $\mathcal{V} = \{0\}$ and $\mathcal{U} = \mathbb{R}^n$ giving $\mathcal{M} = \mathbb{R}^n \times \{0\}$, a smooth manifold. When the active constraints  $I(\bar{x})$ are nonempty then 
	$\mathcal{V} = \operatorname{lin} \{\nabla l_i (\bar{x})\}_{i \in I(\bar{x})}\}$ and
	$\mathcal{U} = \{d \in \mathbb{R}^n  \mid \langle \nabla l_i (\bar{x}) ,d \rangle = 0 \text{ for } i \in I(\bar{x}) \}$. Then $v(u)$ is the solution (or implicit function) associated with the system of equation $l_i (\bar{x} + (u,v)) =0$ for 
	$i \in I(\bar{x})$,  in the unknowns $v \in \mathcal{V}$. The implicit function theorem now furnishes existence, uniqueness and differentiability. Given this clear connection to implicit functions it would be interesting to relate these ideas to a more modern theory of implicit functions \cite{Dontchev:1}.
\end{remark}
\smallskip

Existence of convex subgradients indicates a hidden convexification.

\begin{lemma}
\label{lem:conv}Consider $h:\mathcal{U^{\prime }}\rightarrow \mathbb{R}_{\infty }$ is a proper lower semi-continuous function. Then
\begin{equation*}
\partial _{\operatorname{co}}h\left( u\right) \subseteq \partial \left[ \operatorname{co}h\right] \left( u\right) .
\end{equation*}
When $\partial _{\operatorname{co}}h\left( u\right) \neq \emptyset $ then $\operatorname{co}h\left( u\right) =h\left( u\right) $ and we have $\partial _{\operatorname{co}
}h\left( u\right) =\partial \left[ \operatorname{co}h\right] \left( u\right)
\subseteq \partial _{p}h\left( u\right) \neq \emptyset .$ If in addition $h$
is differentiable we have $\nabla h\left( u\right) =\nabla \left( \operatorname{co}
h\right) \left( u\right) $.
\end{lemma}

\begin{proof}
If $z_{\mathcal{U}^{\prime}}\in \partial _{\operatorname{co}}h\left( u\right) $ then
\begin{eqnarray}
h\left( u^{\prime }\right) -h\left( u\right) &\geq &\langle z_{\mathcal{U^{\prime }}},u^{\prime }-u\rangle \quad \text{for all }u^{\prime }\in
\mathcal{U^{\prime }}  \label{neqn:7} \\
\text{hence \quad }\operatorname{co}h\left( u^{\prime }\right) &\geq &h\left(
u\right) +\langle z_{\mathcal{U^{\prime }}},u^{\prime }-u\rangle  \notag
\end{eqnarray}
and so for $u^{\prime }=u$ we have $\operatorname{co}h\left( u\right) \geq h\left(
u\right) \geq \operatorname{co}h\left( u\right) $ giving equality. Thus
\begin{equation}
\operatorname{co}h\left( u^{\prime }\right) -\operatorname{co}h\left( u\right) \geq \langle
z_{\mathcal{U}},u^{\prime }-u\rangle \quad \text{for all }u^{\prime }\in
\mathcal{U^{\prime }}\text{. }  \label{neqn:8}
\end{equation}
Hence $\partial _{\operatorname{co}}h\left( u\right) \subseteq \partial \left[ \operatorname{co}h\right] \left( u\right) .$ When $\partial _{\operatorname{co}}h\left( u\right)
\neq \emptyset $ then $\operatorname{co}h\left( u\right) =h\left( u\right) $ and (\ref{neqn:8})\ gives (\ref{neqn:7}) as $h\left( u^{\prime }\right) \geq
\operatorname{co}h\left( u^{\prime }\right) $ is always true. In particular (\ref{neqn:7}) implies $z_{\mathcal{U^{\prime }}}\in \partial _{p}h\left(
u\right) $ and when $h$ is actually differentiable at $u$ then$\ \partial _{\operatorname{co}}h\left( u\right) =\partial \left[ \operatorname{co}h\right] \left(
u\right) \subseteq \partial _{p}h\left( u\right) =\left\{ \nabla h\left(
u\right) \right\} .$
\end{proof}
\smallskip

\begin{remark}
	\label{rem:lem} Assume $g\left( x\right) :=f\left( x\right) -\langle \bar{z}
	,x\rangle $ possessing a tilt stable local minimum at $\left\{ \bar{x}
	\right\} =m_{f}\left( \bar{z} \right)$ (and hence $\partial _{\operatorname{co}}f\left( \bar{x}\right)  ) \ne \emptyset$).  In \cite[Theorem 3.3]{Lem:1} it is
	observed that the optimality condition applied to the minimization problem
	that defines $L_{\mathcal{U^{\prime }}}^{\varepsilon }\left( u\right)
	=\inf_{v^{\prime }\in \mathcal{V^{\prime }}}\left\{ \operatorname{co}h\left(
	u+v^{\prime }\right) -\langle \bar{z}_{\mathcal{V}^{\prime }},v^{\prime
	}\rangle \right\} $ (which attains its minimum at $v(u)$) gives rise to
	\begin{equation}
	\partial L_{\mathcal{U^{\prime }}}^{\varepsilon }\left( u\right) =\left\{ z_{\mathcal{U}^{\prime }}\mid z_{\mathcal{U}^{\prime }}+\bar{z}_{\mathcal{V^{\prime }}}\in \partial \operatorname{co}h\left( u+v\left( u\right) \right)
	\right\} , \label{eqn:2}
	\end{equation}
	assuming $(u,v(u)) \in \operatorname{int} B_{\varepsilon }^{\mathcal{U^{\prime }}}\left( 0\right)  \times \operatorname{int} B_{\varepsilon }^{\mathcal{V^{\prime }}}\left( 0\right) $. 
	Applying (\ref{prop:6:3}) and Lemma \ref{lem:conv} we have 
	\begin{eqnarray*}
		\partial L_{\mathcal{U^{\prime }}}^{\varepsilon }\left( 0\right) &=&\left\{
		z_{\mathcal{U}}\mid z_{\mathcal{U}^{\prime }}+\bar{z}_{\mathcal{V^{\prime }}}\in \partial \operatorname{co}h\left( 0\right) 
		=\partial _{\operatorname{co}} h \left( 0 \right)  =\partial _{\operatorname{co}}f\left( \bar{x}\right) \right\} \\
		&\subseteq &\left\{ z_{\mathcal{U}^{\prime }}\mid z_{\mathcal{U}^{\prime }}+
		\bar{z}_{\mathcal{V^{\prime }}}\in \partial f\left( \bar{x}\right) \right\}
		=\left\{ \bar{z}_{\mathcal{U}^{\prime }}\right\} \quad \text{as }\mathcal{U}^{\prime }\subseteq \mathcal{U} .
	\end{eqnarray*}
	Thus $\nabla L_{\mathcal{U^{\prime }}}^{\varepsilon }\left( 0\right) =\bar{z}_{\mathcal{U}^{\prime }}$ exists (as was first observed in \cite[Theorem 3.3]
	{Lem:1} for convex functions). Moreover we also have $L_{\mathcal{U^{\prime }}}^{\varepsilon }\left( 0\right) =\inf_{v^{\prime }\in \mathcal{V^{\prime }}}\left\{ \operatorname{co}h\left( v^{\prime }\right) -\langle \bar{z}_{\mathcal{V}^{\prime }},v^{\prime }\rangle \right\} =\operatorname{co}h\left( 0\right) =f\left(
	\bar{x}\right) $ because $m_{h}\left( \bar{z}_{\mathcal{U^{\prime }}}+\bar{z}_{\mathcal{V^{\prime }}}\right) =\left\{ 0\right\} .$ Furthermore, 
	due to the inherent Lipschitz continuity implied by tilt stability (see Proposition \ref{prop:LU}) we must have for $\delta$ sufficiently small 
	$\partial L_{\mathcal{U^{\prime }}}^{\varepsilon }\left( u\right)  \ne \emptyset $ for all 
	$u\in
	B_{\delta}^{\mathcal{U^{\prime }}}\left( 0\right) $.
\end{remark}
\smallskip

Even without the assumption of tilt stability we have the following.

\begin{proposition}
\label{prop:m}Consider $f:\mathbb{R}^{n}\rightarrow \mathbb{R}_{\infty }$ is
a proper lower semi-continuous function and
\begin{equation*}
v\left( u\right) \in \operatorname{argmin}_{v^{\prime }\in \mathcal{V^{\prime }}\cap
B_{\varepsilon }\left( 0\right) }\left\{ f\left( \bar{x}+u+v^{\prime
}\right) -\langle \bar{z}_{\mathcal{V^{\prime }}},v^{\prime }\rangle
\right\} : B_{\varepsilon }^{\mathcal{U^{\prime }}}\left( 0\right)
\rightarrow \mathcal{V^{\prime }}.
\end{equation*}
Then when $z_{\mathcal{U^{\prime }}}\in \partial _{\operatorname{co}} L_{\mathcal{U^{\prime }}}^{\varepsilon } \left( u\right) $
we have for $g\left(
w\right) :=\operatorname{co}h\left( w\right) $ that
\begin{equation}
\left( u,v\left( u\right) \right) \in m_{h}\left( z_{\mathcal{U^{\prime }}}+
\bar{z}_{\mathcal{V^{\prime }}}\right) =\operatorname{argmin}\left\{ g\left(
u+v\right) -\langle z_{\mathcal{U^{\prime }}}+\bar{z}_{\mathcal{V^{\prime }}
},u+v\rangle \right\} \text{ for all }u\in B_{\varepsilon }^{\mathcal{U^{\prime }}}\left( 0\right) .  \label{neqn:11}
\end{equation}
\end{proposition}

\begin{proof}
As $z_{\mathcal{U^{\prime }}}\in \partial _{\operatorname{co}} L_{\mathcal{U^{\prime }}}^{\varepsilon } \left( u\right) $ we have for any $u^{\prime }\in
B_{\varepsilon }^{\mathcal{U^{\prime }}}\left( 0\right) $ that
\begin{eqnarray*}
L_{\mathcal{U^{\prime }}}^{\varepsilon }\left( u^{\prime }\right) &\geq &L_{\mathcal{U^{\prime }}}^{\varepsilon }\left( u\right) +\langle z_{\mathcal{U^{\prime }}},u^{\prime }-u\rangle \\
&=&\inf_{v^{\prime }\in \mathcal{V^{\prime }}}\left\{ h\left( u+v^{\prime
}\right) -\langle \bar{z}_{\mathcal{V^{\prime }}},v^{\prime }\rangle
\right\} +\langle z_{\mathcal{U^{\prime }}},u^{\prime }-u\rangle \\
&=&\left\{ h\left( u+v\left( u\right) \right) -\langle \bar{z}_{\mathcal{V^{\prime }}},v\left( u\right) \rangle \right\} +\langle z_{\mathcal{U^{\prime }}},u^{\prime }-u\rangle \\
&=&h\left( u+v\left( u\right) \right) -\langle z_{\mathcal{U^{\prime }}}+
\bar{z}_{\mathcal{V^{\prime }}},u+v\left( u\right) \rangle +\langle z_{\mathcal{U^{\prime }}},u^{\prime }\rangle .
\end{eqnarray*}
Hence for all $v^{\prime }\in \mathcal{V^{\prime }}$ we have
\begin{eqnarray*}
h\left( u^{\prime }+v^{\prime }\right) -\langle \bar{z}_{\mathcal{V^{\prime }}},v^{\prime }\rangle &\geq &L_{\mathcal{U^{\prime }}}^{\varepsilon }\left(
u^{\prime }\right) \\
&\geq &h\left( u+v\left( u\right) \right) -\langle z_{\mathcal{U^{\prime }}}+
\bar{z}_{\mathcal{V^{\prime }}},u+v\left( u\right) \rangle +\langle z_{\mathcal{U^{\prime }}},u^{\prime }\rangle
\end{eqnarray*}
or for all $\left( u^{\prime },v^{\prime }\right) \in B_{\varepsilon }^{\mathcal{U^{\prime }}}\left( 0\right) \oplus \mathcal{V^{\prime }}$ (using
orthogonality of the spaces), we have
\begin{equation}
h\left( u^{\prime }+v^{\prime }\right) -\langle z_{\mathcal{U^{\prime }}}+
\bar{z}_{\mathcal{V^{\prime }}},u^{\prime }+v^{\prime }\rangle \geq h\left(
u+v\left( u\right) \right) -\langle z_{\mathcal{U^{\prime }}}+\bar{z}_{\mathcal{V^{\prime }}},u+v\left( u\right) \rangle .  \label{neqn:16}
\end{equation}
That is $\left( u,v\left( u\right) \right) \in m_{h}\left( z_{\mathcal{U^{\prime }}}+\bar{z}_{\mathcal{V^{\prime }}}\right) $ and we may now apply
Proposition \ref{prop:co}.
\end{proof}
\smallskip

In the following we repeatedly use the fact that when a function has a
supporting tangent plane to its epigraph one can take the convex closure of
the epigraph and the resultant set will remain entirely to that same side of
that tangent hyperplane. This will be true for partial convexifications as
convex combinations cannot violate the bounding plane. 

\begin{proposition}
	\label{cor:conv}Consider $f:\mathbb{R}^{n}\rightarrow \mathbb{R}_{\infty }$
	is a proper lower semi-continuous function and $v\left( u\right) \in \operatorname{argmin}_{v^{\prime }\in \mathcal{V^{\prime}}\cap B_{\varepsilon }\left(
		0\right) }\left\{ f\left( \bar{x}+u+v^{\prime }\right) -\langle \bar{z}_{\mathcal{V^{\prime}} },v^{\prime }\rangle \right\} .$ Then when $z_{\mathcal{U^{\prime}}}\in \partial _{\operatorname{co}}  L_{\mathcal{U^{\prime}}
	}^{\varepsilon } \left( u\right) $ we have
	\begin{equation*}
	k_{v}^{\ast }\left( z\right) +k_{v}\left( u\right) = \langle z_{\mathcal{U^{\prime}} },u\rangle
	\end{equation*}
	where $k_{v}\left( u\right) :=h\left( u+v\left( u\right) \right) -\langle
	\bar{z}_{\mathcal{V^{\prime}}},u+v\left( u\right) \rangle $ i.e. $z_{\mathcal{U^{\prime}}}\in \partial _{\operatorname{co}}k_{v}\left( u\right) $ and in
	particular $\bar{z}_{\mathcal{U^{\prime}}}\in \partial _{\operatorname{co}
	}k_{v}\left( u\right) =\partial \operatorname{co} k_{v}\left( u\right) $ and $k_{v}\left( u\right) =\operatorname{co}k_{v}\left( u\right) $. Moreover for $u\in
	\mathcal{U^{\prime}}$ we have
	\begin{eqnarray}
	k_{v}\left( u\right) &=&\left[ \operatorname{co}h\right] \left( u+v\left( u\right)
	\right) -\langle \bar{z}_{\mathcal{V^{\prime}}},v\left( u\right) \rangle
	\notag \\
	&=&h\left( u+v\left( u\right) \right) -\langle \bar{z}_{\mathcal{V^{\prime}}
	},v\left( u\right) \rangle =\operatorname{co}k_{v}\left( u\right) ,  \label{neqn:19}
	\end{eqnarray}
	\begin{equation}
	\text{so\quad }h\left( u+v\left( u\right) \right) =\left[ \operatorname{co}h\right]
	\left( u+v\left( u\right) \right) .  \label{neqn:36}
	\end{equation}
\end{proposition}

\begin{proof}
	By (\ref{neqn:16})\ we have
	\begin{equation}
	h\left( u^{\prime }+v^{\prime }\right) -\langle z_{\mathcal{U^{\prime}}}+
	\bar{z}_{\mathcal{V^{\prime}}},u^{\prime }+v^{\prime }\rangle \geq h\left(
	u+v\left( u\right) \right) -\langle z_{\mathcal{U^{\prime}}}+ \bar{z}_{\mathcal{V^{\prime}}},u+v\left( u\right) \rangle .  \label{neqn:17}
	\end{equation}
	So $z_{\mathcal{U^{\prime}}}+\bar{z}_{\mathcal{V^{\prime}}} \in \partial _{\operatorname{co}}h\left( u+v\left( u\right) \right) \neq \emptyset $ and by Lemma \ref{lem:conv} we have $\operatorname{co} h\left( u+v\left( u\right) \right) =h\left(
	u+v\left( u\right) \right) $. Hence
	\begin{eqnarray*}
		h\left( u^{\prime }+v^{\prime }\right) -\langle z_{\mathcal{U^{\prime}}}+
		\bar{z}_{\mathcal{V^{\prime}}},u^{\prime }+v^{\prime }\rangle &\geq &\left[
		\operatorname{co}h\right] \left( u^{\prime }+v^{\prime }\right) -\langle z_{\mathcal{U^{\prime}}}+\bar{z}_{\mathcal{V^{\prime}}},u^{\prime }+ v^{\prime }\rangle
		\\
		&\geq &h\left( u+v\left( u\right) \right) -\langle z_{\mathcal{U^{\prime}}}+
		\bar{z}_{\mathcal{V^{\prime}}},u+v\left( u\right) \rangle
	\end{eqnarray*}
	On placing $v^{\prime }=v\left( u^{\prime }\right) $ we have $h\left(
	u+v\left( u\right) \right) =\left[ \operatorname{co}h\right] \left( u+v\left(
	u\right) \right) $ when $u^{\prime }=u$ and otherwise
	\begin{equation*}
	k_{v}\left( u^{\prime }\right) -\langle z_{\mathcal{U^{\prime}}},u^{\prime
	}+v\left( u^{\prime }\right) \rangle \geq k_{v}\left( u\right) -\langle z_{\mathcal{U^{\prime}} },u+v\left( u\right) \rangle
	\end{equation*}
	or by orthogonality we have for all $u^{\prime }\in \mathcal{U^{\prime}}$
	that
	\begin{equation*}
	k_{v}\left( u^{\prime }\right) -\langle z_{\mathcal{U^{\prime}}},u^{\prime
	}\rangle \geq k_{v}\left( u\right) -\langle z_{\mathcal{U^{\prime}}
	},u\rangle .
	\end{equation*}
	Hence $-k_{v}^{\ast }\left( z_{\mathcal{U^{\prime}}}\right) \geq k_{v}\left(
	u\right) -\langle z_{\mathcal{U^{\prime}}},u\rangle $ implying $\langle z_{\mathcal{U^{\prime}} },u\rangle \geq k_{v}\left( u\right) +k_{v}^{\ast
	}\left( z_{\mathcal{U^{\prime}} }\right) .$ The reverse inequality is supplied by the  Fenchel inequality which gives the
	result $z_{\mathcal{U^{\prime}}}\in \partial _{\operatorname{co}}k_{v}\left(
	u\right) =\partial \operatorname{co}k_{v}\left( u\right) $ and $k_{v}\left( u\right)
	=\operatorname{co}k_{v}\left( u\right) $ follows from Lemma \ref{lem:conv}.
	
	Moreover we have from (\ref{neqn:17}) that
	\begin{eqnarray*}
		h\left( u^{\prime }+v^{\prime }\right) &-&\langle z_{\mathcal{U^{\prime}}}+
		\bar{z}_{\mathcal{V^{\prime}}},u^{\prime }+v^{\prime }\rangle \geq h\left(
		u+v\left( u\right) \right) -\langle z_{\mathcal{U^{\prime}}}+ \bar{z}_{\mathcal{V^{\prime}}},u+v\left( u\right) \rangle \\
		&=&\left[ \operatorname{co}h\right] \left( u+v\left( u\right) \right) - \langle z_{\mathcal{U^{\prime}}}+\bar{z}_{\mathcal{V^{\prime}}}, u+v\left( u\right)
		\rangle \\
		&\geq &\operatorname{co}k_{v}\left( u\right) -\langle z_{\mathcal{U^{\prime}}
		},u\rangle
	\end{eqnarray*}
	and hence (using orthogonality)
	\begin{eqnarray*}
		\left[ \operatorname{co}h\right] \left( u^{\prime }+v^{\prime }\right) &-& \langle
		z_{\mathcal{U^{\prime}}}+\bar{z}_{\mathcal{V^{\prime}}},u^{\prime
		}+v^{\prime }\rangle \geq k_{v}\left( u\right) -\langle z_{\mathcal{U^{\prime}}},u\rangle \\
		&=&\left\{ \left[ \operatorname{co}h\right] \left( u+v\left( u\right) \right)
		-\langle \bar{z}_{\mathcal{V^{\prime}}},v\left( u\right) \rangle \right\}
		-\langle z_{\mathcal{U^{\prime}}},u\rangle \\
		&\geq &\operatorname{co}k_{v}\left( u\right) -\langle z_{\mathcal{U^{\prime}}
		},u\rangle .
	\end{eqnarray*}
	On placing $v^{\prime }=v\left( u^{\prime }\right) $ we have
	\begin{eqnarray*}
		\left[ \operatorname{co}h\right] \left( u^{\prime }+v\left( u^{\prime }\right)
		\right) &-&\langle z_{\mathcal{U^{\prime}}}+\bar{z}_{\mathcal{V^{\prime}}
		},u^{\prime }+v\left( u^{\prime }\right) \rangle \\
		&\geq &\left[ \operatorname{co}h\right] \left( u +v\left( u 
		\right) \right) -\langle \bar{z}_{\mathcal{V^{\prime}}},v\left( u
		\right) \rangle  -\langle z_{\mathcal{U^{\prime}}
		},u\rangle \\
		&\geq &k_{v}\left( u\right) -\langle z_{\mathcal{U^{\prime}}},u\rangle \geq
		\operatorname{co} k_{v}\left( u\right) -\langle z_{\mathcal{U^{\prime}}},u\rangle .
	\end{eqnarray*}
	and $u^{\prime }=u$ and using the identities $k_{v}\left( u\right) =\operatorname{co}
	k_{v}\left( u\right) $ and $\left[ \operatorname{co}h\right] \left( u+v\left(
	u\right) \right) =h\left( u+v\left( u\right) \right) $ for $u\in \mathcal{U^{\prime}}$ we have (\ref{neqn:19}).
\end{proof}

\subsection{\protect\smallskip Subhessians and the localised $\mathcal{U}^{\prime }$-Lagrangian}

Now that we have some theory of the localised $\mathcal{U}^{\prime }$-Lagrangian we may study its interaction with the notion of subhessian. As
we will be applying these results locally around a tilt stable local minimum
we are going to focus on the case when we have $L_{\mathcal{U}}^{\varepsilon }\left( u\right) =\inf_{v\in \mathcal{V}}\left\{ \operatorname{co}h\left( u+v\right) -\langle \bar{z}_{\mathcal{V}},v\rangle \right\} $ and $\mathcal{U}^{\prime
} = \mathcal{U}$. The following is a small variant of \cite[Corollary
3.5]{Lem:1}.

\begin{lemma}
\label{lem:littleOh}Suppose $f:\mathbb{R}^{n}\rightarrow \mathbb{R}_{\infty
} $ is quadratically minorised and is prox--regular at $\bar{x}\ $ for $\bar{z}\in \partial f(\bar{x})$ with respect to $\varepsilon $ and $r.$ Suppose
in addition that $f-\langle \bar{z},\cdot \rangle $ possesses a tilt stable
local minimum at $\bar{x}$, where $\bar{z}\in \operatorname{rel}$-$\operatorname{int}\partial f(\bar{x})$, $\mathcal{U}^{\prime }=\mathcal{U}$, $v\left( u\right) \in \operatorname{argmin}_{v\in \mathcal{V}\cap B_{\varepsilon }\left( 0\right) }
\left[ f\left( \bar{x}+u+v\right) -\langle \bar{z}_{\mathcal{V}},v\rangle
\right] $ and $L_{\mathcal{U}}^{\varepsilon }\left( u\right) =\inf_{v\in
\mathcal{V}}\left\{ \operatorname{co}h\left( u+v\right) -\langle \bar{z}_{\mathcal{V}
},v\rangle \right\} $. Then we have $v\left( u\right) =o\left( \left\Vert
u\right\Vert \right) $ in the following sense:
\begin{equation*}
\forall \varepsilon^{\prime\prime} >0,\quad \exists \delta >0:\quad \left\Vert u\right\Vert
\leq \delta \quad \implies \quad \left\Vert v\left( u\right) \right\Vert
\leq \varepsilon^{\prime\prime}  \left\Vert u\right\Vert .
\end{equation*}
\end{lemma}

\begin{proof}
As noted in Remark \ref{rem:lem} we have $\nabla L_{\mathcal{U}}^{\varepsilon }\left( 0\right) =\bar{z}_{\mathcal{U}}$ existing where $L_{\mathcal{U}}^{\varepsilon }\left( u\right) =\inf_{v\in \mathcal{V}}\left\{
\operatorname{co}h\left( u+v\right) -\langle \bar{z}_{\mathcal{V}},v\rangle \right\}
$ is a convex function finite locally around $u=0.$ Consequently we have
for $u\in B^{\mathcal{U}}_{\varepsilon }\left( 0\right) $ we have
\begin{equation*}
L_{\mathcal{U}}^{\varepsilon }\left( u\right) =L_{\mathcal{U}}^{\varepsilon
}\left( 0\right) +\langle \nabla L_{\mathcal{U}}^{\varepsilon }\left(
0\right) ,u\rangle +o\left( \left\Vert u\right\Vert \right) =f\left( \bar{x}
\right) +\langle \bar{z}_{\mathcal{U}},u\rangle +o\left( \left\Vert
u\right\Vert \right) .
\end{equation*}
Invoking (\ref{neqn:14}) in the proof of Lemma \ref{lem:sharp} we have an $\varepsilon ^{\prime }>0$ such that for $u\in B_{\varepsilon ^{\prime
}}\left( 0\right) \cap \mathcal{U}$
\begin{equation*}
f\left( \bar{x}+u+v\right) \geq f\left( \bar{x}\right) +\langle \bar{z}_{\mathcal{V}},v\rangle +\langle \bar{z}_{\mathcal{U}},u\rangle +\left(
\varepsilon ^{\prime }-\frac{r\left\Vert v\right\Vert }{2}\right) \left\Vert
v\right\Vert -\frac{r}{2}\left\Vert u\right\Vert ^{2}\quad \text{for all }
v\in \varepsilon ^{\prime }B_{1}\left( 0\right) \cap \mathcal{V}\text{.}
\end{equation*}
When we choose $v\in B_{\min \{\varepsilon ^{\prime },\varepsilon ^{\prime
}/r\}}\left( 0\right) $ we have
\begin{equation*}
f\left( \bar{x}+u+v\right) \geq f\left( \bar{x}\right) +\langle \bar{z}_{\mathcal{U}} + \bar{z}_{\mathcal{V}}, u + v\rangle  +\frac{\varepsilon ^{\prime }}{2}\left\Vert v\right\Vert -\frac{r}{2}\left\Vert
u\right\Vert ^{2}.
\end{equation*}
As $v(\cdot )$ is Lipschitz continuous with $v(0)=0$ there exists $\delta'>0$ such that
$\|u \| < \delta'$ implies $v(u) \in  B_{\min \{\varepsilon ^{\prime },\varepsilon ^{\prime
	}/r\}}\left( 0\right) $. Then we have 
\begin{eqnarray*}
f\left( \bar{x}\right) +\langle \bar{z}_{\mathcal{U} },u\rangle
+o\left( \left\Vert u\right\Vert \right) &=&L_{\mathcal{U }}^{\varepsilon }\left( u\right) =f\left( \bar{x}+u+v\left( u\right) \right)
-\langle \bar{z}_{\mathcal{V} },v\left( u\right) \rangle \\
&\geq &f\left( \bar{x}\right) +\langle \bar{z},u+v\left( u\right) \rangle
-\langle \bar{z}_{\mathcal{V} },v\left( u\right) \rangle +\frac{\varepsilon ^{\prime }}{2}\left\Vert v\left( u\right) \right\Vert -\frac{r}{2}\left\Vert u\right\Vert ^{2} \\
&=&f\left( \bar{x}\right) +\langle \bar{z}_{\mathcal{U} },u\rangle +
\frac{\varepsilon ^{\prime }}{2}\left\Vert v\left( u\right) \right\Vert -
\frac{r}{2}\left\Vert u\right\Vert ^{2}.
\end{eqnarray*}
Hence
\begin{equation*}
\frac{2}{\varepsilon ^{\prime }}\left[ o\left( \left\Vert u\right\Vert
\right) +\frac{r}{2}\left\Vert u\right\Vert ^{2}\right] \geq \left\Vert
v\left( u\right) \right\Vert
\end{equation*}
and given any $\varepsilon^{\prime\prime} >0$ we choose $\delta >0$ with  $\delta \leq \delta'$ 
such that $\left\Vert
u\right\Vert \leq \delta $ implies $\frac{2}{\varepsilon ^{\prime }}\left[
\frac{o\left( \left\Vert u\right\Vert \right) }{\left\Vert u\right\Vert }+
\frac{r}{2}\left\Vert u\right\Vert \right] \leq \min \left\{ \varepsilon^{\prime\prime} 
,\varepsilon ^{\prime }/r,\varepsilon ^{\prime }\right\} .$
\end{proof}

\bigskip We may now further justify our definition of "fast track" at $\bar{x}.$ In \cite{Hare:2} and other works "fast tracks" are specified as a
subspace on which both $u\mapsto L_{\mathcal{U}}^{\varepsilon }\left(
u\right) $ and $u\mapsto v\left( u\right) $ are twice continuously
differentiable. In particular $L_{\mathcal{U}}^{\varepsilon }\left( \cdot
\right) $ admits a Taylor expansion.

\begin{proposition}
\label{prop:Lagsubjet}Suppose $f:\mathbb{R}^{n}\rightarrow \mathbb{R}_{\infty }$ is quadratically minorised and is prox--regular at $\bar{x}\ $
for $\bar{z}\in \partial f(\bar{x})$ with respect to $\varepsilon $ and $r.$
Suppose in addition that $f-\langle \bar{z},\cdot \rangle $ possesses a
tilt stable local minimum at $\bar{x}$, $\bar{z}+B_{\varepsilon }\left(
0\right) \cap \mathcal{V}\subseteq \partial f(\bar{x})$, $\mathcal{U}^{\prime }=\mathcal{U}$, $L_{\mathcal{U}}^{\varepsilon }\left( u\right)
=\inf_{v\in \mathcal{V}}\left\{ \operatorname{co}h\left( u+v\right) -\langle \bar{z}_{\mathcal{V}},v\rangle \right\} $ and $\left\{ v\left( u\right) \right\} =
\operatorname{argmin}_{v\in \mathcal{V}\cap B_{\varepsilon }\left( 0\right) }\left[
f\left( \bar{x}+u+v\right) -\langle \bar{z}_{\mathcal{V}},v\rangle \right] $ for $u \in B^{\mathcal{U}}_{\varepsilon} (0)$. Then for $\bar{z}_{\mathcal{U}}=\nabla L_{\mathcal{U}}^{\varepsilon 
}\left( 0\right) $ we have
\begin{equation*}
\left( L_{\mathcal{U}}^{\varepsilon }\right) _{\_}^{\prime \prime }\left( 0,\bar{z}_{\mathcal{U}},w\right) =\left( \operatorname{co}h\right) _{\_}^{\prime
\prime }\left( 0+v\left( 0\right) ,\left( \bar{z}_{\mathcal{U}},\bar{z}_{\mathcal{V}}\right) ,w\right) =f_{\_}^{\prime \prime }\left( \bar{x},\left(
\bar{z}_{\mathcal{U}},\bar{z}_{\mathcal{V}}\right) ,w\right) \quad \text{for
all }w\in \mathcal{U}\text{. }
\end{equation*}
\end{proposition}

\begin{proof}
Consider the second order quotient
\begin{eqnarray*}
&&\frac{2}{t^{2}}\left[ L_{\mathcal{U}}^{\varepsilon }\left( 0+tu\right) -L_{\mathcal{U}}^{\varepsilon }\left( 0\right) -t\langle \bar{z}_{\mathcal{U}},u\rangle \right] \\
&&\qquad =\frac{2}{t^{2}}\left[ f\left( \bar{x}+tu+v\left( tu\right) \right)
-f\left( \bar{x}\right) -\langle \bar{z}_{\mathcal{U}}+\bar{z}_{\mathcal{V}},tu+v\left( tu\right) \rangle \right] \\
&&\qquad =\frac{2}{t^{2}}\left[ f\left( \bar{x}+t\left[ u+\frac{v\left(
tu\right) }{t}\right] \right) -f\left( \bar{x}\right) -t\langle \bar{z}_{\mathcal{U}}+\bar{z}_{\mathcal{V}},u+\frac{v\left( tu\right) }{t}\rangle
\right] .
\end{eqnarray*}
Apply Lemma \ref{lem:littleOh},  for an arbitrary $\delta >0$ by taking $\varepsilon^{\prime\prime} = 
\frac{\delta}{\|u\|}$ and obtaining the existence of $\gamma >0$ such that 
for $\| t u\| \leq t[ \delta + \|w\|] \leq \gamma$ we have $\|v (t u ) \| \leq \varepsilon^{\prime\prime} \|t u \|$. This implies for all $\delta > 0$ that when $t[ \delta + \|w\|] \leq \gamma$  and $u \in B_{\delta} (w )$ 
we have $\frac{\| v\left( tu\right)  \| }{t} \leq \delta$ 
  and so $\frac{v\left( tu\right) }{t}\in B_{\delta }\left( 0\right) $. Thus for $t<\eta := \frac{\gamma}{\delta + \|w\|}$ we
have
\begin{eqnarray*}
&&\inf_{u\in B_{\delta }\left( w\right) \cap \mathcal{U}}\frac{2}{t^{2}}
\left[ L_{\mathcal{U}}^{\varepsilon }\left( 0+tu\right) -L_{\mathcal{U}}^{\varepsilon }\left( 0\right) -t\langle \bar{z}_{\mathcal{U}},u\rangle
\right] \\
&&\qquad \geq \inf_{u\in B_{\delta }\left( w\right) \cap \mathcal{U}
}\inf_{v\in B_{\delta }\left( 0\right) \cap \mathcal{V}}\frac{2}{t^{2}}\left[
f\left( \bar{x}+t\left[ u+v\right] \right) -f\left( \bar{x}\right) -t\langle
\bar{z}_{\mathcal{U}}+\bar{z}_{\mathcal{V}},u+v\rangle \right] \\
&&\qquad \geq \inf_{h\in B_{\delta }\left( w,0\right) }\frac{2}{t^{2}}\left[
f\left( \bar{x}+th\right) -f\left( \bar{x}\right) -t\langle \bar{z}_{\mathcal{U}}+\bar{z}_{\mathcal{V}},h\rangle \right]
\end{eqnarray*}
and so
\begin{eqnarray*}
&&\liminf_{t\downarrow 0}\inf_{u\in B_{\delta }\left( w\right) \cap \mathcal{U}}\frac{2}{t^{2}}\left[ L_{\mathcal{U}}^{\varepsilon }\left( 0+tu\right)
-L_{\mathcal{U}}^{\varepsilon }\left( 0\right) -t\langle \bar{z}_{\mathcal{U}
},u\rangle \right] \\
&&\qquad \geq \liminf_{t\downarrow 0}\inf_{h\in B_{\delta }\left( w,0\right)
}\frac{2}{t^{2}}\left[ f\left( \bar{x}+th\right) -f\left( \bar{x}\right)
-t\langle \bar{z}_{\mathcal{U}}+\bar{z}_{\mathcal{V}},h\rangle \right] .
\end{eqnarray*}
Taking the infimum over $\delta >0$ gives $\left( L_{\mathcal{U}}^{\varepsilon }\right) _{\_}^{\prime \prime }\left( 0,\bar{z}_{\mathcal{U}},w\right) \geq f_{\_}^{\prime \prime }\left( \bar{x},\left( \bar{z}_{\mathcal{U}},\bar{z}_{\mathcal{V}}\right) ,w\right) =\left( \operatorname{co}
h\right) _{\_}^{\prime \prime }\left( 0+v\left( 0\right) ,\left( \bar{z}_{\mathcal{U}},\bar{z}_{\mathcal{V}}\right) ,w\right) $ (because $f\left( \bar{x}+\cdot \right) $ and $\operatorname{co}h\left( \cdot \right) $ agree locally).
Conversely consider
\begin{eqnarray*}
&&\inf_{u\in B_{\delta }\left( w\right) \cap \mathcal{U}}\inf_{\ v\in
\mathcal{V\cap }B_{\delta }\left( 0\right) }\frac{2}{t^{2}}\left[ \operatorname{co}
h\left( 0+t\left[ u+v\right] \right) -\operatorname{co}h\left( 0\right) -t\langle
\bar{z}_{\mathcal{U}}+\bar{z}_{\mathcal{V}},u+v\rangle \right] \\
&\geq &\inf_{u\in B_{\delta }\left( w\right) \cap \mathcal{U}}\inf_{\ v\in
\mathcal{V}}\frac{2}{t^{2}}\left[ \operatorname{co}h\left( 0+\left[ tu+v\right]
\right) -\operatorname{co}h\left( 0\right) -\langle \bar{z}_{\mathcal{U}}+\bar{z}_{\mathcal{V}},tu+v\rangle \right] \\
&=&\inf_{u\in B_{\delta }\left( w\right) \cap \mathcal{U}}\left[ \frac{2}{t^{2}}\left[ \inf_{v\in \mathcal{V}}\left\{ \operatorname{co}h\left( tu+v\right)
-\langle \bar{z}_{\mathcal{V}},v\rangle \right\} \right] -L_{\mathcal{U}}^{\varepsilon }\left( 0\right) -t\langle \bar{z}_{\mathcal{U}},u\rangle
\right] \\
&=&\inf_{u\in B_{\delta }\left( w\right) \cap \mathcal{U}}\frac{2}{t^{2}}\left[ L_{\mathcal{U}}^{\varepsilon }\left( 0+tu\right) -L_{\mathcal{U}}^{\varepsilon }\left( 0\right) -t\langle \bar{z}_{\mathcal{U}},u\rangle
\right]
\end{eqnarray*}
and on taking a limit infimum as $t\downarrow 0$ and then an infimum over $\delta >0$ gives $\left( \operatorname{co}h\right) _{\_}^{\prime \prime }\left(
0+v\left( 0\right) ,\left( \bar{z}_{\mathcal{U}},\bar{z}_{\mathcal{V}}\right) ,w\right) \geq \left( L_{\mathcal{U}}^{\varepsilon }\right)
_{\_}^{\prime \prime }\left( 0,\bar{z}_{\mathcal{U}},w\right) $ and thus
equality.
\end{proof}

Denote $\partial _{\mathcal{U^{\prime }}}^{2,-}\left( \operatorname{co}h\right)
\left( x,z\right) =P_{\mathcal{U^{\prime }}}^{T}\partial ^{2,-}\left( \operatorname{co}h\right) \left( x,z\right) P_{\mathcal{U^{\prime }}}$.

\begin{corollary}
Posit the assumption of Proposition \ref{prop:Lagsubjet}. Then we have
\begin{equation*}
\operatorname{dom}\left( L_{\mathcal{U}}^{\varepsilon }\right) _{\_}^{\prime \prime
}\left( 0,\bar{z}_{\mathcal{U}},\cdot \right) \ \subseteq \operatorname{dom}
f_{\_}^{\prime \prime }\left( \bar{x},\left( \bar{z}_{\mathcal{U}},\bar{z}_{\mathcal{V}}\right) ,\cdot \right)
\end{equation*}
and $\partial _{\mathcal{U}}^{2,-}\left( \operatorname{co}h\right) \left( \bar{x},
\bar{z}\right) \subseteq \partial ^{2,-}L_{\mathcal{U}}^{\varepsilon }\left(
0,\bar{z}_{\mathcal{U}}\right) $.
\end{corollary}

\begin{proof}
As $\operatorname{dom}\left( \operatorname{co}h\right) \subseteq \mathcal{U}$ we have $\operatorname{dom}\left( L_{\mathcal{U}}^{\varepsilon }\right) _{\_}^{\prime \prime
}\left( 0,\bar{z}_{\mathcal{U}},\cdot \right) \subseteq  \mathcal{U}$ and
hence by Proposition \ref{prop:Lagsubjet} we have $\operatorname{dom}\left( L_{\mathcal{U}}^{\varepsilon }\right) _{\_}^{\prime
\prime }\left( 0,\bar{z}_{\mathcal{U}},\cdot \right) \ \subseteq \operatorname{dom}f_{\_}^{\prime \prime }\left( \bar{x},\left( \bar{z}_{\mathcal{U}},\bar{z}_{\mathcal{V}}\right) ,\cdot \right) =\operatorname{dom}\left( \operatorname{co}h\right)
_{\_}^{\prime \prime }\left( 0,\left( \bar{z}_{\mathcal{U}},\bar{z}_{\mathcal{V}}\right) ,\cdot \right) $. Now take $Q\in \partial ^{2,-}\left(
\operatorname{co}h\right) \left( \bar{x},\bar{z}\right) $ and so we have $\langle
Qu,u\rangle \leq \left( \operatorname{co}h\right) _{\_}^{\prime \prime }\left(
0,\left( \bar{z}_{\mathcal{U}},\bar{z}_{\mathcal{V}}\right) ,u\right) $ for
all $u\in \operatorname{dom}\left( \operatorname{co}h\right) _{\_}^{\prime \prime }\left(
0,\left( \bar{z}_{\mathcal{U}},\bar{z}_{\mathcal{V}}\right) ,\cdot \right) $. Hence for all $u\in \operatorname{dom}\left( L_{\mathcal{U}}^{\varepsilon }\right)
_{\_}^{\prime \prime }\left( 0,\bar{z}_{\mathcal{U}},\cdot \right) \subseteq
\operatorname{dom}\left( \operatorname{co}h\right) _{\_}^{\prime \prime }\left( 0,\left(
\bar{z}_{\mathcal{U}},\bar{z}_{\mathcal{V}}\right) ,\cdot \right) \cap
\mathcal{U}$ we have
\begin{equation*}
\langle \left[ P_{\mathcal{U}}^{T}QP_{\mathcal{U}}\right] u,u\rangle
=\langle QP_{\mathcal{U}}u,P_{\mathcal{U}}u\rangle \leq \left( \operatorname{co}
h\right) _{\_}^{\prime \prime }\left( 0,\left( \bar{z}_{\mathcal{U}},\bar{z}_{\mathcal{V}}\right) ,u\right) =\left( L_{\mathcal{U}}^{\varepsilon
}\right) _{\_}^{\prime \prime }\left( 0,\bar{z}_{\mathcal{U}},u\right)
\end{equation*}
and so $P_{\mathcal{U}}^{T}QP_{\mathcal{U}}\in \partial ^{2,-}L_{\mathcal{U}}^{\varepsilon }\left( 0,\bar{z}_{\mathcal{U}}\right) .$ That is $\partial _{\mathcal{U}}^{2,-}\left( \operatorname{co}h\right) \left( \bar{x},\bar{z}\right) =P_{\mathcal{U}}^{T}\partial ^{2,-}\left( \operatorname{co}h\right) \left( \bar{x},\bar{z}\right) P_{\mathcal{U}}\subseteq \partial ^{2,-}L_{\mathcal{U}}^{\varepsilon }\left( 0,\bar{z}_{\mathcal{U}}\right) .$
\end{proof}

If we assume more (which is very similar to the "Partial Smoothness" of \cite{Lewis:2}) we obtain the following which can be viewed as a less stringent
version of the second order expansions studied in \cite[Theorem 3.9]{Lem:1},
\cite[Equation (7)]{Mifflin:2004:2} and \cite[Theorem 2.6]{Miller:1}. This
result suggests that the role of assumptions like that of Proposition \ref{prop:reg} part \ref{part:4}, which are also a consequence of the definition of
partial smoothness (via the continuity of the $w\mapsto \partial f(w)$ at $x$
relative to $\mathcal{M}$) could be to build a bridge to the identity $\mathcal{U}=\mathcal{U}^{2}$ (see the discussion in Remark \ref{rem:fasttrack} below).

\begin{corollary}
\label{cor:Lagsubjet}Posit the assumption of Proposition \ref{prop:Lagsubjet}
and assume the assumption of Proposition \ref{prop:reg} part \ref{part:4}
i.e. suppose we have $\varepsilon >0$ such that for
all $z_{\mathcal{V}}\in B_{\varepsilon }\left( \bar{z}_{\mathcal{V}}\right)
\cap \mathcal{V}\subseteq \partial _{\mathcal{V}}f\left( \bar{x}\right) $
there is a common
\begin{equation}
v\left( u\right) \in \operatorname{argmin}_{v\in \mathcal{V}\cap B_{\varepsilon
}\left( 0\right) }\left\{ f\left( \bar{x}+u+v\right) -\langle z_{\mathcal{V}},v\rangle \right\}
\cap \operatorname{int} B_{\varepsilon} (0) \label{eqn:3}
\end{equation}
for all $u\in B^{\mathcal{U}}_{\varepsilon }\left( 0\right) $. In addition suppose  there
exists $\varepsilon >0$ such that for all $u\in B^{\mathcal{U}}_{\varepsilon }\left(
0\right) $ we have $u \mapsto \nabla L_{\mathcal{U}}^{\varepsilon
}\left( u\right) :=z_{\mathcal{U}}(u)$ existing, and is a continuous function. Then 
\begin{equation*}
\left( L_{\mathcal{U}}^{\varepsilon }\right) _{\_}^{\prime \prime }\left(
u,z_{\mathcal{U}},w\right) =\left( \operatorname{co}h\right) _{\_}^{\prime \prime
}\left( u+v\left( u\right) ,\left( z_{\mathcal{U}},\bar{z}_{\mathcal{V}
}\right) ,w\right) \quad \text{for all }w\in \mathcal{U}\text{.} 
\end{equation*}
\end{corollary}

\begin{proof}
All the assumption of Proposition \ref{prop:Lagsubjet} are local in nature
except for the assumption that $\bar{z}+B_{\varepsilon }\left( 0\right) \cap
\mathcal{V}\subseteq \partial f(\bar{x}).$ Discounting this assumption for
now we note that we can perturb $\bar{x}$ (to $\bar{x}+u+v\left( u\right) $)
and $\bar{z}$ (to $\left( z_{\mathcal{U}},\bar{z}_{\mathcal{V}}\right) \in
\partial \left( \operatorname{co}h\right) \left( u+v\left( u\right) \right) $)
within a sufficiently small neighbourhood $u\in B^{\mathcal{U}}_{\varepsilon }\left(
0\right)$ and still have the assumption of prox-regularity,
tilt stability (around a different minimizer of our tilted function) and still use the same selection function $v\left( \cdot
\right) $. Regarding this outstanding assumption the optimality conditions 
associated with (\ref{eqn:3}) imply
 $\bar{z}_{\mathcal{V}} +B_{\varepsilon }\left(
0\right) \cap \mathcal{V}\subseteq \partial_{\mathcal{V}} f(\bar{x}+u+v\left( u\right) )$. 
As we have $\nabla L_{\mathcal{U}}^{\varepsilon
}\left( u\right) =z_{\mathcal{U}}(u)$ existing from (\ref{eqn:2}) that 
$ \partial \operatorname{co}h ( u+v\left( u\right) ) - ({z}_{\mathcal{U}}(u) , \bar{z}_{\mathcal{V}}) 
\subseteq \{0\} \oplus \mathcal{V}$ and so   
$ \partial \operatorname{co}h ( u+v\left( u\right) ) = \left\{ {z}_{\mathcal{U}}(u)\right\}
\oplus \partial _{\mathcal{V}}\operatorname{co}h( u+v\left( u\right) )   =  \left\{ {z}_{\mathcal{U}} (u)\right\}
\oplus \partial _{\mathcal{V}}f\left( \bar{x} + u+v\left( u\right)  \right) . $ Hence 
\[
 ({z}_{\mathcal{U}} (u), \bar{z}_{\mathcal{V}})  + \{0\} \oplus B^{\mathcal{V}}_{\varepsilon }\left(
 0\right) \subseteq  \partial \operatorname{co}h ( u+v\left( u\right) ) 
 =  \partial f( \bar{x} + u+v\left( u\right) ). 
\]
This furnishes the final assumption that is required to invoke Proposition
\ref{prop:Lagsubjet} at points near to $\left( \bar{x},\bar{z}\right) $.
\end{proof}

\begin{remark}
\label{rem:fasttrack}We omit the details here, as they are not central to this
paper, but one may invoke \cite[Corollary 3.3]{eberhard:6} to deduce from
Corollary \ref{cor:Lagsubjet} that
\begin{equation*}
q\left( \underline{\partial }^{2}L_{\mathcal{U}}^{\varepsilon }\left( 0,\bar{z}_{\mathcal{U}}\right) \right) \left( w\right) =q\left( \underline{\partial
}^{2}\left( \operatorname{co}h|_{\mathcal{U}}\right) \left( 0,\bar{z}\right) \right)
\left( w\right) \quad \text{for all }w\in \mathcal{U}\text{.}
\end{equation*}
It follows that $\mathcal{U}\supseteq \operatorname{dom}q\left( \underline{\partial }
^{2}L_{\mathcal{U}}^{\varepsilon }\left( 0,\bar{z}_{\mathcal{U}}\right)
\right) \left( \cdot \right) \ =\operatorname{dom}q(\underline{\partial }^{2}\left(
\operatorname{co}h|_{\mathcal{U}}\right) (0,\bar{z}))\left( \cdot \right) \cap
\mathcal{U}=\mathcal{U}^{2}\cap \mathcal{U}$. Then the imposition of the
equality $\mathcal{U}=\mathcal{U}^{2}\mathcal{\ }$(our definition of a fast
track) implies $\underline{\partial }^{2}L_{\mathcal{U}}^{\varepsilon
}\left( 0,\bar{z}_{\mathcal{U}}\right) =\underline{\partial }^{2}\left(
\operatorname{co}h|_{\mathcal{U}}\right) (0,\bar{z})=\underline{\partial }^{2}\left(
f|_{\mathcal{U}}\right) (\bar{x},\bar{z}).$ It would be enlightening to have
a result that establishes this identity without a-priori assuming  $\mathcal{U}=\mathcal{U}^{2}$ but having this as a consequence.
\end{remark}

Recall that for a convex function $\left( \operatorname{co}h\right) ^{\ast }$ its
subjet is nonempty at every point at which it is subdifferentiable. The following result 
will be required later in the paper.

\begin{lemma}
Suppose $f:\mathbb{R}^{n}\rightarrow \mathbb{R}_{\infty }$ is a proper lower
semi-continuous function possessing a
tilt stable local minimum at $\bar{x}$. Suppose in addition  that $\bar{z}=0\in \operatorname{rel}$-$\operatorname{int}\partial f\left( \bar{x}\right) $, $\mathcal{U^{\prime }\subseteq U}$, $$v\left( u\right) \in \operatorname{argmin}_{v^{\prime }\in \mathcal{V^{\prime }}\cap B_{\varepsilon }\left( 0\right) }f\left( \bar{x}+u+v^{\prime }\right) $$ 
and $u:=P_{\mathcal{U^{\prime }}}\left[ m_f \left( z_{\mathcal{U^{\prime }}}+
\bar{z}_{\mathcal{V^{\prime }}}\right) \right] .$

\begin{enumerate}
\item Then $z_{\mathcal{U^{\prime }}}+0_{\mathcal{V}}\in \partial \left(
\operatorname{co}h\right) \left( u+v\left( u\right) \right) =\partial _{\operatorname{co}}h\left( u+v\left( u\right) \right) $ and consequently $z_{\mathcal{U^{\prime }}}\in \partial L_{\mathcal{U^{\prime }}}^{\varepsilon }\left(
u\right) =\partial k_{v}\left( u\right) .$

\item Suppose $\left( \left( u+v\left( u\right) \right) ,Q\right) \in
\partial ^{2,-}\left( \operatorname{co}h\right) ^{\ast }\left( z_{\mathcal{U^{\prime
}}}+0_{\mathcal{V^{\prime }}}\right) $ then $P_{\mathcal{U^{\prime }}}^{T}QP_{\mathcal{U^{\prime }}}\in \partial ^{2,-}\left( \operatorname{co}h\right)
^{\ast }|_{\mathcal{U^{\prime }}}\left( z_{\mathcal{U^{\prime }}}\right)
=\partial ^{2,-}k_{v}^{\ast }\left( z_{\mathcal{U^{\prime }}}\right) .$
Consequently when $\nabla ^{2}k_{v}^{\ast }\left( z_{\mathcal{U^{\prime }}
}\right) $ exists
\begin{equation}
\partial _{\mathcal{U^{\prime }}}^{2,-}\left( \operatorname{co}h\right) ^{\ast
}\left( u+v\left( u\right) ,z_{\mathcal{U^{\prime }}}+0_{\mathcal{V^{\prime }
}}\right) :=P_{\mathcal{U^{\prime }}}^{T}\partial ^{2,-}\left( \operatorname{co}
h\right) ^{\ast }\left( z_{\mathcal{U^{\prime }}}+0_{\mathcal{V^{\prime }}
}\right) P_{\mathcal{U^{\prime }}}=\nabla ^{2}k_{v}^{\ast }\left( z_{\mathcal{U^{\prime }}}\right) -\mathcal{P}(\mathcal{U}^{\prime}).  \label{neqn:13}
\end{equation}
\end{enumerate}
\end{lemma}

\begin{proof}
Note that as $\bar{z}=0$ by (\ref{neqn:33}) we have $k_{v}^{\ast }\left( z_{\mathcal{U^{\prime }}}\right) =h^{\ast }\left( z_{\mathcal{U^{\prime }}}+0_{\mathcal{V^{\prime }}}\right) =(\operatorname{co}h)^{\ast }\left(
z_{\mathcal{U^{\prime }}}+0_{\mathcal{V^{\prime }}}\right) $. Invoking Proposition \ref{cor:conv} we have  $k_{v}\left( u\right) :=h\left( u+v\left( u\right) \right) =\operatorname{co}h\left(
u+v\left( u\right) \right) $, and by Lemma \ref{lem:27} we then have 
\begin{equation*}
k_{v}\left( u\right) +k_{v}^{\ast }\left( z_{\mathcal{U^{\prime }}}\right)
=\langle z_{\mathcal{U^{\prime }}},u\rangle =\operatorname{co}h\left( u+v\left(
u\right) \right) +h^{\ast }\left( z_{\mathcal{U^{\prime }}}+0_{\mathcal{V^{\prime }}}\right)
\end{equation*}
and hence $z_{\mathcal{U^{\prime }}}+0_{\mathcal{V^{\prime }}}\in \partial
(\operatorname{co} h)\left( u+v\left( u\right) \right) =\partial _{\operatorname{co}}h\left(
u+v\left( u\right) \right) $. Taking into account Remark \ref{rem:lem} we
have for all $u\in \mathcal{U^{\prime }\cap B}_{\varepsilon }\left( 0\right)
$ that
\begin{equation*}
z_{\mathcal{U^{\prime }}}+0_{\mathcal{V^{\prime }}}\in \partial \operatorname{co}
h\left( u+v\left( u\right) \right) \quad \iff \quad z_{\mathcal{U^{\prime }}}\in \partial L_{\mathcal{U^{\prime }}}^{\varepsilon }\left( u\right)
=\partial k_{v}\left( u\right) .
\end{equation*}

For the second part we have from the definition of $\left( \left( u+v\left(
u\right) \right) ,Q\right) \in \partial ^{2,-}\left( \operatorname{co}h\right)
^{\ast }\left( z_{\mathcal{U^{\prime }}}+0_{\mathcal{V^{\prime }}}\right) $
locally around $z_{\mathcal{U^{\prime }}}+0_{\mathcal{V^{\prime }}}$ that
\begin{eqnarray*}
\left( \operatorname{co}h\right) ^{\ast }\left( y\right) &\geq &\left( \operatorname{co}
h\right) ^{\ast }\left( z_{\mathcal{U^{\prime }}}+0_{\mathcal{V^{\prime }}}\right) +\langle u+v\left( u\right) ,z_{\mathcal{U^{\prime }}}+0_{\mathcal{V^{\prime }}}\rangle \\
&&\qquad +\frac{1}{2}\langle Q\left( y-\left( z_{\mathcal{U^{\prime }}}+0_{\mathcal{V^{\prime }}}\right) \right) ,\left( y-\left( z_{\mathcal{U^{\prime
}}}+0_{\mathcal{V^{\prime }}}\right) \right) \rangle +o\left( \left\Vert
y-\left( z_{\mathcal{U^{\prime }}}+0_{\mathcal{V^{\prime }}}\right)
\right\Vert ^{2}\right) .
\end{eqnarray*}
Restricting to $\mathcal{U^{\prime }}$ we have the following locally around $z_{\mathcal{U^{\prime }}}$
\begin{eqnarray*}
\left( \operatorname{co}h\right) ^{\ast }|_{\mathcal{U^{\prime }}}\left( y_{\mathcal{U^{\prime }}}\right) &\geq &\left( \operatorname{co}h\right) ^{\ast }|_{\mathcal{U^{\prime }}}\left( z_{\mathcal{U^{\prime }}}\right) +\langle u,z_{\mathcal{U^{\prime }}}\rangle \\
&&\qquad +\frac{1}{2}\langle QP_{\mathcal{U^{\prime }}}\left( y_{\mathcal{U^{\prime }}}-z_{\mathcal{U^{\prime }}}\right) ,P_{\mathcal{U^{\prime }}
}\left( y_{\mathcal{U^{\prime }}}-z_{\mathcal{U^{\prime }}}\right) \rangle
+o\left( \left\Vert P_{\mathcal{U^{\prime }}}\left( y_{\mathcal{U^{\prime }}
}-z_{\mathcal{U^{\prime }}}\right) \right\Vert ^{2}\right) \\
&=&\left( \operatorname{co}h\right) ^{\ast }|_{\mathcal{U^{\prime }}}\left( z_{\mathcal{U^{\prime }}}\right) +\langle u,z_{\mathcal{U^{\prime }}}\rangle \\
&&\qquad +\frac{1}{2}\langle \left( P_{\mathcal{U^{\prime }}}^{T}QP_{\mathcal{U^{\prime }}}\right) \left( y_{\mathcal{U^{\prime }}}-z_{\mathcal{U^{\prime }}}\right) ,\left( y_{\mathcal{U^{\prime }}}-z_{\mathcal{U^{\prime
}}}\right) \rangle +o\left( \left\Vert y_{\mathcal{U^{\prime }}}-z_{\mathcal{U^{\prime }}}\right\Vert ^{2}\right)
\end{eqnarray*}
and by (\ref{neqn:33}) we have $k_{v}^{\ast }\left( z_{\mathcal{U^{\prime }}
}\right) =\left( \operatorname{co}h\right) ^{\ast }\left( z_{\mathcal{U^{\prime }}
}+0_{\mathcal{V^{\prime }}}\right) =\left( \operatorname{co}h\right) ^{\ast }|_{\mathcal{U^{\prime }}}\left( z_{\mathcal{U^{\prime }}}^{\prime }\right) .$
Hence 
$$
P_{\mathcal{U^{\prime }}}^{T}QP_{\mathcal{U^{\prime }}}\in \partial
^{2,-}\left( \operatorname{co}h\right) ^{\ast }|_{\mathcal{U^{\prime }}}\left( z_{\mathcal{U^{\prime }}}\right) =\partial ^{2,-}k_{v}^{\ast }\left( z_{\mathcal{U^{\prime }}}\right) .
$$
When $\nabla ^{2}k_{v}^{\ast }\left( z_{\mathcal{U^{\prime }}}\right) $ exists we have $\partial ^{2,-}k_{v}^{\ast}\left( z_{\mathcal{U^{\prime }}}\right) =\nabla ^{2}k_{v}^{\ast }\left( z_{\mathcal{U^{\prime }}}\right) -\mathcal{P}(\mathcal{U})$ giving (\ref{neqn:13}).
\end{proof}
\smallskip

In the following we shall at times use the alternate notation $z_{\mathcal{U^{\prime }}}+z_{\mathcal{V^{\prime }}}=(z_{\mathcal{U^{\prime }}},z_{\mathcal{V^{\prime }}})$ to contain the notational burden of the former. The
proof of the next proposition follows a similar line of argument as in \cite[Proposition 3.1]{ebpenot:2}.

\begin{proposition}
\label{prop:subhessianinverse}Suppose $f:\mathbb{R}^{n}\rightarrow \mathbb{R}_{\infty }$ is a proper lower semi-continuous function and suppose that $\bar{x}$ is a tilt stable local minimum of $f$. In addition suppose $\bar{z}=0\in \operatorname{rel}$-$\operatorname{int}\partial f\left( \bar{x}\right) $, $\mathcal{U^{\prime }\subseteq U}$ and $v\left( u\right) \in \operatorname{argmin}_{v^{\prime
}\in \mathcal{V^{\prime }}\cap B_{\varepsilon }\left( 0\right) }f\left( \bar{x}+u+v^{\prime }\right) $ with $u:=P_{\mathcal{U^{\prime }}}\left[ m\left(
z_{\mathcal{U^{\prime }}}+\bar{z}_{\mathcal{V^{\prime }}}\right) \right] $
or $z_{\mathcal{U^{\prime }}}\in \partial L_{\mathcal{U^{\prime }}}^{\varepsilon }\left( u\right) =\partial k_{v}\left( u\right) $. Suppose $k_{v}^{\ast }:\mathcal{U^{\prime }}\rightarrow \mathbb{R}_{\infty }$ is a $C^{1,1}\left( B_{\varepsilon }\left( 0\right) \right) $ function for some $\varepsilon >0$ with $\nabla ^{2}k_{v}^{\ast }\left( z_{\mathcal{U^{\prime }}}\right) $ existing as a positive definite form. Then for $u:=\nabla
k_{v}^{\ast }\left( z_{\mathcal{U^{\prime }}}\right) $ we have
\begin{equation*}
Q=\nabla ^{2}k_{v}^{\ast }\left( z_{\mathcal{U^{\prime }}}\right) \quad
\implies \text{\quad }Q^{-1}\in \partial _{\mathcal{U^{\prime }}}^{2,-}(\operatorname{co}h)\left( u+v\left( u\right) ,z_{\mathcal{U^{\prime }}}+0_{\mathcal{V^{\prime }}}\right) .
\end{equation*}
\end{proposition}

\begin{proof}
When $u:=P_{\mathcal{U^{\prime }}}\left[ m\left( z_{\mathcal{U^{\prime }}}+
\bar{z}_{\mathcal{V^{\prime }}}\right) \right] $ or $z_{\mathcal{U^{\prime }}
}\in \partial L_{\mathcal{U^{\prime }}}^{\varepsilon }\left( u\right)
=\partial k_{v}\left( u\right) $ (see Proposition \ref{prop:LU}) we have $z_{\mathcal{U^{\prime }}}+0_{\mathcal{V}^{\prime}}\in \partial \left( \operatorname{co}h\right) \left( u+v\left( u\right)
\right) $. We show that $\left(
\begin{array}{cc}
Q^{-1} & 0 \\
0 & 0
\end{array}
\right) $ is a subhessian of $\operatorname{co}h$ at $(u,v(u))$ for $z_{\mathcal{U^{\prime }}}+0_{\mathcal{V^{\prime }}}\in \partial (\operatorname{co}h)\left(
u+v\left( u\right) \right) $ and hence deduce that (by definition) $Q^{-1}\in \partial _{\mathcal{U^{\prime }}}^{2,-}\left( \operatorname{co}h\right)
\left( u+v\left( u\right) ,z_{\mathcal{U^{\prime }}}+0_{\mathcal{V^{\prime }}}\right) $. Expanding $k_{v}^{\ast }$ via a second order Taylor expansion
around $z_{\mathcal{U^{\prime }}}$ we have for all $y_{\mathcal{U^{\prime }}}-z_{\mathcal{U^{\prime }}}\in B_{\varepsilon }\left( 0\right) $ a function $\delta \left( \varepsilon \right) \rightarrow 0$ as $\varepsilon \rightarrow
0$ with
\begin{eqnarray*}
\left( \operatorname{co}h\right) ^{\ast }\left( y_{\mathcal{U^{\prime }}}+0_{\mathcal{V^{\prime }}}\right) &=&k_{v}^{\ast }\left( y_{\mathcal{U^{\prime }}
}\right) \\
&=&k_{v}^{\ast }\left( z_{\mathcal{U^{\prime }}}\right) +\langle y_{\mathcal{U^{\prime }}}-z_{\mathcal{U^{\prime }}},u\rangle +\frac{1}{2}\langle Q\left(
y_{\mathcal{U^{\prime }}}-z_{\mathcal{U^{\prime }}}\right) ,\left( y_{\mathcal{U^{\prime }}}-z_{\mathcal{U^{\prime }}}\right) \rangle +o\left(
\left\Vert \left( y_{\mathcal{U^{\prime }}}-z_{\mathcal{U^{\prime }}}\right)
\right\Vert ^{2}\right) \\
&\leq &\left( \operatorname{co}h\right) ^{\ast }\left( \left( z_{\mathcal{U^{\prime }
}},0_{\mathcal{V^{\prime }}}\right) \right) +\langle y_{\mathcal{U^{\prime }}
}-z_{\mathcal{U^{\prime }}}+0_{\mathcal{V^{\prime }}},u+v\left( u\right)
\rangle \\
&&\qquad +\frac{1}{2}\langle \left(
\begin{array}{cc}
Q & 0 \\
0 & 0
\end{array}
\right) \left( \left( y_{\mathcal{U^{\prime }}},0_{\mathcal{V^{\prime }}
}\right) -\left( z_{\mathcal{U^{\prime }}},0\right) \right) ,\left( \left(
y_{\mathcal{U^{\prime }}},0_{\mathcal{V^{\prime }}}\right) -\left( z_{\mathcal{U^{\prime }}},0_{\mathcal{V^{\prime }}}\right) \right) \rangle \\
&&\qquad \qquad \qquad \qquad \qquad +\frac{1}{2} \delta \left( \varepsilon \right)
\left\Vert \left( \left( y_{\mathcal{U^{\prime }}},0_{\mathcal{V^{\prime }}
}\right) -\left( z_{\mathcal{U^{\prime }}},0_{\mathcal{V^{\prime }}}\right)
\right) \right\Vert ^{2}
\end{eqnarray*}
Then as $\operatorname{co}h\left( u+v\left( u\right) \right) =\langle z_{\mathcal{U^{\prime }}},u\rangle -\left( \operatorname{co}h\right) ^{\ast }\left( \left( z_{\mathcal{U^{\prime }}},0_{\mathcal{V^{\prime }}}\right) \right) $ and $\langle Q\left( y_{\mathcal{U^{\prime }}}-z_{\mathcal{U^{\prime }}}\right)
,\left( y_{\mathcal{U^{\prime }}}-z_{\mathcal{U^{\prime }}}\right) \rangle
\geq \alpha \left\Vert y_{\mathcal{U^{\prime }}}-z_{\mathcal{U^{\prime }}
}\right\Vert ^{2}$ we have
\begin{eqnarray*}
\operatorname{co}h\left( u^{\prime }+v\left( u^{\prime }\right) \right)
&=&\sup_{\left( y_{\mathcal{U^{\prime }}},y_{\mathcal{V^{\prime }}}\right)
}\left\{ \langle \left( y_{\mathcal{U^{\prime }}},y_{\mathcal{V^{\prime }}
}\right) ,\left( u^{\prime },v\left( u^{\prime }\right) \right) \rangle
-\left( \operatorname{co}h\right) ^{\ast }\left( y_{\mathcal{U^{\prime }}}+y_{\mathcal{V^{\prime }}}\right) \right\} \\
&\geq &\sup_{y_{\mathcal{U^{\prime }}}}\left\{ \langle \left( y_{\mathcal{U^{\prime }}},0_{\mathcal{V^{\prime }}}\right) ,\left( u^{\prime },v\left(
u^{\prime }\right) \right) \rangle -\left( \operatorname{co}h\right) ^{\ast }\left(
y_{\mathcal{U^{\prime }}}+0_{\mathcal{V^{\prime }}}\right) \right\} \\
&\geq &\operatorname{co}h\left( u+v\left( u\right) \right) +\langle z_{\mathcal{U^{\prime }}},u^{\prime }-u\rangle +\sup_{y_{\mathcal{U^{\prime }}}-z_{\mathcal{U^{\prime }}}\in B_{\varepsilon }\left( 0\right) }\left\{ \langle
y_{\mathcal{U^{\prime }}}-z_{\mathcal{U^{\prime }}},u^{\prime }-u\rangle
\right. \\
&&\qquad \qquad \left. -\frac{1}{2}\left( 1+\alpha ^{-1}\delta \left(
\varepsilon \right) \right) \langle Q\left( y_{\mathcal{U^{\prime }}}-z_{\mathcal{U^{\prime }}}\right) ,\left( y_{\mathcal{U^{\prime }}}-z_{\mathcal{U^{\prime }}}\right) \rangle \right\}
\end{eqnarray*}
and when $u^{\prime }-u\in \left( 1+\alpha ^{-1}\delta \left( \varepsilon
\right) \right) \varepsilon \left\Vert Q^{-1}\right\Vert ^{-1}B_{1}\left(
0\right) $ we have the supremum attained at
\begin{equation*}
y_{\mathcal{U^{\prime }}}-z_{\mathcal{U^{\prime }}}=\left( 1+\alpha
^{-1}\delta \left( \varepsilon \right) \right) ^{-1}Q^{-1}\left( u^{\prime
}-u\right) \in B_{\varepsilon }\left( 0\right) .
\end{equation*}
Hence when $\left( u^{\prime },v^{\prime }\right) \in B_{\gamma \left(
\varepsilon \right) }\left( 0\right) $, for $\gamma \left( \varepsilon
\right) :=\left( 1+\alpha ^{-1}\delta \left( \varepsilon \right) \right)
\varepsilon \left\Vert Q^{-1}\right\Vert ^{-1}$, we have
\begin{eqnarray*}
\operatorname{co}h\left( u^{\prime }+v^{\prime }\right) &\geq &\operatorname{co}h\left(
u^{\prime }+v\left( u^{\prime }\right) \right) \geq \operatorname{co}h\left(
u+v\left( u\right) \right) +\langle \left( z_{\mathcal{U^{\prime }}}+0_{\mathcal{V^{\prime }}}\right) ,\left( u^{\prime },v^{\prime }\right) -\left(
u,v\left( u\right) \right) \rangle \\
&&\qquad \qquad +\frac{1}{2}\langle \left(
\begin{array}{cc}
Q^{-1} & 0 \\
0 & 0
\end{array}
\right) \left( u^{\prime },v^{\prime }\right) -\left( u,v\left( u\right)
\right) ,\left( u^{\prime },v^{\prime }\right) -\left( u,v\left( u\right)
\right) \rangle \\
&&\qquad \qquad \qquad \qquad -\beta \left( \varepsilon \right) \left\Vert
\left( u^{\prime },v^{\prime }\right) -\left( u,v\left( u\right) \right)
\right\Vert ^{2}
\end{eqnarray*}
where $\beta \left( \varepsilon \right) =\left[ 1-\left( 1+\alpha
^{-1}\delta \left( \varepsilon \right) \right) ^{-1}\right] \left\Vert
Q^{-1}\right\Vert \rightarrow 0$ as $\varepsilon \rightarrow 0$. That is
\begin{equation*}
\left(
\begin{array}{cc}
Q^{-1} & 0 \\
0 & 0
\end{array}
\right) \in \partial ^{2,-} (\operatorname{co} h)\left( u+v\left( u\right) ,z_{\mathcal{U^{\prime
}}}+0_{\mathcal{V^{\prime }}}\right) ,
\end{equation*}
from which the result follows.
\end{proof}

\section{The Main Result}

The main tools we use to establish our results are the convexification that
tilt stable local minimum enable us to utilise \cite{Drusvy:1}, the
correspondence between tilt stability and the strong metric regularity of
the locally restricted inverse of the subdifferential and the connection
conjugacy has to inversion of subdifferentials of convex functions \cite{Drusvy:1, Art:1}. These tools and the coderivative
characterisation (\ref{neqn:2}) of tilt stability (being applicable to
convex functions) allows a chain of implications to be forged. The
differentiability properties we seek may be deduced via strong metric
regularity or alternatively via the results of \cite{Miroslav:1} after
invoking the Mordukhovich coderivative criteria for the Aubin property for
the associated subdifferential.

Once again we will
consider subspaces $\mathcal{U^{\prime}} \subseteq \mathcal{U}$.
We  now show that tilt stability is inherited by $k_{v}$.

\begin{proposition}
\label{prop:tilt1}Consider $f:\mathbb{R}^{n}\rightarrow \mathbb{R}_{\infty }
$ is a proper lower semi-continuous function and $v\left( u\right) \in \operatorname{argmin}_{v^{\prime }\in \mathcal{V}\cap B_{\varepsilon }\left( 0\right)
}f\left( \bar{x}+u+v^{\prime }\right) .$ Suppose that $f$ has a tilt stable
local minimum at $\bar{x}$ for $0\in \partial f\left( \bar{x}\right) $ then $v\left( \cdot \right) :\mathcal{U^{\prime}}\rightarrow \mathcal{V^{\prime}}$
is uniquely defined and the associated function $k_{v}\left( \cdot \right) :
\mathcal{U^{\prime}} \rightarrow \mathbb{R}_{\infty }$ has a tilt stable
local minimum at $0$.
\end{proposition}

\begin{proof}
In this case we have $\left( \bar{z}_{\mathcal{U}^{\prime}}, \bar{z}_{\mathcal{V^{\prime}} }\right) =\left( 0,0\right) $. By tilt stability we have $m\left( \cdot \right) $ a single valued Lipschitz functions and hence $v\left( \cdot \right) $ is unique. From Proposition \ref{prop:LU} and $\left\{ u\right\} =P_{\mathcal{U^{\prime}}}\left[ m\left( z_{\mathcal{U^{\prime}}}\right) \right] $ we have $z_{\mathcal{U^{\prime}}}\in \partial
_{\operatorname{co}}\left[ L_{\mathcal{U^{\prime}}}^{\varepsilon }+\delta
_{B_{\varepsilon }^{\mathcal{U^{\prime}}}\left( 0\right) }\right] \left(
u\right) $ and from Propositions \ref{prop:m} and \ref{prop:co} that
\begin{eqnarray*}
\left\{ \left( u,v\left( u\right) \right) \right\} &=&m \left( z_{\mathcal{U^{\prime}} }+\bar{z}_{\mathcal{V^{\prime}}}\right) =\operatorname{argmin}_{\left(
u^{\prime },v^{\prime }\right) }\left\{ g\left( u^{\prime }+v^{\prime
}\right) -\langle z_{\mathcal{U^{\prime}}},u^{\prime }+v^{\prime }\rangle
\right\} \\
&=&\operatorname{argmin}_{\left( u^{\prime },v^{\prime }\right) }\left\{ h\left(
u^{\prime }+v^{\prime }\right) -\langle z_{\mathcal{U^{\prime}}},u^{\prime
}+v^{\prime }\rangle \right\} \\
\text{and so }\left\{ u\right\} &=&\operatorname{argmin}_{u^{\prime }\in \mathcal{U^{\prime}} }\left\{ \left[ h\left( u^{\prime }+v\left( u^{\prime }\right)
\right) -\langle 0,u^{\prime }+v\left( u^{\prime }\right) \rangle \right]
-\langle z_{\mathcal{U^{\prime}}},u^{\prime }\rangle \right\} \\
&=&\operatorname{argmin}_{u^{\prime }\in \mathcal{U^{\prime}}}\left\{ k_{v}\left(
u^{\prime }\right) -\langle z_{\mathcal{U^{\prime}}},u^{\prime }\rangle
\right\}
\end{eqnarray*}
implying $\left\{ u\right\} =P_{\mathcal{U^{\prime}}}\left[ m\left( z_{\mathcal{U^{\prime}} }+0\right) \right] \subseteq \operatorname{argmin}_{u^{\prime
}\in \mathcal{U^{\prime}}}\left\{ k_{v}\left( u^{\prime }\right) -\langle z_{\mathcal{U^{\prime}}},u^{\prime }\rangle \right\} =\left\{ u\right\} $.
Hence
\begin{equation*}
\operatorname{argmin}_{u^{\prime }\in \mathcal{U^{\prime}} }\left\{ k_{v}\left(
u^{\prime }\right) -\langle z_{\mathcal{U^{\prime}}},u^{\prime }\rangle
\right\} =P_{\mathcal{U^{\prime}} }\left[ m\left( z_{\mathcal{U^{\prime}}
}\right) \right]
\end{equation*}
is clearly a single valued, locally Lipschitz function of $z_{ \mathcal{U^{\prime}} }\in B_{\varepsilon }^{\mathcal{U^{\prime}}}\left( 0\right)
\subseteq \mathcal{U^{\prime}}$.
\end{proof}

\begin{remark}
Clearly Proposition \ref{prop:tilt1} implies $k_{v}\left( \cdot \right) :
\mathcal{U}^{2}\rightarrow \mathbb{R}_{\infty }$ has a tilt stable local
minimum at $0$ relative to $\mathcal{U}^{2}\subseteq \mathcal{U}$.
\end{remark}

The following will help connect the positive definiteness of the densely
defined Hessians of the convexification $h$ with the associated uniform
local strong convexity of $f$. This earlier results \cite[Theorem 24,  Corollary 39]{eberhard:8} may be
compared with Theorem 3.3 of \cite{Drusvy:1} in that it links "stable strong
local minimizers of $f$ at $\bar{x}$" to tilt stability. We say $f_{z}:=f-\langle z,\cdot \rangle $ has a strict local minimum order two at $x^{\prime }$ relative to $B_{\delta }(\bar{x})\ni x^{\prime }$ when $f_{z}(x)\geq f_{z}\left( x^{\prime }\right) +\beta \left\Vert x-x^{\prime
}\right\Vert ^{2}$ for all $x\in B_{\delta }(\bar{x})$. It is a classical
fact that this is characterised by the condition $\left( f_{z}\right)
_{\_}^{\prime \prime }\left( x,0,h\right) >0$ for all $\left\Vert
h\right\Vert =1,$ see \cite[Theorem 2.2]{Studniarski:1986}.

 The following
result gives conditions on $f$,  in finite dimensions,
such that the coderivative in the second order sufficiency condition (\ref{neqn:2}) is uniformly bounded away from zero by a constant $\beta >0$.
Then indeed   (\ref{neqn:2})  is equivalent to this strengthened condition. This follows from a uniform bound on the associated quadratic minorant associated with the
strong stable local minimum. This phenomena was also observed in \cite[Theorem 5.36]{BonnansShapiroBook}
in the case of infinite dimensions for a  class of optimisation problems. As we already know this is 
true for $C^{1,1}$ functions (see Corollary \ref{cor:equiv}) and as we know that application of the 
infimal convolution to prox-regular functions produces a $C^{1,1}$ function,  there is a clear path
to connect these results. Indeed this is the approach used in \cite{eberhard:8,eberhard:9}. 

\begin{theorem} [\protect\cite{eberhard:8}, Theorem 34  part 1.]
\label{tilt:eb}  \label{thm:main}Suppose $f:\mathbb{R}^{n}\rightarrow \mathbb{R}_{\infty }$ is
lower--semicontinuous, prox--bounded (i.e. minorised by a quadratic) and $0\in \partial _{p}f(\bar{x})$.

 Suppose in addition there exists $\delta >0$ and $\beta >0$ such that
for all $(x,z)\in B_{\delta }(\bar{x},0)\cap \operatorname{Graph}\,\partial _{p}f$
the function $f-\langle z,\cdot \rangle $ has a strict local minimum order
two at $x$ in the sense that there exists $\gamma >0$ (depending on $x,y$)
such that for each $x^{\prime }\in  {B}_{\gamma }(x)$ we have
\begin{equation}
f(x^{\prime })-\langle z,x^{\prime }\rangle \geq f(x)-\langle z,x\rangle
+\beta \Vert x-x^{\prime }\Vert ^{2}.  \label{neqn:130}
\end{equation}

Then we have for all $\Vert w\Vert=1$ and $0\neq p\in{D}^{\ast}(\partial
_{p}f)(\bar{x},0)(w)$ that $\langle w,p\rangle\geq\beta>0$.

\end{theorem}

\begin{corollary} \label{cor:strict}
Suppose $f:\mathbb{R}
^{n}\rightarrow \overline{\mathbb{R}}$ a is lower semi--continuous,
prox--bounded and $f$ is both prox--regular at $\bar{x}$ with respect to $0\in \partial _{p}f(\bar{x})$ and subdifferentially continuous there.  Then the following are equivalent:
\begin{enumerate}
	\item \label{equiv:1} For all $\Vert w\Vert=1$ and $p\in{D}^{\ast}(\partial
	_{p}f)(\bar{x},0)(w)$ we have $\langle w,p\rangle>0$.
	\item \label{equiv:2}  There exists $\beta >0$ such that for all $\Vert w\Vert=1$ and $p\in{D}^{\ast}(\partial
	_{p}f)(\bar{x},0)(w)$ we have $\langle w,p\rangle \geq \beta >0$.
\end{enumerate}
\end{corollary}

\begin{proof}
We only need show \ref{equiv:1} implies \ref{equiv:2}. By  \cite[Theorem 1.3]{rock:7} we have  \ref{equiv:1} implying a tilt stable local minimum at $\bar{x}$. Now apply \cite[Theorem 3.3]{Drusvy:1} to deduce the existence of a $\delta >0$  such that for all $(x,z)\in B_{\delta }(\bar{x},0)\cap \operatorname{Graph}
\,\partial _{p}f$   we have $x$ a strict local minimizer order two of
the function $f-\langle z,\cdot \rangle $ in the sense that (\ref{neqn:130})  holds for some uniform value $\beta >0$ for all  $x^{\prime} \in B_{\gamma} (x)$. Now apply Theorem \ref{tilt:eb} to obtain \ref{equiv:2}.
\end{proof}
\smallskip 

Another
condition equivalent to all of those in  \cite[Theorem 1.3]{rock:7}   is the
following
\begin{equation}
f_{s}^{\prime \prime }\left( x,z,u\right) >0\text{\quad for all }\left(
x,z\right) \in B_{\delta }(\bar{x},0)\cap \operatorname{Graph}\,\partial _{p}f,
\label{neqn:44}
\end{equation}
which is motivated by the classical observation that $f^{\prime \prime
}\left( x,z,u\right) >0$ implies $f-\langle z,\cdot \rangle $ has a strict
local minimum order 2 at $x$ (see \cite[Theorem 2.2]{Studniarski:1986}). We will show that a 
stronger version gives an equivalent characterisation in Corollary \ref{cor:equiv} below. The following
construction is also standard. Denote
\begin{equation*}
\hat{D}^{\ast }\left( \partial _{p}f\right) (x,z)(w)=\{v\in \mathbb{R}^{n}\mid (v,-w)\in \hat{N}_{\operatorname{Graph}\,\partial _{p}f}(x,z)\},
\end{equation*}
where $\hat{N}_{\operatorname{Graph}\,\partial _{p}f}(x,z)=\left(
\limsup_{t\downarrow 0}\frac{\operatorname{Graph}\,\partial _{p}f-(x,p)}{t}\right)
^{\circ }$ is the contingent normal cone. Then we have $D^{\ast }\left(
\partial _{p}f\right) (\bar{x},0)(w)=g$-$\limsup_{\left( x,z\right)
\rightarrow _{S_{p}\left( f\right) }\left( \bar{x},0\right) }\hat{D}^{\ast
}\left( \partial _{p}f\right) (x,z)(w)$ (the graphical limit supremum \cite[page 327]{rock:6}).

\begin{corollary}\label{cor:equiv}
\label{cor:ebnonzero} Suppose $f:\mathbb{R}^{n}\rightarrow  {\mathbb{R}_{\infty}}$ a is lower semi--continuous, prox--bounded and $f$ is both
prox--regular at $\bar{x}$ with respect to $0\in \partial _{p}f(\bar{x})$
and subdifferentially continuous there. Then the following are equivalent:

\begin{enumerate}
\item \label{ppart:1}For all $\Vert w\Vert =1$ and $p\in {D}^{\ast
}(\partial _{p}f)(\bar{x},0)(w)$ we have $\langle w,p\rangle >0$.

\item \label{ppart:2} There exists $\beta >0$ such that for all  $\Vert w\Vert =1$ we have $f_{s}^{\prime \prime
}\left( x,z,w\right) \geq \beta >0$ for all $\left( x,z\right) \in B_{\delta }(\bar{x},0)\cap \operatorname{Graph}\,\partial _{p}f$, for some $\delta >0$.
\end{enumerate}
Moreover the $\beta$ in part \ref{ppart:2} may be taken as that in Corollary \ref{cor:strict} part \ref{equiv:2}. 
\end{corollary}

\begin{proof}
(\ref{ppart:1}.$\implies $\ref{ppart:2}.) By Corollary \ref{cor:strict} we have \ref{ppart:1}
equivalent to condition \ref{equiv:2} of Corollary \ref{cor:strict} (for some fixed $\beta >0$). 
Now define $G := f - \frac{\beta'}{2}\| \cdot - \bar{x} \|^2$ (for $0<\beta' < \beta$). Apply the  sum rule for the limiting  subgradient and that for the  coderivatives \cite[Theorem 10.41]{rock:6} to deduce that $0 \in \partial G (\bar{x}) = \partial f(\bar{x})-\beta' \times 0$ and also 
${D}^{\ast} (\partial G)(\bar{x},0) (w) \subseteq {D}^{\ast} (\partial f )(\bar{x},0) (w) -\beta' w$.
Then for any $v \in {D}^{\ast} (\partial G)(\bar{x},0) (w)$ we have 
$\langle v, w \rangle = \langle p , w \rangle -\beta' \|w \|^2 >0$. Now apply Theorem 3.3 of \cite{Drusvy:1} to deduce there exists a strict local minimum order two for $G_z := G - \langle z , \cdot \rangle$ at each 
 $\left( x,z\right) \in B_{\delta }(\bar{x},0)\cap \operatorname{Graph}\,\partial G$ for some 
 $\delta >0$.
Noting that $\partial G$ and $\partial_p G$ locally coincide around $(\bar{x}, 0)$, after possibly reducing $\delta >0$, we apply  (see \cite[Theorem 2.2]{Studniarski:1986}) to deduce that $(G_z)^{\prime\prime} (x,0,w) 
= f^{\prime\prime} (x,z,w) - \beta' \|w\|^2  > 0$ for all $\beta '< \beta$. This implies \ref{ppart:2}. 

(\ref{ppart:2}.$\implies $\ref{ppart:1}.) Let $ \left( x,z\right) \in B_{\delta }(\bar{x},0)\cap \operatorname{Graph}\,\partial _{p}f$. We use the fact that $f_{s}^{\prime
\prime }\left( x,z,w\right) > \beta' >0$ for all $0< \beta' < \beta$ and  $\Vert w\Vert =1$ implies $$f_{z}:=f-\langle z,\cdot \rangle -\frac{\beta'}{2} \| \cdot - x \|^2$$ has $x$ as a strict local minimum
order 2 at $x$. We may now apply \cite[Theorem 67]{eberhard:9} to deduce that   for all $y\in {\hat{D}}^{\ast}(\partial _{p}f_{z})(x,0)(w)$ we
have $\langle w,y\rangle \geq  0$. By direct calculation from definitions one may show that 
${\hat{D}}^{\ast}(\partial _{p}f_{z})(x,0)(w)={\hat{D}}^{\ast }(\partial _{p}f)(x,z)(w) - \beta' w$
and hence $\langle p , w \rangle \geq \beta' \|w\|^2$ for all  $p \in {\hat{D}}^{\ast }(\partial _{p}f)(x,z)(w)$. 
Taking the graphical limit supremum \cite[identity 8(18)]{rock:6} of ${\hat{D}}^{\ast }(\partial _{p}f)(x,z)(\cdot )$
as $\left( x,z\right) \rightarrow _{S_{p}\left( f\right) }\left( \bar{x},0\right) $ gives \ref{ppart:1}.
\end{proof}
\smallskip

One of the properties that follows from \cite[Theorem 1.3]{rock:7} is that the Aubin Property (or pseudo-Lipschitz property) holds for the mapping $z \mapsto B_{\delta} (\bar{x}) \cap (\partial f)^{-1} (z)$.  The Aubin property is related to differentiability via the following result. 

\begin{theorem}  [\protect\cite{Miroslav:1}, Theorem 5.3] \label{thm:eb} Suppose $H$ is a Hilbert space
	and $f:H\mapsto \mathbb{R}_{\infty }$ is lower semi-continuous,
	prox-regular, and subdifferentially continuous at $\bar{x}\in \operatorname{int}
	\operatorname{dom}\partial f$ for some $\bar{v}\in \partial f(\bar{x})$. In
	addition, suppose $\partial f$ is pseudo-Lipschitz (i.e. possess the Aubin
	property) at a Lipschitz rate $L$ around $\bar{x}$ for $\bar{v}$. Then there
	exists $\varepsilon >0$ such that $\partial f(x)=\{\nabla f(x)\}$ for all $x\in B_{\varepsilon }(\bar{x})$ with $x\mapsto \nabla f(x)$ Lipschitz at the
	rate $L$.
\end{theorem}

\begin{corollary}
	Under the assumption of Proposition \ref{prop:tilt1} we have $z\mapsto
	\partial k_{v}^{\ast }(z)$ a single valued Lipschitz continuous mapping in
	some neighbourhood of $0$.
\end{corollary}

\begin{proof}
	We invoke Theorems \ref{tilt:eb} and \ref{thm:eb}. As $(\operatorname{co}
	k_{v})^{\ast}= k_{v}^{\ast }$ and being a convex function it is prox-regular
	and subdifferentially continuous so $(\partial _{p}\operatorname{co}
	k_{v})^{-1}=(\partial \operatorname{co}k_{v})^{-1}=\partial k_{v}^{\ast }$ is single
	valued and Lipschitz continuous by Theorem \ref{thm:eb}, noting that the
	tilt stability supplies the Aubin property for $\left( \partial \operatorname{co}
	k_{v}\right) ^{-1}$ via  \cite[Theorem 1.3]{rock:7} .
\end{proof}

We include the following for completeness. We wish to apply this in conjuction with 
Alexandrov's theorem and this is valid due to the equivalence of the existence of a 
Taylor expansion and twice differentiability in the extended sense (see \cite[Corollary 13.42, Theorem 13.51]{rock:6}).
\begin{lemma}\label{lem:48}
	Suppose  $f: \mathbb{R}^n \to \mathbb{R}_{\infty}$ is a locally finite convex function at $B_{\varepsilon} (x)$ which is twice differentiable at $\bar{x} \in  B_{\varepsilon} (x)$ with  $\bar{z} := \nabla f(\bar{x})$ and $Q:= \nabla^2 f(\bar{x})$ positive definite. Then we we have 
	\begin{equation}
	\bar{z} := \nabla f(\bar{x}) \text{ and } Q:= \nabla^2 f(\bar{x})
	\quad \iff \quad \bar{x} = \nabla f^{\ast} (\bar{z}) \text{ and } Q^{-1}=\nabla^2 f^{\ast} (\bar{z}). \label{iffQ}
	\end{equation}
\end{lemma}

\begin{proof}
In \cite{Gian:1} it is shown that when $g$ is convex with $g(0)=0$, $\nabla g (0) = 0$ and twice differentiable $x=0$ in the sense that the following Taylor expansion exists:
$
g(y) = \frac{1}{2} (Qy )T y + o(\| y \|^2).
$
Then we have the corresponding Taylor expansion:
$
g^{\ast} (x) = \frac{1}{2} (Q^{-1} x)^T x + o(\|x\|^2). 
$
We may apply \cite[Corollary 13.42]{rock:6} to claim these expansions are equivalent to 
the existence of a Hessian for both functions (twice differentiability in the extended sense) 
where $0 = \nabla g(0)$, $Q = \nabla^2 g(0)$ and $0=\nabla g^{\ast} (0)$,  $Q^{-1} = \nabla^2 g^{\ast} (0)$. Now apply this 
to the function $g(y):=f(y+\nabla f(\bar{x})) - \langle \nabla f (\bar{x}) , y+\nabla f(\bar{x}) \rangle$ noting that $g^{\ast} (z) = f^{\ast} (z + \nabla f(\bar{x})) - \langle \bar{x}, z \rangle$. 
We have $\nabla^2 g (0) = \nabla^2 f(\bar{x})$, $\nabla g^{\ast} (0) = 0$ implies 
$f^{\ast} (\nabla f(\bar{x})) = \bar{x}$ and $\nabla^2 g^{\ast} (0) = Q^{-1} = \nabla f^{\ast} (\nabla f(\bar{x}))$, demonstrating the forward implication ($\implies$) in (\ref{iffQ}). To obtain the reverse implication we apply the proven result to the convex function $f^{\ast}$ using the 
bi-conjugate formula $f^{\ast\ast} = f$. 
\end{proof}

Our main goal is to demonstrate that the restriction of $f$ to the set $\mathcal{M}:=\left\{ \left( u,v\left( u\right) )\mid u\in \mathcal{U}^{2}\right) \right\} $ coincides with a $C^{1,1}$ smooth
function of $u \in \mathcal{U}$. Consequently we will be focusing on the case when
$\mathcal{U}^2$ is a linear subspace and so take $\mathcal{U^{\prime}}
\equiv \mathcal{U}^2$ in our previous results. The next result demonstrates
when there is a symmetry with respect to conjugation in the tilt stability
property for the auxiliary function $k_{v}$.

\begin{theorem}
Consider $f:\mathbb{R}^{n}\rightarrow \mathbb{R}_{\infty }$ is a proper
lower semi-continuous function, which is a prox-regular function at $\bar{x}$
for $0\in \partial f(\bar{x})$ with a nontrivial subspace $\mathcal{U}^{2}=b^{1}\left( \underline{\partial }^{2}f\left( \bar{x},0\right) \right)
\subseteq \mathcal{U}$. Denote $\mathcal{V}^{2}=(\mathcal{U}^{2})^{\perp }$, 
  let $v\left( u\right) \in \operatorname{argmin}_{v^{\prime }\in \mathcal{V}^{2}\cap B_{\varepsilon }\left( 0\right) }f\left( \bar{x}+u+v^{\prime
}\right) :\mathcal{U}^{2}\rightarrow \mathcal{V}^{2}$ and  $k_{v}\left(
u\right) :=h\left( u+v\left( u\right) \right) :\mathcal{U}^{2}\rightarrow
\mathbb{R}_{\infty }$. Suppose also that $f$ has a tilt stable local minimum
at $\bar{x}$ for $0\in \partial f\left( \bar{x}\right) $ then for $p\neq 0$
we have
\begin{equation}
\forall q\in D^{\ast }\left( \nabla k_{v}^{\ast }\right) \left( 0,0\right)
\left( p\right) \quad \text{we have \quad }\langle p,q\rangle >0
\label{neqn:15}
\end{equation}
and hence $k_{v}^{\ast }$ has a tilt stable local minimum at $0\in \partial
k_{v}^{\ast }\left( 0\right) .$
\end{theorem}

\begin{proof}
On application of Propositions \ref{prop:tilt1} and \ref{prop:co} we have $\operatorname{co}k_{v}\left( \cdot \right) :\mathcal{U}^2\rightarrow \mathbb{R}_{\infty }$ possessing a tilt stable local minimum at $0$. As $\operatorname{co}
k_{v}\left( \cdot \right) $ is convex it is prox-regular at $0$ for $0\in
\partial \operatorname{co} k_{v}\left( 0\right) \ $and subdifferentially continuous
at $0$ \cite[Proposition 13.32]{rock:6}. Hence we may apply \cite[Theorem 1.3]{rock:7} to obtain the equivalent condition for tilt stability. For all $q \ne 0$
\begin{equation}
\langle p,q\rangle >0\text{\quad for all \quad }p\in D^{\ast }\left(
\partial \left[ \operatorname{co}k_{v}\right] \right) \left( 0,0\right) \left(
q\right) .  \label{neqn:101}
\end{equation}
Now apply Corollary \ref{cor:strict} to deduce the existence of $\beta >0$ such that  
 $\langle p,q\rangle \geq \beta >0$ for all $(p,q)$ taken in (\ref{neqn:101})
with $\| q \| =1$.

For this choice of $v(\cdot)$ we have $k_v = L^{\varepsilon}_{\mathcal{U}^2}$. From Proposition \ref{prop:reg} part \ref{part:3}, Remark \ref{rem:lem} and
Lemma \ref{lem:conv} we see that $\nabla k_{v}(0)=\{0\}=\nabla \operatorname{co}
k_{v}(0)$. Then whenever $x^{k}\in S_{2}(\operatorname{co}k_{v})$ with $x^{k}\rightarrow 0$ (as we always have $z^{k}=\nabla \operatorname{co}
k_{v}(x^{k})\rightarrow 0=\nabla \operatorname{co}k_{v}(0)$) it follows from
Corollary \ref{cor:ebnonzero} that we have
\begin{equation}
\left( \operatorname{co}k_{v}\right) _{s}^{\prime \prime }\left( x^{k},\nabla \operatorname{co}k_{v}\left( x^{k}\right) ,h\right) =\langle \nabla ^{2}\operatorname{co}
k_{v}(x^{k})h,h\rangle >0\,\text{\quad for all }h\in \mathcal{U}^{2}
\label{neqn:46}
\end{equation}
for $k$ sufficiently large. By Alexandrov's theorem this positive
definiteness of Hessians must hold on a dense subset of some neighbourhood
of zero. By the choice of $v(\cdot )$ we have $k_{v}(u)=L_{\mathcal{U}^{2}}^{\varepsilon }(u)$ and hence we may assert that $\partial _{\operatorname{co}}L_{\mathcal{U}^{2}}^{\varepsilon }(u)=\partial \operatorname{co}k_{v}(u)\neq
\emptyset $ in some neighbourhood of the origin in $\mathcal{U}^{2}$.

Since $\left[ \operatorname{co}k_{v}\right] ^{\ast }=k_{v}^{\ast }$ and $\nabla
k_{v}^{\ast }$ $=\left[ \partial \operatorname{co}k_{v}\right] ^{-1}$ we may apply
\cite[identity 8(19)]{rock:6} to deduce that for $\left\Vert q\right\Vert =1$
we have
\begin{equation*}
-q\in D^{\ast }\left( \left[ \partial \operatorname{co}k_{v}\right] ^{-1}\right)
\left( 0,0\right) \left( -p\right) =D^{\ast }\left( \nabla k_{v}^{\ast
}\right) \left( 0,0\right) \left( -p\right) .
\end{equation*}
Hence we can claim that for $q\neq 0$, after a sign change, that $\langle
p,q\rangle =\langle -p,-q\rangle \geq \beta >0$. We need to rule out the
possibility that $0\in D^{\ast }\left( \nabla k_{v}^{\ast }\right) \left(
0,0\right) \left( p\right) $ for some $p\neq 0$. To this end we may use the
fact that $k_{v}^{\ast }$ is $C^{1,1}$ (and convex) and apply \cite[ Theorem
13.52]{rock:6} to obtain the following characterisation of the convex hull
of the coderivative in terms of limiting Hessians. Denote $S_{2}(k_{v}^{\ast
}):=\{x\mid \nabla ^{2}k_{v}^{\ast }(x)\text{ exists}\}$ then
\begin{equation*}
\operatorname{co}D^{\ast }(\nabla k_{v}^{\ast })(0,0)(p)=\operatorname{co}\{Ap\mid
A=\lim_{k}\nabla ^{2}k_{v}^{\ast }(z^{k})\text{ for some $z^{k}$($\in
S_{2}(k_{v}^{\ast })$)$\rightarrow 0$}\}.
\end{equation*}
Now suppose $0\in D^{\ast }\left( \nabla k_{v}^{\ast }\right) \left(
0,0\right) \left( p\right) $ then there exists $A^{i}=\lim_{k}\nabla
^{2}k_{v}^{\ast }(z_{i}^{k})$ for $z_{i}^{k}\rightarrow 0$ such that $0=q:=\sum_{i=1}^{m}\lambda _{i}A^{i}p\in \operatorname{co} D^{\ast }\left( \nabla
k_{v}^{\ast }\right) \left( 0,0\right) \left( p\right) $. As $p\neq 0$ we
must then have $\langle p,q\rangle =p^{T}(\sum_{i=1}^{m}\lambda
_{i}A^{i})p=0 $ where  $B:=\sum_{i=1}^{m}\lambda _{i}A^{i}$ is a symmetric
positive semi-definite matrix. The inverse $(A_{k}^{i})^{-1}$ exists (relative to $\mathcal{U}^2$) due to (\ref{neqn:46}).
 Now apply the duality formula for Hessians Lemma \ref{lem:48}  to deduce that when $x_{i}^{k}:=\nabla k_{v}^{\ast }(z_{i}^{k})$
then $A_{k}^{i}=\nabla ^{2}k_{v}^{\ast }(z_{i}^{k})$ iff $(A_{k}^{i})^{-1}=\nabla ^{2}(\operatorname{co} k_{v})(x_{i}^{k})$. 

%As $k_v (x_{i}^{k}) = L^{\varepsilon}_{\mathcal{U}^2} (x_{i}^{k})$ we then
%have $\nabla (\operatorname{co} L^{\varepsilon}_{\mathcal{U}^2}) (x_{i}^{k})
%\in \partial_{\operatorname{co}}  L^{\varepsilon}_{\mathcal{U}^2} (x_{i}^{k}) \ne \emptyset$.

We now apply Lemma \ref{lem:boundprox} to deduce that the limiting
subhessians of $h(w):=f(\bar{x}+w)$ satisfy (\ref{neqn:34}). We will want to
apply this bound to the limiting subhessians of $\operatorname{co}h$ at $x_{i}^{k}+v(x_{i}^{k})$. To this end we demonstrate that $\Delta _{2}h\left(
x_{i}^{k}+v(x_{i}^{k}),\left( z_{i}^{k},0\right) ,t,w\right) \geq \Delta
_{2}\left( \operatorname{co}h\right) \left( x_{i}^{k}+v(x_{i}^{k}),\left(
z_{i}^{k},0\right) ,t,w\right) $ for all $t\in \mathbb{R}$ and any $w$. This
follows from Lemma \ref{lem:conv}, Proposition \ref{cor:conv} in that $\left( z_{i}^{k},0\right) \in \partial \operatorname{co} h\left(
x_{i}^{k}+v(x_{i}^{k})\right) =\partial h\left(
x_{i}^{k}+v(x_{i}^{k})\right) ,$ $\operatorname{co}h\left(
x_{i}^{k}+v(x_{i}^{k})\right) =h\left( x_{i}^{k}+v(x_{i}^{k})\right) $ and $\operatorname{co}h\left( u+v\right) \leq h\left( u+v\right) $ for all $\left(
u,v\right) \in \mathcal{U}^2\times \mathcal{V}^2$. On taking the a limit
infimum for $t\rightarrow 0$ and $w\rightarrow u\in \mathcal{U}^2$ we obtain
\begin{eqnarray*}
&&q\left( \partial ^{2,-}\left( \operatorname{co}h\right) \left(
x_{i}^{k}+v(x_{i}^{k}),\left( z_{i}^{k},0\right) \right) \right) \left(
u\right) =\left( \operatorname{co}h\right) _{s}^{\prime \prime }\left(
x_{i}^{k}+v(x_{i}^{k}),\left( z_{i}^{k},0\right) ,u\right) \\
&\leq &h^{\prime \prime }\left( x_{i}^{k}+v(x_{i}^{k}),\left(
z_{i}^{k},0\right) ,u\right) =q\left( \partial ^{2,-}h\left(
x_{i}^{k}+v(x_{i}^{k}),\left( z_{i}^{k},0\right) \right) \right) \left(
u\right) .
\end{eqnarray*}
Hence the bound in (\ref{neqn:34}) involving the constant $M>0$ applies to any $Q_k \in \partial ^{2,-}\left( \operatorname{co}h\right) \left( x_{i}^{k}+v(x_{i}^{k}),\left( z_{i}^{k},0\right) \right) $
for $k$ large.

As $A_{k}^{i}=\nabla ^{2}k_{v}^{\ast }(z_{i}^{k})$ by Proposition \ref{prop:subhessianinverse} we have $(A_{k}^{i})^{-1}=(\nabla^2 _{\mathcal{U}^{2}}h^{\ast }(z_{i}^{k}+0_{\mathcal{V}^{2}}))^{-1}\in \partial _{\mathcal{U}^{2}}^{2,-}\left( \operatorname{co}h\right) (x_{i}^{k}+v(x_{i}^{k}))$ and on
restricting to the $\mathcal{U}^{2}$ space and using (\ref{neqn:33}), (\ref{neqn:34}) and (\ref{neqn:36}) we get for all $p\in \mathcal{U}^{2}$ that
\begin{equation*}
\langle A_{k}^{i},pp^{T}\rangle =\langle \nabla ^{2}k_{v}^{\ast
}(z_{i}^{k}),pp^{T}\rangle =\langle \nabla _{\mathcal{U}^{2}}^{2}h^{\ast
}(z_{i}^{k}+0_{\mathcal{V}}),pp^{T}\rangle =\langle \left[ (A_{k}^{i})^{-1}
\right] ^{-1},pp^{T}\rangle \geq \frac{1}{M}.
\end{equation*}
Thus $\{A_{k}^{i}\}$ are uniformly positive definite. By \cite{Gian:1} we
have $(A_{k}^{i})^{-1}=\nabla ^{2}(\operatorname{co}k_{v})(x_{i}^{k})$ existing at $x_{i}^{k}$ and hence
\begin{equation*}
(A_{k}^{i})^{-1} u=\nabla ^{2}(\operatorname{co}k_{v})(x_{i}^{k})u\in D^{\ast
}(\nabla \operatorname{co}k_{v})(x_{i}^{k}, z_{i}^{k})(u)\quad \text{for all $u\in
\mathcal{U} ^{2}$.}
\end{equation*}
Then, for $u\neq 0$, by Theorem \ref{tilt:eb} we have $\langle \nabla ^{2}(\operatorname{co}k_{v})(x_{i}^{k})u,u\rangle \geq \frac{\beta }{2}>0$ for $k$ large
implying $\left\{ (A_{k}^{i})^{-1}\right\} $ remain uniformly positive
definite on $\mathcal{U}^{2}$. Hence $\left\{ A_{k}^{i}\right\} $ remain
uniformly bounded within a neighbourhood of the origin within $\mathcal{U}^{2}$. Thus on taking the limit we get $A^{i}=\lim_{k}A_{k}^{i}$ is positive
definite and hence $B:=\sum_{i=1}^{m}\lambda _{i}A^{i}$ is actually positive
definite, a contradiction.

As $k_{v}^{\ast }$ is convex and finite at $0$, it is prox-regular and
subdifferentially continuous at $0$ for $0\in \partial k_{v}^{\ast }\left(
0\right) $ by \cite[Proposition 13.32]{rock:6}. Another application of \cite[Theorem 1.3]{rock:7} allows us to deduce that $k_{v}^{\ast }$ has a tilt
stable local minimum at $0\in \nabla k_{v}^{\ast }\left( 0\right) .$
\end{proof}

We may either use the strong metric regularity property to obtain the
existence of a smooth manifold or utilizes the Mordukhovich criteria for the
Aubin property \cite{rock:6} and the results of \cite{Miroslav:1} on single
valuedness of the subdifferential satisfying a pseudo-Lipschitz property,
namely:\bigskip

\begin{proof}
\textbf{[of Theorem \ref{thm:1}]} \textit{using strong metric regularity}
\newline
Note first that $\mathcal{U}^{2}\subseteq \mathcal{U}$ corresponds to (\ref{eqn:44}) for $\bar{z}=0$. Let $\{v\left( u\right)\} = \operatorname{argmin}_{v^{\prime }\in
\mathcal{V}^{2}\cap B_{\varepsilon }\left( 0\right) }f\left( \bar{x}+u+v^{\prime }\right) $. We apply either \cite[Theorem 1.3]{rock:7} or \cite[Theorem 3.3]{Drusvy:1} that asserts that as $k_{v}^{\ast }$ is prox-regular
and subdifferentially continuous at $0$ for $0\in \partial k_{v}^{\ast
}\left( 0\right) $ then $\partial k_{v}^{\ast }$ is strongly metric regular
at $\left( 0,0\right) .$ That is there exists $\varepsilon >0$ such that
\begin{equation*}
B_{\varepsilon }\left( 0\right) \cap \left( \partial k_{v}^{\ast }\right)
^{-1}\left( u\right)
\end{equation*}
is single valued and locally Lipschitz for $u\in \mathcal{U}^{2}$
sufficiently close to $0$. But as $\left( \partial k_{v}^{\ast }\right)
^{-1}=\partial k_{v}^{\ast \ast }=\partial \left[ \operatorname{co}k_{v}\right] $ is
a closed convex valued mapping (and hence has connected images) we must have
the existence of $\delta >0$ such that for $u\in B_{\delta }^{\mathcal{U}^{2}}\left( 0\right) $ we have $\partial \left[ \operatorname{co}k_{v}\right] \left(
\cdot \right) $ a singleton locally Lipschitz mapping (giving
differentiability). As $\{v\left( u\right) \} = \operatorname{argmin}_{v^{\prime }\in
\mathcal{V}^{2}\cap B_{\varepsilon }\left( 0\right) }\left\{ h\left(
u+v^{\prime }\right) -\langle \bar{z}_{\mathcal{V}^{2}},v^{\prime }\rangle
\right\} :\mathcal{U}^{2}\cap B_{\varepsilon }\left( 0\right) \rightarrow
\mathcal{V}^{2}$ we have $k_{v}(u)=L_{\mathcal{U}^{2}}^{\varepsilon }(u)$
for $u\in \operatorname{int}B_{\varepsilon }^{\mathcal{U}^{2}}\left( 0\right) $.
Hence $\nabla \operatorname{co} L_{\mathcal{U}^{2}}^{\varepsilon }(u)\in \partial _{\operatorname{co}}L_{\mathcal{U}^{2}}^{\varepsilon }(u)\neq \emptyset $ and by
Corollary \ref{cor:conv} we have on $\mathcal{U}^{2}$ that $h\left(
u+v\left( u\right) \right) =\left[ \operatorname{co}h\right] \left( u+v\left(
u\right) \right) $ and hence
\begin{equation*}
\partial \left[ \operatorname{co}k_{v}\right] \left( u\right) =\partial \left[ \operatorname{co}h\right] \left( u+v\left( u\right) \right) =\partial g\left( u+v\left(
u\right) \right)
\end{equation*}
is single valued implying $\nabla _{u}g\left( u+v\left( u\right) \right) $
exists where $g\left( \cdot \right) :=\left[ \operatorname{co}h\right] \ \left(
\cdot \right) .$ \bigskip
\end{proof}

\begin{corollary}\label{cor:1}
Under the assumptions of Theorem \ref{thm:1} we have $\nabla   L_{\mathcal{U}^{2}}^{\varepsilon }(u)$ existing  as a Lipschitz function locally  on $B^{\mathcal{U}}_{\varepsilon} (0)$.
\end{corollary}

\begin{proof}
Applying Corollary \ref{cor:conv} again we can assert that under our current
assumptions that locally we have $\operatorname{co}k_{v} = k_{v} =   L_{\mathcal{U}^{2}}^{\varepsilon }$ and hence $\nabla k_{v} (u) = \nabla   L_{\mathcal{U}^{2}}^{\varepsilon }(u)$ exists as a Lipschitz function locally 
on $B^{\mathcal{U}}_{\varepsilon} (0)$. 
\end{proof}
\bigskip

\begin{proof}
\textbf{[of Theorem \ref{thm:1}]} \textit{using the single valuedness of the
subdifferential satisfying a pseudo-Lipschitz property.} \newline
We show that $D^{\ast }(\partial \lbrack \operatorname{co}k_{v}])(0 , 
0)(0)=\{0\}$. To this end we use (\ref{neqn:15}). Indeed this implies that $q\neq 0$ for any $p\neq 0$ for all $q \in D^{\ast} (\nabla k^{\ast}_v ) (0,0)(0) (p)$. Applying the result \cite[identity 8(19)]{rock:6} on inverse functions and coderivatives  we have $q=0$ implies $p=0$ for all  $p\in D^{\ast }(\partial \lbrack \operatorname{co}
k_{v}])(0,0)(q)$. Hence we have $D^{\ast }(\partial \lbrack \operatorname{co}
k_{v}])(0, 0)(0)=\{0\}$. Now apply the Mordukhovich criteria for the
Aubin property \cite[Theorem 9.40]{rock:6} to deduce that $\partial \lbrack
\operatorname{co}k_{v}]$ has the Aubin property at $0$ for $0\in \partial \lbrack
\operatorname{co}k_{v}](0)$. Now apply Theorem \ref{thm:eb} to deduce that $u\mapsto
\nabla \lbrack \operatorname{co}k_{v}](u)$ exists a single valued Lipschitz mapping
in some ball $B_{\delta }^{\mathcal{U}^2}\left( 0\right) $ in the space $\mathcal{U}^2$. We now finish the proof as before in the first version.
\end{proof}

If we assume more, essentially what is needed to move towards partial smoothness
we get a $C^{1,1}$ smooth manifold. \bigskip

\begin{proof}
\textbf{[of Theorem \ref{thm:2}] } First note that when we have (\ref{eqn:1}) holding using $f$ then we must (\ref{eqn:1}) holding using $g := \operatorname{co}
h $. Thus by Proposition \ref{prop:reg} part \ref{part:4} have (\ref{neqn:26}) holding using $g$ (via the convexification argument). As $g\left(
w\right) :=\left[ \operatorname{co}h \right] \left( w\right) $ for $w\in
B_{\varepsilon }\left( 0\right) $ is a convex function we have $g$ a regular in $B_{\varepsilon
}\left( 0\right) $. Moreover as $g\left( u+v\left( u\right) \right) =f\left(
\bar{x}+u+v\left( u\right) \right) $ (and $g\left( w\right) \leq f\left(
\bar{x}+w\right) $ for all $w$ ) we have the regular subdifferential of $g$ (at $u + v(u)$)
contained in that of $f$ (at $\bar{x} + u + v(u)$). As $g$ is regular the singular subdifferential
coincides with the recession directions of the regular subdifferential \cite[Corollary 8.11]{rock:6} and so are contained in the recession direction of
the regular subdifferential of $f$. We are thus able to write down the
following inclusion
\begin{equation*}
\partial ^{\infty }g \left( u+v\left( u\right) \right)
\subseteq \partial ^{\infty }f\left( \bar{x}+u+v\left( u\right) \right)
=\left\{ 0\right\} .
\end{equation*}
By the tilt stability we have $v$ a locally Lipschitz single valued mapping.
Thus by the basic chain rule of subdifferential calculus we have
\begin{equation*}
\left\{ \nabla _{u}g\left( u+v\left( u\right) \right) \right\} = \left( e_{\mathcal{U}}\oplus \partial v\left( u\right) \right) ^{T}\partial g\left(
u\oplus v\left( u\right) \right)
\end{equation*}
is a single valued Lipschitz mapping. Under the additional assumption we
have via Proposition \ref{prop:reg} part 4 that, $\operatorname{cone}\left[ \partial
_{\mathcal{V}}g\left( u+v\left( u\right) \right) \right] \supseteq \mathcal{V}$ for 
$u\in B_{\varepsilon }\left( 0\right) \cap \mathcal{U}$. As $\partial
v\left( u\right) \subseteq \mathcal{V}$ it cannot be multi-valued and still
have $\left( e_{\mathcal{U}}\oplus \partial v\left( u\right) \right)
^{T}\partial g\left( u\oplus v\left( u\right) \right) $ single valued. This
implies the limiting subdifferential $\partial v\left( u\right) $ is locally single
valued and hence $\nabla v\left( u\right) $ exists locally. The upper-semi-continuity of the subdifferential and the single-valuedness implies $u \mapsto \nabla v(u)$ is a continuous mapping. 
\end{proof}

The following example demonstrates the fact that even if $\partial
_{w}g\left( u+v\left( u\right) \right) $ is multi-valued we still have $\left( e_{\mathcal{U}}, \nabla v\left( u\right) \right) ^{T}\partial
_{w}g\left( u+v\left( u\right) \right) $ single valued.

\begin{example}
If $f:\mathbb{R}^{2}\rightarrow \mathbb{R}$ is given by $f=\max \{ f_1,
f_2\} $ where $f_1=w_1^2+(w_2-1)^2$ and $f_2=w_2$, then $\partial
_{w}g\left( u+v\left( u\right) \right) $ is multi valued but $\left( e_{\mathcal{U}},\partial v\left( u\right) \right) ^{T}\partial _{w}g\left(
u+v\left( u\right) \right) $ is single valued. 
\end{example}

Using the notation in the Theorem, we put $\bar{w}=0,$ find that $\partial
f\left( 0\right) =\{\alpha \left( 0,1-\sqrt{5}\right) +(1-\alpha )\left(
0,1\right) \;|\;0\leq \alpha \leq 1\}$ so we have $\mathcal{U}=\left\{
\alpha \left( 1,0\right) \mid \alpha \in \mathbb{R}\right\} $ and $\mathcal{V}=\left\{ \alpha \left( 0,1\right) \mid \alpha \in \mathbb{R}\right\} $.
With $\epsilon <1/2$ then
\begin{equation*}
v\left( u\right) =\frac{3}{2}-\frac{\sqrt{9-4u^{2}}}{2},\quad g\left(
u+v\left( u\right) \right) =f\left( \bar{x}+u+v\left( u\right) \right) =
\frac{3}{2}-\frac{\sqrt{9-4u^{2}}}{2}.
\end{equation*}
It follows that
\begin{equation*}
\nabla v\left( u\right) =\frac{2u}{\sqrt{9-4u^{2}}}\quad \text{and}\quad
\left( e_{\mathcal{U}},\partial v\left( u\right) \right) ^{T}=\left( 1,\frac{2u}{\sqrt{9-4u^{2}}}\right) ^{T}.
\end{equation*}
Now we consider $\partial _{w}g\left( u+v\left( u\right) \right) =\partial
_{w}f\left( u+v\left( u\right) \right) $. At $u+v\left( u\right) $, from $f_{1}$ we know
\begin{equation*}
t_{1}=(2u,1-\sqrt{9-4u^{2}})=\nabla _{w}f_{1}\left( u+v\left( u\right)
\right)
\end{equation*}
and from $f_{2}$ we know
\begin{equation*}
t_{2}=(0,1)=\nabla _{w}f_{2}\left( u+v\left( u\right) \right) .
\end{equation*}
Thus
\begin{equation*}
\partial _{w}f\left( u+v\left( u\right) \right) =\{\alpha t_{1}+(1-\alpha
)t_{2}\;|\;0\leq \alpha \leq 1\},
\end{equation*}
that is, $\partial _{w}g\left( u+v\left( u\right) \right) $ is multi valued.
However, for all such $\alpha $, we have
\begin{equation*}
\left( e_{\mathcal{U}},\partial v\left( u\right) \right) ^{T}(\alpha
t_{1}+(1-\alpha )t_{2})=2\alpha u+(1-\alpha \sqrt{9-4u^{2}})\frac{2u}{\sqrt{9-4u^{2}}}=\frac{2u}{\sqrt{9-4u^{2}}}.
\end{equation*}
Therefore $\left( e_{\mathcal{U}},\partial v\left( u\right) \right)^{T}\partial _{w}g\left( u+v\left( u\right) \right) $ is single valued.
\bigskip

We may now demonstrate that we have arrived at a weakening of the second order expansions studied in \cite[Theorem 3.9]{Lem:1},
\cite[Equation (7)]{Mifflin:2004:2} and \cite[Theorem 2.6]{Miller:1}. 

\begin{corollary}\label{cor:53}
Under the assumption of  Theorem \ref{thm:2} we have the following local lower Taylor estimate holding: there exists $\delta >0$ such that for all $u \in B_{\delta } (0) \cap \mathcal{U}$   we have for all $u'+v' \in B_{\delta} (u + v(u))$ 
\begin{eqnarray*}
f(\bar{x} + u' + v' ) &\geq & f(\bar{x} + u + v(u)) + \langle z_{\mathcal{U}}(u) + \bar{z}_{\mathcal{V}}
,u' + v' - (u+v(u) \rangle \\
&& \quad  + \frac{1}{2}  (u' - u)^T Q (u' - u) + o(\|u'  -u\|^2 ), 
\end{eqnarray*}
for all $Q \in \partial^{2,-} L_{\mathcal{U}}^{\varepsilon }(u, z_{\mathcal{U}}(u))$, where 
$ z_{\mathcal{U}}(u) := \nabla  L_{\mathcal{U}}^{\varepsilon } (u) $. 
\end{corollary}

\begin{proof}
We apply Corollary \ref{cor:Lagsubjet} taking note of the observation in remark \ref{rem:27}
to obtain the following chain of inequalities. As $Q \in \partial^{2,-} L_{\mathcal{U}}^{\varepsilon }(u, z_{\mathcal{U}}(u))$ we have 
\begin{eqnarray*}
f(\bar{x} + u' + v') - \langle  \bar{z}_{\mathcal{V}} , v' \rangle & \geq & f(\bar{x} + u' + v(u')) 
	- \langle  \bar{z}_{\mathcal{V} }, v(u') \rangle = L_{\mathcal{U}}^{\varepsilon }(u')\\
		&\geq &  L_{\mathcal{U}}^{\varepsilon }(u) + \langle \nabla  L_{\mathcal{U}}^{\varepsilon } (u) ,  u'-u \rangle + \frac{1}{2} (u'- u)Q(u'-u ) + o(\|u' - u \|^2 ) \\
		&=& f(\bar{x} + u + v(u))  - \langle  \bar{z}_{\mathcal{V} }, v(u) \rangle
		+ \langle z_{\mathcal{U}}(u) ,  u'-u \rangle + \frac{1}{2} (u'- u)Q(u'-u ) \\
		&& \qquad\qquad\qquad\qquad\qquad\qquad\qquad\qquad\qquad  + o(\|u' - u \|^2 ) ,
\end{eqnarray*} 
where we have used Corollary \ref{cor:1} to deduce that  $\nabla  L_{\mathcal{U}}^{\varepsilon } (u)  =  z_{\mathcal{U}} (u)$ exists 
locally as a Lipschitz continuous function. The result now follows using the orthogonality of the $\mathcal{U}$ and $\mathcal{V}$. 
\end{proof}

\begin{remark}
The function described in Theorem \ref{thm:2} are quite closely related to
the partial smooth class introduced by Lewis \cite{Lewis:2,Lewis:1}. Lewis
calls $f$ partially smooth at $x$ relative to a manifold $\mathcal{M}$ iff

\begin{enumerate}
\item \label{part1} We have $f|_{\mathcal{M}}$ is smooth around $x$;

\item \label{part2} for all points in $\mathcal{M}$ close to $x$ we have $f$
is regular and has a subgradient;

\item \label{part3} we have ${f_{\_}}^{\prime}(x,h) >- {f_{\_}}^{\prime}(x,-h)$ for all $h \in N_{\mathcal{M}} (x)$ and

\item \label{part4} the subgradient mapping $w \mapsto \partial f (w)$ is
continuous at $x$ relative to $\mathcal{M}$.
\end{enumerate}

It is not difficult to see that $\{0\}\times \mathcal{V}=N_{\mathcal{M}}(x)$. Clearly we have \ref{part1} and \ref{part3} holding for the function
described in Theorem \ref{thm:2}. As functions that are prox-regular at a
point $(x,0)\in \operatorname{Graph}\partial f$ are not necessarily regular at $x$
then \ref{part2} is not immediately obvious, although a subgradient must
exist. By Proposition \ref{prop:reg} the restricted function (to $\mathcal{U}
$) is indeed regular. Moreover the "convex representative" given by $g:=
\operatorname{co}h$ is regular, thanks to convexity. The potential for $w\mapsto
\partial g(w)$ to be continuous at $0$ (relative to $\mathcal{M}$) is clearly bound to the need for $w_{\mathcal{V}}\mapsto \partial _{\mathcal{V}}g(w_{\mathcal{V}})$ to be
continuous at $0$. As $0\in \operatorname{int}\partial _{\mathcal{V}}g\left(
u+v\left( u\right) \right) $ for $u\in B_{\varepsilon }\left( 0\right) \cap
\mathcal{U}$ this problem may be reduced to investigating whether $u\mapsto
\operatorname{int}\partial _{\mathcal{V}}g\left( u+v\left( u\right) \right) $ is
lower semi-continuous at $0$. This is not self evident either. So the
question as to whether $g$ is partially smooth is still open. The solution
to this issue may lie in the underlying assumption that  $\mathcal{U}=
\mathcal{U}^{2}$ in Theorem \ref{thm:2} (see the discussion in Remark \ref{rem:fasttrack}).
On balance the authors would conjecture that the functions we described in Theorem \ref{thm:2} 
are most likelihood partially smooth, despite failing to engineer a proof.  
\end{remark}

We would like to finish this section with some remarks regarding the related work in \cite{Lewis:2}. Because of the gap we still currently have in providing a bridge to the concept of partial smoothness we can't make direct comparisons with the results of \cite{Lewis:2}. Moreover in \cite{Lewis:2} the authors deal with $C^2$-smooth manifolds while the natural notion of  smoothness for this work is of type $C^{1,1}$. It would be interesting to see to what degree the very strong results of \cite{Lewis:2} carry over to this context. That is, a study of tilt stability of partially smooth functions under pinned by a $C^1$ or at least $C^{1,1}$-smooth manifold. This may be another avenue to close the gap that still exists.

\section{Appendix A}\label{Appendix:A}

The prove Proposition \ref{limpara} we need the following results regarding
the variation limits of rank-1 supports.

\begin{proposition}[\protect\cite{eberhard:6}, Corollary 3.3]
\label{ebcor:var}Let $\{\mathcal{A}(v)\}_{v\in W}$ be a family of non-empty
rank-1 representers (i.e. $\mathcal{A}(v)\subseteq \mathcal{S}\left(
n\right) $ and $-\mathcal{P}\left( n\right) \subseteq 0^{+}\mathcal{A}(v)$
for all $v$) and $W$ a neighbourhood of $w$. Suppose that $\limsup_{v\rightarrow w}\mathcal{A}(v)=\mathcal{A}(w)$. Then
\begin{equation}
\limsup_{v\rightarrow w}\inf_{u\rightarrow h}q\left( \mathcal{A}(v)\right)
(u)=q\left( \mathcal{A}(w)\right) (h)  \label{ebneqn:3.30}
\end{equation}
\end{proposition}

Recall that $(x^{\prime },z^{\prime })\rightarrow _{S_{p}(f)}(\bar{x},z)$
means $x^{\prime }\rightarrow ^{f}\bar{x}$, $\ z^{\prime }\in \partial
_{p}f(x^{\prime })$ and $z^{\prime }\rightarrow z$.

\begin{corollary}
Let $f:\mathbb{R}^{n}\rightarrow \mathbb{R}_{\infty }$ be proper and lower
semicontinuous with $h\in b^{1}(\underline{\partial }^{2}f(\bar{x},\bar{z}
)). $ Then
\begin{equation}
q\left( \underline{\partial }^{2}f(\bar{x},\bar{z})\right) \left( h\right)
=\limsup_{(x^{\prime },z^{\prime })\rightarrow _{S_{p}(f)}(\bar{x},\bar{z}
)}\inf_{u\rightarrow h}q\left( \partial ^{2,-}f(x^{\prime },z^{\prime
})\right) (u). \label{neqn:47}
\end{equation}
\end{corollary}

\begin{proof}
Use Proposition \ref{ebcor:var} and Remark \ref{rem:limhess}.
\end{proof}

Denote the infimal convolution of $f$ by $f_{\lambda }(x):=\inf_{u\in
\mathbb{R}^{n}}\left( f(u)+\frac{1}{2\lambda }\Vert x-u\Vert ^{2}\right) $.
Recall that $f_{\lambda }\left( x\right) -\frac{1}{2\lambda }\left\Vert
x\right\Vert ^{2}=-\left( f\ +\frac{\lambda }{2}\Vert \cdot \Vert
^{2}\right) ^{\ast }(\lambda x)$ and this $f_{\lambda }$ is always
para-concave. Recall that in \cite[Lemma 2.1]{eberhard:6}, it is observed
that $f$ is locally $C^{1,1}$ iff $f$ is simultaneously a locally
para-convex and para-concave function. Recall \cite[Proposition 4.15]{rock:6}
that states that the limit infimum of a collection of convex sets is also
convex and that the upper epi-limit of a family of functions has an
epi-graph that is the limit infimum of the family of epi-graphs.
Consequently the epi-limit supremum of a family of convex functions give
rise to convex function.

\begin{proof}
(of Proposition \ref{limpara}) Begin by assuming $f$ is locally para-convex.
Let $\frac{c}{2}>0$ be the modulus of para--convexity of $f$ on $B_{\delta }(
\bar{x})$, $x\in B_{\delta }(\bar{x})$ with $z\in \partial f\left( x\right) $
and $\partial ^{2,-}f\left( x,z\right) \neq \emptyset $. Let $C_{t}(x)=\{h\mid x+th\in B_{\delta }(\bar{x})\}$ then we have
\begin{equation*}
h\mapsto \left( \frac{2}{t^{2}}\right) \left( f(x+th)-f(x)-t\langle
z,h\rangle \right) +\frac{c}{t^{2}}\left( \Vert x+th\Vert ^{2}-\Vert x\Vert
^{2}-t\langle 2x,h\rangle \right)
\end{equation*}
convex on $C_{t}(x)$ since $x\mapsto f(x)+\frac{c}{2}\Vert x\Vert ^{2}$ is
convex on $B_{\delta }(\bar{x})$. Next note that for every $K>0$ there
exists a $\bar{t}>0$ such that for $0<t<\bar{t}$ we have $B_{K}(0)\subseteq
C_{t}(x)$. Once again restricting $f$ to $B_{\delta }(\bar{x})$ we get a
family
\begin{equation}
\{h\mapsto \Delta _{2}f(x,t,z,h)+\frac{c}{t^{2}}\left( \Vert x+th\Vert
^{2}-\Vert x\Vert ^{2}-t\langle 2x,h\rangle \right) \}_{t<\bar{t}}
\label{neqn:3}
\end{equation}
of convex functions with domains containing $C_{t}(x)\,$(for each $t$) and
whose convexity (on their common domain of convexity) will be preserved
under an upper epi--limit as $t\downarrow 0$. Thus, using the fact that $\frac{c}{t^{2}}\left( \Vert x+th\Vert ^{2}-\Vert x\Vert ^{2}-t\langle
2x,h\rangle \right) $ converges uniformly on bounded sets to $c\Vert h\Vert
^{2}$, we have the second order circ derivative (introduced in \cite{ebioffe:4}) given by:
\begin{align*}
f^{\uparrow \uparrow }(x,z,h)+c\Vert h\Vert ^{2}& :=\limsup_{(x^{\prime
},z^{\prime })\rightarrow _{S_{p}}(x,z),t\downarrow 0}\inf_{u^{\prime
}\rightarrow h}(\Delta _{2}f(x^{\prime },t,z^{\prime },u^{\prime }) \\
& \qquad \qquad +\frac{c}{t^{2}}\left( \Vert x+th\Vert ^{2}-\Vert x\Vert
^{2}-t\langle 2x,h\rangle \right) )
\end{align*}
which is convex on $B_{K}(0)$, for every $K>0$, being obtained by taking an
epi-limit supremum of a family of convex functions given in (\ref{neqn:3}).
We then have $h\mapsto f^{\uparrow \uparrow }(x,z,h)+c\Vert h\Vert ^{2}$
convex (with $f^{\uparrow \uparrow }(x,z,\cdot )$ having a modulus of
para-convexity of $c$).

From \cite{com:2}, Proposition 4.1 particularized to $C^{1,1}$ functions $f$
we have that there exists a $\eta \in \lbrack x,y]$ such that
\begin{equation}
f(y)\in f(x)+\langle \nabla f(x),y-x\rangle +\frac{1}{2}\langle \overline{D}^{2}f(\eta ),(y-x)(y-x)^{T}\rangle .  \label{ebneqn:31}
\end{equation}
Using (\ref{ebneqn:31}), Proposition \ref{prop:ebpenot} and the variational
result corollary \ref{ebcor:var}, we have when the limit is finite (for $\bar{z}:=\nabla f(\bar{x})$)
\begin{align*}
f^{\uparrow \uparrow }(\bar{x},\bar{z},h)& :=\limsup_{(x^{\prime },z^{\prime
})\rightarrow _{S_{p}}(\bar{x},\bar{z}),\;t\downarrow 0}\inf_{u^{\prime
}\rightarrow h}\Delta _{2}f(x^{\prime },t,z^{\prime },u^{\prime }) \\
& \leq \limsup_{x^{\prime }\rightarrow \bar{x},\;t\downarrow
0}\inf_{u^{\prime }\rightarrow h}\Delta _{2}f(x^{\prime },t,\nabla
f(x^{\prime }),u^{\prime })\leq \limsup_{\eta \rightarrow \bar{x}}\inf_{u^{\prime }\rightarrow h}q\left( \overline{D}^{2}f(\eta )\right)
(u^{\prime }) \\
& \leq q\left( \overline{D}^{2}f(\bar{x})-\mathcal{P}(n)\right) (h)\leq
q\left( \underline{\partial }^{2}f(\bar{x},\bar{z})\right) (h)\leq
f^{\uparrow \uparrow }(\bar{x},\bar{z},h),
\end{align*}
where the last inequality follows from \cite[Proposition 6.5]{ebioffe:4}.

Now assuming $f$ is quadratically minorised and is prox--regular at $\bar{x}\ $ for $\bar{p}\in \partial f(\bar{x})$ with respect to $\varepsilon $ and $r.$ Let $g(x):=f(x+\bar{x})-\langle \bar{z},x+\bar{x}\rangle $. Then $0\in
\partial g(0)$ and we now consider the infimal convolution $g_{\lambda }(x)$
which is para--convex locally with a modulus $c:=\frac{\lambda r}{2(\lambda
-r)}$, prox--regular at $0$ (see \cite[Theorem 5.2]{polrock:1}). We may now
use the first part of the proof to deduce that $g_{\lambda }^{\uparrow
\uparrow }(0,0,\cdot )$ is para--convex with modulus $c=\frac{2\lambda r}{(\lambda -r)}$ and $g_{\lambda }^{\uparrow \uparrow }(0,0,h)=q\left(
\underline{\partial }^{2}g_{\lambda }(0,0)\right) (h)$ since $g_{\lambda }$
is $C^{1,1}$ (being both para-convex and para-concave). Using Corollary \ref{ebcor:var} and \cite[Proposition 4.8 part 2.]{eberhard:2} we obtain
\begin{equation*}
\limsup_{\lambda \rightarrow \infty }\inf_{h^{\prime }\rightarrow
h}g_{\lambda }^{\uparrow \uparrow }(0,0,h^{\prime })=q\left(
\limsup_{\lambda \rightarrow \infty }\underline{\partial }^{2}g_{\lambda
}(0,0)\right) (h)=q\left( \underline{\partial }^{2}g(0,0)\right) (h).
\end{equation*}
Thus $q\left( \underline{\partial }^{2}g(0,0)\right) (h)+r\Vert h\Vert
^{2}=\limsup_{\lambda \rightarrow \infty }\inf_{h^{\prime }\rightarrow
h}\left( g_{\lambda }^{\uparrow \uparrow }(0,0,h^{\prime })+\frac{\lambda r}{(\lambda -r)}\Vert h\Vert ^{2}\right) $ is convex, being the variational
upper limit of convex functions. One can easily verify that $\underline{\partial }^{2}g(0,0)=\underline{\partial }^{2}f(\bar{x},\bar{z})$ and $g^{\uparrow \uparrow }(0,0,h)=f^{\uparrow \uparrow }(\bar{x},\bar{z},h)$.
\end{proof}

\subsection*{REFERENCES}

{\footnotesize \ \makeatletter
\let\ORIGINALlatex@openbib@code=\@openbib@code
\renewcommand{\@openbib@code}{\ORIGINALlatex@openbib@code%
\adjustmybblparameters} \makeatother
}

{\footnotesize \ \renewcommand{\section}[2]{}
\bibliographystyle{plain}
\bibliography{references}
}

\bigskip

\end{document}